\documentclass[a4,reqno,12pt]{amsart}
\usepackage[pdftex]{graphicx}
\usepackage{color}
\usepackage{amsmath}
\usepackage{amsfonts}
\usepackage{amssymb}
{\large}
\numberwithin{equation}{section}
\newtheorem{theorem}{Theorem}

\newtheorem{lemma}[theorem]{Lemma}

\newtheorem{proposition}[theorem]{Proposition}
\newtheorem{remark}[theorem]{Remark}

\newtheorem{bigthm}{Theorem}   

\providecommand{\U}[1]{\protect\rule{.1in}{.1in}}

\providecommand{\lb}{\left(}
\providecommand{\rb}{\right)}
\providecommand{\lbr}{\left\{}
\providecommand{\rbr}{\right\}}

\providecommand{\dd}{{\rm d}}
\providecommand{\be}[1]{\begin{equation}\label{#1}}
\providecommand{\ee}{\end{equation}}

\providecommand{\1}{\ifmmode {1\hskip -3pt \rm{I}}}

\providecommand{\df}{\stackrel{\Delta}{=}}

\providecommand{\eqvs}{\stackrel{\sim}{=}}
\providecommand{\leqs}{\lesssim}            
\providecommand{\geqs}{\gtrsim}             

\providecommand{\smo}[1]{{\mathrm o}_{#1}}

\providecommand{\bigof}[1]{{\mathrm O}\lb #1\rb }

\providecommand{\taub}{\tau_\beta}

\providecommand{\sumtwo}[2]{\sum_{\substack{#1 \\ #2}}} 
\providecommand{\abs}[1]{\left| #1\right|}

\providecommand{\calC}{\mathcal{C}}
\providecommand{\calD}{\mathcal{D}}
\providecommand{\calE}{\mathcal{E}}

\providecommand{\calG}{\mathcal{G}}
\providecommand{\calH}{\mathcal{H}}

\providecommand{\calK}{\mathcal{K}}
\providecommand{\calL}{\mathcal{L}}

\providecommand{\calR}{\mathcal{R}}
\providecommand{\calS}{\mathcal{S}}
\providecommand{\calT}{\mathcal{T}}

\providecommand{\calV}{\mathcal{V}}

\providecommand{\calY}{\mathcal{Y}}

\providecommand{\fra}{\mathfrak{a}}
\providecommand{\frb}{\mathfrak{b}}
\providecommand{\frc}{\mathfrak{c}}

\providecommand{\frg}{\mathfrak{g}}

\providecommand{\frl}{\mathfrak{l}}

\providecommand{\frn}{\mathfrak{n}}

\providecommand{\frB}{\mathfrak{B}}

\providecommand{\frS}{\mathfrak{S}}

\providecommand{\bbB}{\mathbb{B}}
\providecommand{\bbC}{\mathbb{C}}
\providecommand{\bbD}{\mathbb{D}}
\providecommand{\bbE}{\mathbb{E}}

\providecommand{\bbL}{\mathbb{L}}

\providecommand{\bbN}{\mathbb{N}}

\providecommand{\bbP}{\mathbb{P}}
\providecommand{\bbQ}{\mathbb{Q}}
\providecommand{\bbR}{\mathbb{R}}

\providecommand{\bbZ}{\mathbb{Z}}

\providecommand{\sfa}{\mathsf{a}}

\providecommand{\sfd}{\mathsf{d}}
\providecommand{\sfe}{\mathsf{e}}
\providecommand{\sff}{\mathsf{f}}

\providecommand{\sfu}{\mathsf{u}}
\providecommand{\sfv}{\mathsf{v}}
\providecommand{\sfw}{\mathsf{w}}
\providecommand{\sfx}{\mathsf{x}}
\providecommand{\sfy}{\mathsf{y}}
\providecommand{\sfz}{\mathsf{z}}

\providecommand{\sfA}{\mathsf{A}}

\providecommand{\sfD}{\mathsf{D}}

\providecommand{\sfH}{\mathsf{H}}

\providecommand{\sfM}{\mathsf{M}}

\providecommand{\sfP}{\mathsf{P}}

\providecommand{\sfS}{\mathsf{S}}

\providecommand{\sfU}{\mathsf{U}}
\providecommand{\sfV}{\mathsf{V}}
\providecommand{\sfW}{\mathsf{W}}
\providecommand{\sfX}{\mathsf{X}}
\providecommand{\sfY}{\mathsf{Y}}
\providecommand{\sfZ}{\mathsf{Z}}

\providecommand{\step}[1]{S{\small TEP}\,#1}

\providecommand{\wbf}{w_\beta^{\sff}} 

\begin{document}

\title[{Formation of Facets}]
{\textbf{Formation of Facets for an Effective Model of Crystal Growth }}

\begin{abstract}

We study  an effective model of microscopic facet formation
for low temperature three dimensional microscopic Wulff crystals
above the droplet condensation threshold. The model we consider is
a $2+1$ solid on solid surface coupled with high and low
density bulk Bernoulli fields.
At equilibrium the surface stays flat. Imposing a canonical constraint
on excess number of particles forces the surface to ``grow'' through the sequence of
spontaneous creations of macroscopic size monolayers. We prove that
at all sufficiently low temperatures, as
the excess particle constraint is tuned, the
model undergoes an infinite sequence of first order transitions,
which traces an infinite sequence of first order transitions in the
underlying variational problem. Away from transition
values of canonical constraint
we  prove sharp concentration results for the rescaled level lines
around solutions of the limiting variational problem.

\end{abstract}

\author{Dmitry Ioffe}
\address{Faculty of IE\&M, Technion, Haifa 32000, Israel}
\email{ieioffe@ie.technion.ac.il}
\thanks{The research of D.I. was supported by
Israeli Science Foundation grant 1723/14, by The Leverhulme Trust through International
Network Grant {\em Laplacians, Random Walks, Bose Gas, Quantum Spin Systems}  and
by the Meitner Humboldt Award.}
\author{Senya Shlosman}
\address{Skolkovo Institute of Science and Technology, Moscow, Russia;
Aix Marseille Universit\'e, Universit\'e de Toulon,
CNRS, CPT UMR 7332, 13288, Marseille, France;
Inst. of the Information Transmission   Problems,
RAS, Moscow, Russia}
\email{senya.shlosman@univ-amu.fr}
\thanks{The research of S.S.  has been carried out in the framework of the 
Labex Archimede (ANR-11-LABX-0033) and of the A*MIDEX project (ANR-11-IDEX-0001-02), 
funded by the "Investissements d'Avenir" French Government programme managed by the French National Research Agency (ANR). 
Part of this work has been carried out at IITP RAS. The support of 
Russian Foundation for Sciences (project No. 14-50-00150) is gratefully acknowledged.}

\date{\today}

\maketitle

\section{Introduction}
Low temperature three dimensional equilibrium crystal shapes exhibit flat
facets, see e.g. \cite{BFL82}, \cite{BMF86}, \cite{MS99}.
It is known that lattice oriented low temperature microscopic
interfaces stay flat under Dobrushin boundary conditions \cite{Dobrushin}.
But the  $\bbL_1$-theory, which is the base of the
microscopic justification of Wulff construction \cite{B99,CP00,BIV00} in three and higher
dimensions,
does not directly address
fluctuations of microscopic shapes on scales smaller than the linear size of
the system. In particular, the existence and the microscopic structure of facets
remains an open question even at very low temperatures.

{
In a sense, this
issue is complementary to a large body of works, see for instance
\cite{CKP01, CerfK01, FS03, Ken08} and references to  later papers,
and also a recent review \cite{Ok16},
which focus on study of the corners of
zero or low temperature microscopic crystals.
}

For low temperature $2+1$ SOS (solid on solid) interfaces under canonical
constraints on the volume below the microscopic surface, existence of flat
microscopic facets was established in \cite{BSS05}. Here we consider
facet formation for a SOS model coupled to  high and low density
bulk Bernoulli fields which are supposed to mimic  coexisting phases of the
three dimensional model.

The phenomenon of droplet condensation in the framework of the Ising model was
first described in the papers \cite{DS1}, \cite{DS2}. There it was considered
the Ising model at low temperature $\beta^{-1}$, occupying a $d$-dimensional
box $T_{N}^{d}$ of the linear size $2N$ with periodic boundary conditions. The
ensemble was the canonical one:   the total magnetization,
\[
M_{N}\,=\,\,\sum\sigma_{t},
\]
was fixed. In case $M_{N}=m^{\ast}\left(  \beta\right)  \left\vert T_{N}%
^{d}\right\vert ,$ where $m^{\ast}\left(  \beta\right)  >0$ is the spontaneous
magnetization, the typical configuration looks as a configuration of the
$\left(  +\right)  $-phase:  the spins are taking mainly the values $+1,$
while the values $-1$ are seen rarely, and the droplets of minuses in the box
$T_{N}^{d}$ are at most of the size of $K\left(  d\right)  \ln N.$ In order to
observe the condensation of small $\left(  -\right)  $-droplets into a big one
it is necessary to increase their amount, which can be achieved by
considering  a different canonical constraint:
\[
M_{N}=m^{\ast}\left(  \beta\right)  \left\vert T_{N}^{d}\right\vert -b_{N},
\]
$b_{N}>0.$ It turns out that if ${b_{N}}/{\left\vert T_{N}^{d}\right\vert
^{\frac{d}{d+1}}}\rightarrow0$ as $N\rightarrow\infty,$ then in the
corresponding canonical ensemble all the droplets are still microscopic, not
exceeding $K\left(  d\right)  \ln N$ in linear size. On the other hand, once
$\liminf_{N\rightarrow\infty}
{b_{N}}/{\left\vert T_{N}^{d}\right\vert
^{\frac{d}{d+1}}}>0,$ the situation becomes different: among many $\left(
-\right)  $-droplets there is one, $\mathcal{D},$ of the linear size $\ $of
the order of $\left(  b_{N}\right)  ^{1/d}\geq N^{\frac{d}{d+1}},$ while all
the rest of the droplets are still at most logarithmic. Therefore $b_{N}%
\sim\left\vert T_{N}^{d}\right\vert ^{\frac{d}{d+1}}$ can be called the
\textit{condensation threshold, or dew-point}.

When $b_{N}$ grows beyond the condensation threshold, the big droplet
$\mathcal{D}$ grows as well. To study this growth process, or specifically
to try to get an insight of the process of formation of new atomic-scale
layers on microscopic facets,
we have suggested in
our paper "Ising model fog drip: the first two droplets" \cite{IS08} a simplified
growth model, where one puts the observer on the surface $\mathcal{S}$ of
$\mathcal{D}$ and studies the evolution of this surface $\mathcal{S}$ as the
volume of $\mathcal{D}$ grows.

It was argued in   \cite{IS08} that the evolution of $\mathcal{S}$ proceeds through
the spontaneous creations of extra monolayers. Each new monolayer has
one-particle thickness, while the breadth of the $k$-th monolayer is
$\sim c_{k}N,$ with $c_{k}\geq c_{crit}=c_{crit}\left(  \beta\right)  >0.$ It
then grows in breadth for some time, until a new monolayer is spontaneously
created at its top, of the size of $c_{k+1}N.$

In \cite{IS08} we were able to analyze this process only for the first two monolayers.
Our technique at this time was not sufficient, and we were unable to control
the effect of the interaction between the two monolayers when their boundaries
come into contact and start to influence each other. This technique was later
developed in our paper \cite{IST15}, so we are able now to conclude our studies. The
present paper thus contains the material of what we have promised in \cite{IS08} to
publish under the provisional title "Ising model fog drip, II: the puddle".

In the present work we can handle any finite number of monolayers. What we
find quite interesting is that the qualitative picture of the process of
growth of monolayers changes, as $k$ increases. Namely, for
the few initial values of $k=1,2,...,k_{c}$ the size $c_{k}N$ of the $k$-th
monolayer at the moment of its creation is strictly smaller than the size  of
the underlying $\left(  k-1\right)  $-th layer. Thus, the picture is qualitatively
the same as that of the lead crystal -
{
there is an extensive physical literature on the
latter subject, for instance
see Figure~2 in \cite{B03} and  discussions in \cite{EBWTR-RW,BYS}
}
However, for $k>k_{c}$ this is not the case any more, and the size $c_{k}N$ is
exactly the same as the size of the underlying $\left(  k-1\right)  $-th layer
at the creation moment.

Still, creation of each new layer $k$ bears a signature of
first order transition - at the moment of creation all the underlying layers shrink.
This transition resembles spontaneous creation of mesoscopic size droplets
in two-dimensional Ising model \cite{BCK03}, and as in the latter work it is
related to first order transitions in the underlying variational problem.

Our picture has to be compared with a similar one, describing the
layer formation in the SOS model above the wall, studied in
a series of works \cite{CLMST14,CMT14, CLMST13}.
Unlike our model, all the layers of the SOS surface above the wall have different
asymptotic shapes. The reason is that the repulsion from the wall results in
different
area tilts for different layers there, and, accordingly, gives rise to  different
solutions of the corresponding variational problem.
{
Another important difference is that in the SOS model  \cite{CLMST14} one never sees the top
monolayer  detached from the rest of them, as in the model we consider.
}
Nevertheless, we believe that in our model with   $k$ monolayers the fluctuations of
their boundaries  in the vicinity of the (vertical) wall are, as in the case
of entropic repulsion \cite{CLMST13}, of
the order of $N^{1/3}$, and their  behavior is given, after
appropriate scaling, by $k$ non-intersecting Ferrari-Spohn diffusions
\cite{FS05},
as in \cite{ISV15, IVW17}. See \cite{IV16} for a review.


\section{The Model and the Results}

\subsection{The Model}

We will study the following simple model of
facets formation on interfaces between two coexisting phases which was introduced in
\cite{IS08}:
The system is
confined to the 3D box
\[
\Lambda_{N}=B_{N}\times\left\{  -\frac{N+1}{2},-\frac{N-1}{2},...,\frac
{N-1}{2},N+\frac{N+1}{2}\right\}  ,
\]
where  $N\in2\bbN$ is even, $B_{N}$ is a two-dimensional $N\times N$ box;
\[
B_{N}\,=\,\left\{  -N,\dots,N\right\}  ^{2} = N\bbB_1\cap \bbZ^2,
\]
 and $\bbB_1 = [-1,1]^2$.
The interface $\Gamma$  between two phases is supposed to be an SOS-type surface; it is
uniquely defined by a function
\begin{equation}
 h_{\Gamma}:\mathring{B}_{N}\rightarrow\left\{  -\frac{N}{2},-\frac{N}{2}+1,...,\frac{N}{2}\right\}  ,
\label{eq:hFunction}%
\end{equation}
where $\mathring{B}_{N}$ is the interior of $B_{N}$. We assume that
the interface $\Gamma$ is pinned at zero height on the boundary $\partial
B_{N}$, that is $h_{\Gamma}\equiv0$ on $\partial B_{N}$. Such a
surface $\Gamma$ splits $\Lambda_{N}$ into two parts; let us denote by $V_{N}\left(
\Gamma\right)  $ and $S_{N}\left(  \Gamma\right)  $ the upper and the lower
halves.
We suppose that $\Gamma$ separates the low density phase (vapor) in the
upper half of the box from the high density phase (solid) in the lower half.
This is modeled in the following fashion: First of all, the marginal
distribution of $\Gamma$ obeys the SOS statistics at an inverse temperature $\beta$.
That is we associate with
$\Gamma$ a weight,
\be{eq:GammaWeight-SOS}
w_{\beta}(\Gamma) = {\rm exp}\lbr -\beta \sum_{x\sim y}
\abs{ h_\Gamma (x ) - h_\Gamma (y )}\rbr,
\ee
where we extended  $h_\Gamma \equiv 0$ outside ${B}_{N}$,
and the sum is over all unordered pairs of nearest neighbors of $\bbZ^d$.

Next, $\Gamma$ is coupled to high and low density Bernoulli bulk fields:
Let $0<p_{v} = p_v (\beta ) <p_{s} = p_s (\beta ) <1$.
A relevant choice of $p_v, p_s$ with a simplification of
the three dimensional Ising model in mind
would be  $p_{v} (\beta ) = {\rm e}^{-6\beta} = 1- p_s (\beta )$.
In the sequel we shall assume\footnote{Actually, main results
hold even with faster decay than \eqref{eq:p-beta}.}
\be{eq:p-beta}
\liminf_{\beta\to\infty} \frac{1}{\beta}
\log\lb \min \lbr p_v (\beta ) , 1-p_s (\beta )\rbr\rb > -\infty .
\ee
{At each site }$i\in V_{N}$ { we
place a particle with probability }$p_{v}$,{ while at each site }$i\in
S_{N}$ { we place a particle with probability }$p_{s}.$ Alternatively,
let $\left\{  \xi_{i}^{v}\right\}  $ and $\left\{  \xi_{i}^{s}\right\}  $ be
two independent collection of Bernoulli random variables with parameters
$p_{v}$ and $p_{s}$.
Then the empirical
field of particles \emph{given} interface $\Gamma$ is
\[
\sum_{i\in V_{N}(\Gamma)}\xi_{i}^{v}\delta_{i}+\sum_{j\in S_{N}(\Gamma)}%
\xi_{j}^{s}\delta_{j}.
\]
All together, the joint distribution of the triple $\left(  \Gamma,\xi^{v}%
,\xi^{s}\right)  $ is given by
\begin{equation}
\mathbb{{P}}_{N, \beta} \left(  \Gamma,\xi^{v},\xi^{s}\right)  \propto w_{\beta
}(\Gamma)\prod_{i\in V_{N}}p_{v}^{\xi_{i}^{v}}\left(  1-p_{v}\right)
^{1-\xi_{i}^{v}}\prod_{j\in S_{N}}p_{s}^{\xi_{j}^{s}}\left(  1-p_{s}\right)
^{1-\xi_{j}^{s}}. \label{eq:tripple}%
\end{equation}
We denote the total number of particles in vapor and solid phases, and total number
of particles in the system as
\be{eq:Sigma}
\Xi_{v}=\sum_{i\in V_{N}(\Gamma)}\xi_{i}^{v}\, ,\  \Xi_{s}=\sum_{j\in
S_{N}(\Gamma)}\xi_{j}^{s} \quad {\rm and}\quad \Xi_N = \Xi_{v} + \Xi_{s}
\ee
respectively. The conditional distributions of
$\Xi_{v}$ and $\Xi_{s}$ given $\Gamma$ are binomial $\mathrm{Bin}\left(
|V_{N}|,p_{v}\right)  $ and $\mathrm{Bin}\left(  |S_{N}|,p_{s}\right)  $.

By the definition of the model, the expected total number of particles
\be{eq:p}
  \bbE_{N, \beta}\lb  \Xi_N \rb
  =\frac{p^s +p^v}{2}\abs{\Lambda_N} \equiv  \rho_\beta N^3 .
\ee
Formation of facets is modeled in the following way:  Consider
\be{anot}
\bbP_{N, \beta}^{A} \lb\  \cdot\  \rb =\bbP_{N , \beta}
\lb \ \cdot\  \big|\Xi_N
 {\, \geq \,}
\rho_\beta
N^3 +A N^2  \rb .
\ee
We claim that  the model {exhibits a sequence} of first order transitions as
$A$ in \eqref{anot} varies. The geometric manifestation of these
transitions is {the} spontaneous creation of macroscopic size
monolayers. In \cite{IS08} we have investigated the creation of the first two
monolayers. The task of the current paper is to provide an asymptotic (as $N\to\infty$)
description of typical surfaces $\Gamma$ under ${\mathbb{P}}_{N, \beta}\left(
\, \cdot\,\bigm|\,\Xi_N
{\,\geq \,}
\rho_\beta   N^{3}+ A N^{2}\right)  $ for {\em any} $A$
fixed.

To study this conditional distribution we rely on Bayes rule,
\begin{equation}
\bbP_{N, \beta}^{A} \lb \Gamma \rb
\,=\,\frac{{\mathbb{P}_{N, \beta}}\left(  \Xi_N\,
\geq
\,\rho_\beta N^{3}+ A N^{2}%
\bigm|\Gamma\right)  {\mathbb{P}_{N, \beta} }\left(  \Gamma\right)  }{\sum
_{\Gamma^{\prime}}{\mathbb{P}_{N , \beta} }\left(  \Xi_N \,
\geq
\, \rho_\beta  N^{3}+ A
N^{2}\bigm|\Gamma^{\prime}\right)  {\mathbb{P}_{N, \beta} }\left(  \Gamma^{\prime
}\right)  }. \label{eq:Bayes}%
\end{equation}
The control over the conditional probabilities ${\mathbb{P}}_{N , \beta} \left(
\cdot\bigm|\Gamma\right)  $ comes from volume order local limit theorems for
independent Bernoulli variables, whereas a-priori probabilities
$\mathbb{P}_{N, \beta} %
\left(  \Gamma\right)  $ are derived from representation of $\Gamma$ in terms
of a gas of non-interacting contours. Models with bulk fields  give an alternative
approximation of interfaces in low temperature  3D Ising model,
and they enjoy certain technical advantages  over the usual SOS model with
weights
$w_{\beta}(\Gamma)$ (see
\eqref{Gammaweight}).
In particular volume order limit results
enable a simple control over the phase of intermediate contours.

\subsection{Heuristics and informal statement of the main result}
Let us describe the heuristics behind the claimed sequence of first order transitions:
 To each surface $\Gamma$ corresponds a signed volume $\alpha (\Gamma )$.
 In terms of the height function $h_{\Gamma}$  which was defined in
\eqref{eq:hFunction},
\begin{equation}
\label{eq:alpha}\alpha(\Gamma)\,=\,\int\int h_{\Gamma}(x,y)dxdy,
\end{equation}
 Main contribution
 to $\alpha (\Gamma )$ comes from large microscopic   facets,
 which are encoded by
 large microscopic level lines $\Gamma_1, \dots \Gamma_\ell$;
 \be{eq:Area-Gamma}
 \alpha (\Gamma ) \approx \sum_1^\ell \sfa \lb \Gamma_i \rb  =
 N^2 \sum_1^\ell \sfa \lb \frac{1}{N}\Gamma_i \rb .
 \ee
 The notions of level lines are defined and discussed in Section~\ref{sec:level-lines}.
 Locally large level lines $\Gamma_1, \dots , \Gamma_\ell$ have structure of
 low temperature Ising polymers, and
 they give rise to a {\em two-dimensional} surface tension $\tau_\beta$.
 As we shall explain below the a-priori probability
 of creating surface with a prescribed volume $a N^2$ is asymptotically given
 by
 \be{eq:st-approx}
 \log\bbP_{N, \beta} \lb \alpha (\Gamma ) = a N^2 \rb \approx -N
{\tau_\beta (a )},
\ee
where $\taub (a )$ is the minimal surface tension of  a compatible collection
of simple curves $\gamma_1, \dots , \gamma_\ell \subset \bbB_1$ with total area
$\sum_i \sfa (\gamma_i )=a$. The notion of compatibility is explained in the
beginning of Section~\ref{sec:var}, which is devoted to a careful analysis of
 the minimization problem we consider here. In fact the only relevant
 compatible collections happen to be ordered stacks
\be{eq:stack}
 \mathring{\gamma_\ell}\subseteq \mathring{\gamma_2}\subseteq \dots
 \subseteq\mathring{\gamma_1}\subseteq \bbB_1 .
\ee
Informally, the
asymptotic relation \eqref{eq:st-approx} is achieved when
rescaled $\frac{1}{N}\Gamma_i$ microscopic level lines stay close
to optimal $\gamma_i$-s.
On the other hand, the presence of $\Gamma$-interface
shifts the expected number of
particles in the bulk by the quantity $\Delta_\beta \alpha (\Gamma )$, where
$\Delta_\beta = (p^s - p^v)$. That is,
\be{eq:exp-shifted}
\bbE_{N, \beta} \lb \Xi_N \big| \Gamma \rb =\rho_\beta N^3 +
\Delta_\beta  \alpha (\Gamma ) .
\ee
Therefore, in view of local limit asymptotics for bulk Bernoulli fields,
\be{eq:bulk-ll}
\log\bbP_{N , \beta}\lb \Xi_N \,
\geq
\, \rho_\beta N^3 +A N^2\big| \alpha (\Gamma )
= a N^2\rb
\approx
-\frac{(A N^2- \Delta_\beta  a N^2)^2}{2 N^3 R_\beta} =
-N{\frac{(\delta_\beta  - a)^2}{2D_\beta}},
\ee
where
\be{eq:quantities}
R_\beta :=   2\lb {p^s (1-p^s ) +p^v (1-p^v )}\rb, \quad
D_\beta :=   \frac{R_\beta }{\Delta^2_\beta }\quad  {\rm and} \quad
\delta_\beta  :=   A/\Delta_\beta  .
\ee
Consequently, the following asymptotic relation should hold:
\be{eq:var-asymp}
\frac{1}{N} \log\bbP_{N, \beta}\lb \Xi  \,
\geq
\, \rho_\beta N^3 +A N^2 \rb \approx
- \min_{a} \lbr \frac{(\delta_\beta  - a)^2}{2D_\beta} +\taub (a )\rbr ,
\ee
and, moreover, if ${a}^*$ is the unique
minimizer for the right hand side of \eqref{eq:var-asymp}, then
the conditional distribution
$\log\bbP_{N, \beta}\lb \, \cdot\, \big| \Xi  \geq  \rho_\beta N^3 +A N^2 \rb$
should concentrate on surfaces
$\Gamma$ which have an optimal volume close to
$a^* N^2$ or, in view of
\eqref{eq:Area-Gamma} and \eqref{eq:st-approx}
whose rescaled large level lines $\frac{1}{N}\Gamma_1, \dots ,\frac{1}{N} \Gamma_\ell$
are  macroscopically close to the optimal stack
$\mathring{\gamma_\ell^*} \subseteq \dots  \subseteq  \mathring{\gamma_1^*}$
which produces
$\tau_\beta (a^* )$.
\begin{bigthm}
\label{thm:A}
Assume that bulk occupation probabilities $p_v$ and $p_s$ satisfy
\eqref{eq:p-beta}. Then there exists $\beta_0 <\infty$ such that for any
$\beta >\beta_0$
there exists a sequence of numbers $0 < A_1 (\beta ) <A_2 (\beta )
< A_3 (\beta )  <\dots$ and a sequence of
areas $ 0 < a_1^- (\beta ) < a_1^+ (\beta ) <a_2^- (\beta )
< a_2^+ (\beta ) <\dots $ which satisfies
properties {\bf A1} and {\bf A2} below:\\
{\bf A1.} For any $A\in [0, A_1)$ the  minimizer in the right hand side
\eqref{eq:var-asymp} (with $\delta_\beta  = A/\Delta_\beta$
as in \eqref{eq:quantities})
is $a^*= 0$ which, in terms of contours in \eqref{eq:stack} corresponds to the empty
stack. For any $\ell = 1, 2, \dots$ and for any $A\in (A_\ell , A_{\ell +1})$ the
unique minimizer $a^*$ in the right hand side of \eqref{eq:var-asymp} satisfies
$a^* \in ( a_{\ell}^- , a_\ell^+ )$, and it corresponds to a unique,
{up to compatible shifts within $\bbB_1$,}
stack
of exactly $\ell$-contours $\mathring{\gamma}_\ell^* \subseteq \dots \subseteq
\mathring{\gamma}^*_1 \subseteq \bbB_1$.  \\
{\bf A2.}  For any $A\in [0, A_1)$ there are no large level lines (no microscopic
facets, that is $\Gamma$ stays predominately flat on zero-height level)
with $\bbP_{N, \beta}^A$-probability tending to one. On the other hand,
for any $\ell = 1, 2, \dots$ and for any $A\in (A_\ell , A_{\ell +1})$, there are
exactly $\ell$ microscopic facets of $\Gamma$ with $\bbP_{N, \beta}^A$-probability
tending to
one and, moreover, the rescaled
 large   level lines $\frac{1}{N}\Gamma_1 , \dots , \frac{1}{N}\Gamma_\ell $ concentrate
 in Hausdorff distance $\sfd_{\sfH}$  near the optimal stack
 $\lbr \gamma_1^* , \dots , \gamma_\ell^*\rbr $, as described
 in part {\bf A1} of the
 Theorem, in the
 following sense: For any $\epsilon > 0$ ,
 \be{eq:Hausd-stack}
 {\lim_{N\to\infty}\bbP^A_{N, \beta} \lb \sum_{i=1}^{\ell -1}
 \sfd_{\sfH} \lb \frac{1}{N}\Gamma_i , \gamma^*_i\rb +
 \min_{x : x+\gamma^*_\ell \subset \bbB_1}
 \sfd_{\sfH} \lb \frac{1}{N}\Gamma_\ell , x+ \gamma^*_\ell\rb
 \geq \epsilon \rb = 0 .}
 \ee
\end{bigthm}
\subsection{Structure of the paper}
Section~\ref{sec:var} is devoted to a careful analysis of the multi-layer
minimization problem in the right hand side
\eqref{eq:var-asymp}. Our results actually go beyond {\bf A1} in
Theorem~\ref{thm:A}, namely we give a rather complete description of optimal
stacks $\lbr \gamma_1^*, \dots , \gamma_\ell^*\rbr $ in terms of Wulff shapes and
Wulff plaquettes associated to surface tension $\taub$, and in particular
we explain
geometry behind the claimed infinite sequence of first order transitions. \\
The notions of microscopic contours and level lines are defined in
Section~\ref{sec:level-lines}.  \\
The surface tension $\taub$ is defined in the very beginning of
Section~\ref{sec:Proofs}. \\
The proof part {\bf A2} of
Theorem~A, or more precisely the proof of the corresponding statement for
the reduced model of large contours, Theorem~\ref{thm:B} in the end of
Section~\ref{sec:level-lines}, is relegated to Section~\ref{sec:Proofs}.

{
Throughout the paper we rely on techniques and ideas introduced and developed
in \cite{DKS} and \cite{IST15}. Whenever possible we only sketch proofs which
follow closely the relevant parts
of these papers.
}

\subsection{Some notation}
Let us introduce the following convenient notation: Given two indexed families
of
numbers $\left\{  a_{\alpha}\right\}  $ and $\left\{  b_{\alpha}\right\}  $ we
shall say that $a_{\alpha}\lesssim b_{\alpha}$ \emph{uniformly} in $\alpha$ if
there exists $C$ such that $a_{\alpha}\leq Cb_{\alpha}$ for all indices
$\alpha$. Furthermore, we shall say that $a_{\alpha}\cong b_{\alpha}$ if both
$a_{\alpha}\lesssim b_{\alpha}$ and $b_{\alpha}\lesssim a_{\alpha}$ hold.
Finally, the family $\left\{  a_{\alpha,N}\right\}  $ is ${\small o}_{N}(1)$
if $\sup_{\alpha}\left\vert a_{\alpha,N}\right\vert $ tends to zero as $N$
tends to infinity.

\section{The variational problem}
\label{sec:var}
The variational problems we shall deal with are constrained isoperimetric type
problems for curves lying inside the  box $\bbB_1 = [-1,1]^2\subset \bbR^2$.

Let $\taub$ be a  norm on $\bbR^2$ which
possesses all the lattice symmetries of $\bbZ^2$. For a closed
{piecewise smooth} Jordan curves (called
 loops below)
$\gamma\subset \bbB_1$
 we define $\sfa (\gamma )$ to be the area of its interior
 $\overset{\circ}{\gamma}$ and $\taub (\gamma )$
 to be
 its surface tension,
 \[
  \taub (\gamma ) = \int_\gamma \taub (\frn_s )\dd s .
 \]
A finite family $\calL = \lbr \gamma_1 , \gamma_2, \dots, \gamma_n\rbr\subset \bbB_1$ of
loops  is said to be compatible if
\[
 \forall \, i\neq j\ {\rm either}\, \overset{\circ}{\gamma_i}\cap \overset{\circ}{\gamma_j} =\emptyset
 \ {\rm or}\ \overset{\circ}{\gamma_i}\subseteq \overset{\circ}{\gamma_j}\ {\rm or}\
 \overset{\circ}{\gamma_j}\subseteq \overset{\circ}{\gamma_i} .
\]
Thus compatible families have a structure of
finite number of disjoint stacks of loops,
with loops within each stack being ordered by inclusion.
Given $\sfD_\beta  >0$ consider
the following family, indexed by $\delta >0$, of  minimization problems:
\be{eq:VP-delta}
\tag{${\rm VP}_\delta$}
\min_{\calL} \calE_\beta \lb \calL ~|\delta\rb
:=
\min_{\calL}
\lbr \frac{\lb\delta -\sfa (\calL ) \rb^2}{2 \sfD_\beta } + \taub (\calL )
\rbr ,
\ee
where

\[
\sfa\left(  \mathcal{L}\right)  =\sum_{\gamma\in\mathcal{L}}
\sfa\left(  \gamma\right)\quad {\rm and}\quad \taub (\calL ) = \sum_{\gamma\in \calL}\taub (\gamma )
\]
\subsection{Rescaling of \eqref{eq:VP-delta}.} Let $\sfe$ be a lattice direction.
Set
\be{eq:scaling}
v= \frac{\delta}{\taub (\sfe )\sfD_\beta}, \
\sigma_\beta  = \sfD_\beta\taub (\sfe )\ {\rm  and}\  \tau (\cdot ) =
\frac{\taub (\cdot )}{\taub (\sfe )} .
\ee
Since
\[
 {
 \calE_\beta \lb \calL ~|\delta\rb =
 }
 \frac{\lb\delta -\sfa (\calL ) \rb^2}{2\sfD_\beta} + \taub (\calL )
 =
\frac{\delta^2}{2\sfD_\beta} +
\taub (\sfe )\lbr - v \sfa (\calL ) + \tau (\calL ) +
\frac{\sfa (\calL )^2 }{2\sigma_\beta } \rbr ,
\
\]
we can reformulate the family of variational problems \eqref{eq:VP-delta} as follows:
 for $a \geq  0$ define
 \be{eq:Xi-a}
 \tau (a ) = \min\lbr \tau (\calL )\, :\, \sfa (\calL ) = a\rbr .
\ee
Then study
\be{eq:VP-v}
\tag{${\rm VP}_v$  }
\min_{a\geq 0}\lbr -v a + \tau  ( a) + \frac{a^2}{2\sigma_\beta }\rbr  .
\ee
The problem \eqref{eq:VP-v} has a clear geometric interpretation:
we want to find $a =a (v ) \geq 0$
such that the
straight line
with slope $v$ is the support line at $a (v) $ to the graph
$a\mapsto \tau  (a ) + \frac{a^2}{2\sigma_\beta }$ on $[0, \infty)$.

\subsection{Wulff shapes, Wulff plaquettes and optimal stacks}
By construction $\tau$ inherits lattice symmetries of $\mathbb{Z}^{d}$
and $\tau\left(  \sfe \right)  =1$. In this case,
{the Wulff shape }%
\[
\sfW \triangleq\partial\left\{  x:x\cdot\mathfrak{n}\leq\tau(\mathfrak{n}%
)\ \forall~\mathfrak{n}\in\mathbb{S}^{1}\right\}
\]
\textbf{has} the following properties: Its radius (half of {its}
projection on the $x$ or $y$ axis) equals to $1$. Its area $\sfa (\sfW )$
satisfies
\[
\sfa (\sfW ) \equiv w=\frac{1}{2}\tau (\sfW ) ,
\]
 For any $r>0$ the radius, the area and
the surface tension of the scaled curve $r\sfW$ equal to $r,\ r^{2}w$ and
$2rw$ respectively. The curves $r\sfW$ are (modulo shifts) unique minimizers
of $\tau (\gamma )$
restricted to the loops $\gamma\ $with area $\sfa (\gamma)=r^{2}w:$
\begin{equation}
\label{eq:Wr}
2rw=\tau (  r\sfW )  =\min_{\gamma:\sfa (\gamma)=r^{2}w} \tau (\gamma ) .
\end{equation}
Since $\tau (\sfe )=1$, the maximal radius and, accordingly, the maximal area of the
rescaled Wulff shape which could be inscribed into the  box $\bbB_1$ are
precisely $1$
 and ${w}$. For $b\in[0,{w}]$ let $\sfW_b$ be the Wulff shape of area $b$.
 By convention, $\sfW_b$ is centered at
 the origin.
 Its radius $r_b$ and its surface tension $ \tau (\sfW_b )$ are given by
 \be{eq:Wb}
 r_b = \sqrt{\frac{b}{w}}\quad {\rm and}\quad \tau (\sfW_b )  = 2r_b w =  2\sqrt{b w} .
 \ee
 For $b\in ({w}, 4]$ the optimal (minimal $\tau (\gamma )$)
 loop $\gamma\subseteq \bbB_1$ with
 $\sfa (\gamma ) = b$  is what we call the
Wulff plaquette $\sfP_b$. It is defined as follows. One takes four Wulff shapes of
radius $r\leq 1$ and puts them in four corners of $\bbB_1 ,$ so
that each touches two corresponding sides of $\bbB_1 ;$ then, one takes the convex
envelope of the union of these four Wulff shapes. We will call such Wulff
plaquette as having the radius $r$. It contains four segments of
length $2(1-r),$ and so its surface tension  is $8\left(  1-r\right)  +2wr.$ Its
area is $4-\left(  4-w\right)  r^{2}.$ In this way, the Wulff plaquette $\sfP_b$
of area $b\in [{w}, 4]$ has
the radius $r_b$ and  surface tension $ \tau (\sfP_b )$ given by
\be{eq:Pb}
 r_b = \sqrt{\frac{4-b}{4-w}}\quad {\rm and}\quad \tau (\sfP_b )  =
 8 - 2r_b (4- w) =
 8- 2\sqrt{(4- w)(4-b )} .
 \ee
 \begin{remark}
  By convention $\sfW_{w} = \sfP_{w}$.
  Also note that both in the case of Wulff shapes
  and Wulff plaquettes,
  \be{eq:dtau-b}
 \frac{\dd}{\dd b}\tau \lb \sfW_b \rb = \frac{1}{r_b}\quad {\rm and}\quad
 \frac{\dd}{\dd b}\tau \lb \sfP_b \rb  = \frac{1}{r_b},
 \ee
 for any $b\in (0, w )$ and, respectively, for any $b\in (w , 4)$.
 \end{remark}
\smallskip

\noindent\textbf{Definition.} For $b\in [0,4]$ define the optimal shape
\be{eq:S-shape}
\sfS_b = \sfW_b\1_{b\in[0,{w})} + \sfP_b \1_{b\in[{w}, 4]} .
\ee

Let now $\mathcal{L}$ be a family of compatible
loops.
{
For any $x\in \bbB_1$ let $n_{\calL} (x )$ be the number of
loops $\gamma\in\calL$ such that $x\in \mathring{\gamma}$.
}
The areas $b_\ell = |x\in \bbB_1 : n_{\mathcal{L}} (x)\geq \ell |$ form a
non-increasing sequence. Therefore $\calL^* = \lbr \sfS_{b_1}, \sfS_{b_2} , \dots\rbr  $
is also a compatible family. Obviously $\sfa (\calL^* )= \sfa (\calL )$, but
$\tau (\calL^* )\leq \tau (\calL )$. Consequently, we can restrict attention only to
compatible families which contain exactly one stack of optimal shapes $\sfS_{b_\ell}$.

Furthermore, \eqref{eq:dtau-b} implies that for each $\ell\in \bbN$ optimal $\ell$-stacks
could be only of two types:
Let $a\in \bbR^+$.

\noindent
\textbf{Definition} (Stacks $\calL_\ell^1 (a )$ of type 1).
These contain $\ell -1$
identical Wulff plaquettes and a Wulff shape, all  of the same radius
$r = r^{1, \ell } (a )$,
which should  be recovered from
\be{eq:a-L1}
a  = (\ell -1)(4- (4-w )r^2 ) + wr^2 = 4(\ell -1 ) - r^2 \lb \ell (4- w) -4\rb .
\ee
That is, if $\ell (4-w ) \neq 4$, then
\be{eq:rad-1a}
r^{1 , \ell}  (a) = \sqrt{\frac{4(\ell - 1 ) -a  }{4 (\ell -1 ) -  w }}.
\ee
\begin{remark}
\label{rem:deg}
Of course, if $  \frac{4}{4-w} \in \bbN$, then \eqref{eq:a-L1} does not determine $r$
for $\ell^* = \frac{4}{4-w}$.
In other words in this case for any $r\in [0, 1]$ the stack   of $(\ell^*-1)$
Wulff plaquettes of radius $r$ and the Wulff shape of the very same radius $r$ on the top
of them has area $4( \ell^*-1 )$ and surface tension $8(\ell^*-1)$. This
 introduces a certain degeneracy in the problem, but, upon inspection,
 the inconvenience happens to be of a purely  notational
nature, and in the sequel we shall ignore this case, and  assume that the
surface tension $\tau$
satisfies
\be{eq:non-dg}
\ell^* := \frac{4}{4-w }\not\in\bbN .
\ee
\end{remark}
We proceed to work under assumption \eqref{eq:non-dg}.
The range $ Range(\calL_\ell^1 )$  of areas $a$,  for which $\ell$-stacks of
type 1 are defined is:
\be{eq:range-a1}
 Range(\calL_\ell^1 ) =
\begin{cases}
 \ [4(\ell -1 ), \ell w], \quad &{\rm if}\ \ell < \ell^*,\\
\  [\ell w, 4(\ell -1 )], \quad &{\rm if}\ \ell >  \ell^*.
 \end{cases}
\ee
In either of the two cases above the surface tension
\be{eq:st-L1}
 \tau \lb \calL_\ell^1 (a )\rb = 8 (\ell -1) +
2 {\rm sign}(\ell^* -\ell)
\sqrt{(4(\ell -1)-a )(4 (\ell -1 ) -w )  } .
\ee
In view of \eqref{eq:rad-1a} and \eqref{eq:non-dg}
\be{eq:e1-deriv}
\frac{\dd }{\dd a} \tau \lb \calL_\ell^1 (a )\rb  =
\sqrt{\frac{4 (\ell -1 ) - w}{4(\ell -1)-a }}
= \frac{1}{r^{1 , \ell}  (a)}.
\ee
\smallskip

\noindent
\textbf{Definition} (Stacks $\calL_\ell^2 (a )$ of type 2). These contain
$\ell$ identical Wulff plaquettes of radius
\be{eq:rad-2a}
r^{2 , \ell}  (a) = \sqrt{\frac{4\ell -a}{(4 - w ) \ell}},
\ee
as it follows from $a= \ell (4- (4-w )r^2 )$.
The range $ Range(\calL_\ell^2 )$  of areas $a$,  for which stacks of
type 2 are defined (for a given value of $\ell\in\bbN$ ),  is:
\be{eq:range-a2}
 Range(\calL_\ell^2 ) =  [\ell w , 4\ell]
\ee
Substituting the value of the radius \eqref{eq:rad-2a} into \eqref{eq:Pb} we infer
that the surface tension of a stack of type 2 equals to:
\be{eq:st-L2}
 \tau \lb \calL_\ell^2 (a ) \rb  =
 8\ell - 2 \sqrt{ (4\ell -a ) (4 - w ) \ell }\quad {\rm and}\quad
 \frac{\dd }{\dd a}   \tau \lb \calL_\ell^2 (a )\rb
= \frac{1}{r^{2 , \ell}  (a)}.
\ee
Note that by definition
\be{eq:vr-stacks}
\calL_\ell^1 (4 (\ell -1) ) = \calL^2_{\ell - 1} (4(\ell -1) )\ \text{and}\
\calL_\ell^1 ( \ell w  ) = \calL_\ell^2 ( \ell w ).
\ee
Also for notational convenience we set $\calL^2_0 (0 ) = \emptyset$.
\smallskip

\noindent
Set $\tau_\ell (a ) = \min\lbr \tau\lb \calL_\ell^1 (a ) \rb  ,\tau \lb  \calL_\ell^2 (a )\rb\rbr$, where we
define
\be{eq:InfRange}
\tau \lb \calL_\ell^i (a )\rb  = \infty\quad \text{if $a\not\in Range (\calL_\ell^i )$}
\ee
We can rewrite \eqref{eq:VP-v} as
\be{eq:VP-vl}
\min_{a \geq 0 , {\ell\in\bbN} }\lbr -va +\tau_\ell (a ) + \frac{a^2}{2\sigma_\beta }\rbr =
\min_{a \geq 0 , \ell\in{\bbN}, i=1,2}\lbr -va +
 \calG_\ell^i (a )  \rbr,
\ee
where we put
 \be{eq:G-i}
 \calG_\ell^i (a ) = \tau\lb \calL_\ell^i (a )\rb + \frac{a^2}{2\sigma_\beta},\quad i=1,2.
 \ee
\subsection{Irrelevance of $\calL_\ell^1$-stacks for $\ell >\ell^*$}
\label{sub:irrelevance}
For $\ell >  \ell^*$,
\[
Range (\calL_\ell^1 ) = [\ell w, 4 (\ell -1)]
\subset {Range (\calL_\ell^2 )}  = [\ell w, 4 \ell] .
\]
By the second of \eqref{eq:vr-stacks} the values of $\calL_\ell^1$ and $\calL_\ell^2$ coincide
at the left end point $a = \ell w$. On the other hand, $r_\ell^1 (a ) \leq r_\ell^2 (a )$
for any $a \in [\ell w, 4( \ell -1 )]$. Hence,  \eqref{eq:e1-deriv} and
\eqref{eq:st-L2} imply that $\tau\lb \calL_\ell^1 (a )\rb \geq \tau\lb \calL^2_\ell (a )\rb$
whenever $\ell >\ell^*$.
\subsection{Variational problem for stacks of type $2$.}
We start solving the problem \eqref{eq:VP-vl}, by considering the simpler one:
\be{eq:VP-vl-2}
\min_{a \in \cup_{\ell \ge 0} [w\ell,4\ell] , \ell\in\bbN_0}\lbr -va + \calG_\ell^2 (a )\rbr
\df \min_{a \in \cup_{\ell \ge 0} [w\ell,4\ell] , \ell\in\bbN_0} F_v (\ell , a ) ,
\ee
where we have defined (see \eqref{eq:st-L2}).
\be{eq:Fv}
F_v (\ell , a) = -v a + \frac{a^2}{2\sigma_\beta} + 8\ell - 2 \sqrt{ (4\ell -a ) (4 - w ) \ell } .
\ee
Recall \eqref{eq:InfRange}  that we set $\calG_\ell^2 (a ) = \infty$ whenever
$a\not\in Range (\calL_\ell^2 )$, as described in
\eqref{eq:range-a2}. {In this way the functions  $\calG_\ell^2$ are defined on $\bbR$;
each one has a support line at any slope $v$.}
 In the variational problem \eqref{eq:VP-vl-2} we are looking for the lowest such support line, which
is precisely the support line to the graph of $\calG^2 = \min_{\ell}\calG_\ell^2$.
For a
generic slope $v$ there is exactly one value $\ell (v )$ for which the minimum is realized.
However, for certain critical values $v^{\ast}$ of the slope $v$
it might happen that the minimal support line touches the graphs of $\calG_\ell^2$ for
several different $\ell$-s.
As the following theorem shows, at every such  critical value $v^{\ast}$, the index  $\ell = \ell (v)$ of
the optimal stack $ \mathcal{G}_{\ell}^{2}$ jumps exactly by one unit up, that is
$\ell\left(  v^{\ast}+\varepsilon\right)  =\ell\left(  v^{\ast}-\varepsilon
\right)  +1$ for $\varepsilon>0$ small. Furthermore, these transition are of the first order,
both in terms of the radii and the areas
of optimal stacks.

\begin{theorem}
 \label{thm:VP-vl-2}
There
 exists an
 increasing sequence of critical slopes
 $0= v_0^*<v_1^* < v_2^* < \dots$ and an
 increasing sequence of the area values
 \[
  0 = a_0^+ <  a_1^- < a_1^+ < a_2^- <a_1^+ < a_2^- < \dots
 \]
such that $a_\ell^\pm\in Range \lb \calL^2_\ell\rb = [\ell w, 4\ell]$
for every $\ell\in \bbN$, and:

\noindent
1. For $v\in [0, v_1^* )$ the empty stack $\calL^2_0 (0 )$ is the unique solution to
\eqref{eq:VP-vl-2}.

\noindent
2. For each $\ell\in \bbN$
such that $v^*_\ell < 1+ \frac{\ell w}{\sigma_\beta}$, the minimum in  \eqref{eq:VP-vl-2} is attained,  for all
$v\in (v_\ell^* , 1+ \frac{\ell w}{\sigma_\beta})$,   at
$a=a_\ell^- = \ell w$.
The corresponding stack is
$\calL^2_\ell (\ell w )$.

\noindent
3. For the remaining values of $v\in (v_\ell^*\vee ( 1+ \frac{\ell w}{\sigma_\beta}) , v_{\ell +1}^* )$,
the picture is the following: for
each $\ell\in \bbN$ there exists a continuous increasing bijection
\[
a_\ell : [v_\ell^*\vee ( 1+ \frac{\ell w}{\sigma_\beta}) , v_{\ell +1}^*]
\mapsto [a_\ell^- , a_\ell^+],
\]
such that
 for each $v\in (v_\ell^*\vee ( 1+ \frac{\ell w}{\sigma_\beta}) , v_{\ell +1}^* )$
 the stack
  $\calL^2_\ell \lb a_\ell (v )\rb $ corresponds to the unique solution to \eqref{eq:VP-vl-2}.

\noindent
4. At
critical slopes $v_1^*, v_2^*, \dots $ the transitions happen. There is  a coexistence:
\be{eq:vl-coex}
\frac{(v- a_{\ell-1 }^+)^2}{2\sigma_\beta} + \tau  \lb \calL^2_{\ell -1 } (a_{\ell -1}^+) \rb
=
\frac{(v- a_{\ell }^-)^2}{2\sigma_\beta} + \tau  \lb \calL^2_{\ell  } (a_{\ell }^-) \rb .
\ee
Also, the radii of plaquettes of optimal stacks at coexistence points are increasing:
Set $b_\ell^\pm = a_\ell^\pm /\ell$. Then, $b_{\ell -1}^+ > b_\ell^-$, and hence
\be{eq:rad-l}
r_{b_{\ell -1}^+} <  r_{{b_\ell^-}}
\ee
for every $\ell\in \bbN$.
\end{theorem}
We shall prove Theorem~\ref{thm:VP-vl-2} under additional non-degeneracy assumption
\eqref{eq:non-dg}. However, the proof could be easily modified in order to accommodate the
degenerate case as well.

The fact that the problem should exhibit   first order transitions
could be easily understood from
\eqref{eq:st-L2}. The crux  of the proof below  is
to
show that when $v$ increases, the number $\ell=\ell(v)$ of layers of the corresponding optimal stack
$\calL^2_\ell \lb a_\ell (v )\rb $ either
stays the same or increases by one, and, above all, to deduce
all the results without resorting to explicit and  painful computations.

\begin{proof}[Proof of Theorem~\ref{thm:VP-vl-2}]
Let us start with the following  two facts:

\paragraph{\bf Fact~1.} For every $\ell\geq 1$, the function
$\left\{  \mathcal{G}_{\ell}^{2}\left(  a\right)  \right\}  $ is strictly  convex.
 Let $a_\ell = a_\ell (v)$ be the point where a line with the slope
$v$ supports its graph.
If
\be{eq:v-l-slope}
v  \leq
\frac{\dd^+}{\dd a}\Big|_{a = \ell w}
\left\{  \mathcal{G}_{\ell}^{2}\left(  a\right)  \right\}
 \stackrel{\eqref{eq:st-L2}}{=}
 \frac{1}{r^{2, \ell } (\ell w )} + \frac{\ell w}{\sigma_\beta} =
1 + \frac{\ell w}{\sigma_\beta} ,
\ee
then $a_\ell (v ) = \ell w$.
In the remaining region $v > 1 + \frac{\ell w}{\sigma_\beta}$
the value $a_\ell (v )$ is recovered from:
\be{eq:der-a}
 v = \frac{\dd }{\dd a}\Big|_{a=a_\ell } \tau\lb \calL_\ell^2 (a)\rb  + \frac{a_\ell}{\sigma_\beta}
 \stackrel{\eqref{eq:st-L2},\eqref{eq:rad-2a}}{=}
 \sqrt{\frac{4 -w}{4-a_\ell/\ell}} + \frac{a_\ell}{\sigma_\beta} \df
 \sqrt{\frac{4- w}{4-b_\ell}} + \frac{\ell b_\ell}{\sigma_\beta} .
\ee
Thus, both $a_\ell (v )$ and $b_\ell (v )\df a_\ell (v )/\ell $ are well defined for
any $v\in\bbR$.  Of course we consider only $v\in \left[0  , \infty \right)$. If  $m >\ell$, then, by
definition, we have for all
$v > 1 + \frac{m w}{\sigma_\beta}$
(i.e. when both curves $\calG_\ell^2$ and $\calG_m^2$ have tangent lines with slope $v$)  that
\be{eq:lm-bound1}
v =  \sqrt{\frac{4- w}{4-b_\ell (v )}} + \frac{\ell b_\ell (v )}{\sigma_\beta} =
 \sqrt{\frac{4- w}{4-b_m (v )}} + \frac{m b_m (v )}{\sigma_\beta}.
\ee
If $v \in (1 + \frac{\ell w}{\sigma_\beta}, 1 + \frac{m w}{\sigma_\beta}] $, then $b_m (v ) = w$ and
\be{eq:lm-bound2}
 v = \sqrt{\frac{4- w}{4-b_\ell (v )}} + \frac{\ell b_\ell (v )}{\sigma_\beta} <
 \sqrt{\frac{4- w}{4-b_m (v )}} + \frac{m b_m (v )}{\sigma_\beta},
\ee
Finally, if $v \leq 1 +{\frac{\ell w}{\sigma_\beta}}  $ then $b_\ell (v )= b_m (v )  = w$, and the second inequality
in \eqref{eq:lm-bound2} trivially holds. Together   \eqref{eq:lm-bound1} and \eqref{eq:lm-bound2}  imply  that
for any $v\in [0,\infty)$,
\be{eq:am-bm}
 a_m (v ) > a_\ell (v ) \quad {\rm and} \quad  b_m (v ) \leq  b_\ell (v )  .
\ee
\paragraph{\bf Fact~2.}
It is useful to think about $F_v$ as a function of continuous variables $\ell, a\in\bbR_+$.
By direct inspection $-\sqrt{\ell (4\ell -a )}$ {is strictly convex
on $\bbR^2_+$ and thus also on the
convex sector}
\[
\bbD_w\df \lbr (\ell , a ): 0\leq  \ell w \leq  a \leq 4\ell\rbr \subset
\bbR^2_+ .
\]
Hence, $F_v$ is strictly convex on $\bbD_w$ as well. This has the following implication:
If  $(\ell_1 , a_1 )\neq (\ell_2 , a_2 )$ are such that $F_v (\ell_1 , a_1 ) =
F_v (\ell_2 , a_2 )$, then
\be{eq:s-conv}
F_v (\ell_1 , a_1 )  > F_v (\lambda\ell_1 + (1-\lambda )\ell_2 ,
\lambda a_1  + (1-\lambda ) a_2)
\ee
for any $\lambda \in (0,1)$.
\smallskip

\noindent
Let us go back to \eqref{eq:VP-vl-2}. Clearly $\min_{a\geq 0 , \ell\in \bbN_0} F_v (\ell ,a )$ is
attained for all $v$, and, furthermore
 $(0,0) = {\rm argmin} (F_v )$ for all $v$ sufficiently small. It is also clear that
  $(0,0)\not\in {\rm argmin} (F_v )$ whenever $v$ is sufficiently large.

Therefore there exists the unique minimal values $v_{1}^{\ast} > 0$ and  $\ell_{1}^{\ast}  \ge 1$,  and, accordingly
the value $a_1^- \in Range \lb \calL^2_{\ell_1^{\ast}} \rb = [\ell_1^{\ast} w, 4\ell_1^{\ast} ]$ satisfying the condition
{\[
F_{v_{1}^{\ast}}\left(  0,0\right)  =F_{v_{1}^{\ast}}\left(  \ell_{1}^{\ast},a_{1}^{-}\right)
\]
}

{Let us show that $\ell_{1}^{\ast}=1$. Indeed, assume that $\ell_{1}^{\ast}>1$. But then
for the value $\ell=1$, intermediate between
$\ell=0$ and $\ell_{1}^{\ast}$, we have
$F_{v_1^{\ast}} (1, \frac{a_{1}^{-}}{\ell_{1}^{\ast}} )
> F_{v_1^{\ast}} (0,0 )$, which contradicts the convexity property
\eqref{eq:s-conv}.
}
{
Hence $a_1^- \in [w, 4)$. By the same strict convexity argument,
\[
 F_{v_1^*} (1 , a_1^- )  =  \min_{a}  F_{v_1^*} (1 , a  ) <  \min_{\ell >1 , a } F_{v_1^{\ast}}  (\ell , a )
  \]
By continuity  the inequality above  will persist for $v > v_1^*$
with
$v-v_1^*$ sufficiently small.
Also, $F_v (1, a_1 (v )) < F_v (0,0)$ for every
$v > v^*_1$, since the function $a\mapsto \tau\lb \calL_1^2 (a )\rb$ is strictly convex.
}
This means that there exists the maximal
$v_2^* > v_1^* $, $a_1^+ > a_1^-$  and a continuous bijection
$a_1 : [v_1^* , v_2^* ]\mapsto [a_1^- , a_1^+ ]$ such that
$(1, a (v )) = {\rm argmin} F_v$ on $(v_1^* , v_2^* )$.
\smallskip

Now we can proceed by induction. Let us define the segment  $[v_\ell^* , v_{\ell +1}^*]$ as the
maximal segment of values of
the parameter $v$, for which there exist the corresponding segment $[a_\ell^- , a_\ell^+]$ and a
continuous non-decreasing function  $a_\ell:  [v_\ell^* , v_{\ell +1}^*]
\mapsto [a_\ell^- , a_\ell^+]$ such that $(\ell , a_\ell (v )) = {\rm argmin }F_v $
for $v\in ( v_\ell^* , v_{\ell +1}^* )$. (If there are several such segments for the same
value of $\ell$, we take for $[v_\ell^* , v_{\ell +1}^*]$, by definition,  the leftmost one. Of course, we will show below that it can not be the case, but we do not suppose it now.) Our induction hypothesis is that the open segment  $(v_\ell^* , v_{\ell +1}^*)$  is non-empty, and that
\be{eq:ind-aftervl}
\min_{m <\ell} \min_{a}  F_v (m , a ) > \min_{a} F_v (\ell , a ),
\ee
for $v > v_\ell^*$.
We have already checked it for $\ell =1$.


Clearly, $(\ell , a)\not\in {\rm argmin }F_v$ whenever $v$ is sufficiently large. Thus,  $v_{\ell +1}^* <\infty$.
By induction hypothesis \eqref{eq:ind-aftervl},
\be{eq:ind-vl1-ast}
F_{v_{\ell+1}^*} (m, a ) = F_{v_{\ell+1}^*} (\ell , a_\ell^+ ) = \min F_{v_{\ell+1}^*}
\ee
implies that $ m>\ell$.
As before, using convexity and continuity arguments we can check that if \eqref{eq:ind-vl1-ast}  holds for some
$m >\ell$ and $a$ (with $(m , a )\in \bbD_w$)
then necessarily $m = \ell+1$ and, furthermore,
\[
\min_{a} F_v (\ell +1 , a ) < \min_{m >\ell +1 , a} F_v (m, a )
\]
for $v > v_{\ell+1}^*$ with $v - v_{\ell+1}^*$ sufficiently small.  The first part of the induction step is justified.

Assume finally  that $F_v (\ell , a_\ell ) = F_v (\ell +1 , a_{\ell +1} ) = \min F_v$. Then
$a_\ell < a_{\ell+1}$, as it is stated in \eqref{eq:am-bm}. By the same authority,
$b_{\ell } \geq  b_{\ell +1}$, and hence $r_{b_\ell} \leq
r_{b_{\ell +1}}$.
By
convexity of both $\tau\lb \calL_\ell^2 (a )\rb $ and
$\tau\lb \calL_{\ell +l}^2 (a )\rb $ the inequality $a_\ell < a_{\ell+1}$ implies that
\[
 \min_a F_u (\ell , a) >  \min_a F_u (\ell+1 , a)\ \text{for any $u>v$} .
\]
Consequently,  $\min_{a}  F_u (\ell  , a ) > \min_{a} F_u (\ell +1 , a )$ for any $u > v_{\ell+1}^*$,  and we are home.
\end{proof}
\subsection{Analysis of \eqref{eq:VP-v}}
As we already know, $\ell$-stacks of type 1
cannot be optimal whenever $\ell >\ell^*$ (see definition  \eqref{eq:non-dg}). Let us explore what happens  if
$\ell <\ell^*$. In this case
\[
 Range (\calL_\ell^1 )= [4 (\ell -1), \ell w]\quad {\rm and}\quad
 Range (\calL_\ell^2 ) = [\ell w , 4\ell].
\]
Thus, $Range (\calL_\ell^1 )$ shares endpoints $4(\ell -1 )$ and
$\ell w$ with $Range (\calL_{\ell-1}^2 )$ and, respectively,
$ Range (\calL_\ell^2 )$, and all these ranges have disjoint interiors. So, in principle,
 $\ell$-stack of type 1 may become optimal.
 Note that by our convention,  \eqref{eq:vr-stacks} and \eqref{eq:G-i},
 \be{eq:end-p-eq}
 \calG_{\ell - 1}^2 (4 (\ell -1 )) = \calG^1_\ell (4 (\ell -1 ))\quad {\rm and}\quad
 \calG^1_\ell (\ell w ) = \calG^2_\ell (\ell w ),
 \ee
 so, for $\ell <\ell^*$ the two families $\calG^1_\ell$, $\calG^2_\ell$ merge together into a
 single continuous function. In fact, it is even smooth,
 except that the tangent to its graph becomes vertical at endpoints $4\ell$.
 Let $v^*_\ell$ be the critical slope for variational problem \eqref{eq:VP-vl-2},
 as described in Theorem 3. By construction, there is a line $\frl (v_\ell^* )$
 with slope $v_\ell^*$ which supports both the graphs of $\calG_{\ell -1}^2$ and
 of $\calG_{\ell }^2$.

 \noindent
 {\bf Definition.} Let us say that the graph of
 $\calG^1_\ell$
 {\em sticks out below} $\frl (v_\ell^* )$
 if there exists $a\in Range (\calL_\ell^1 )$ such that
 \be{eq:sticks-out}
 \calG_\ell^1 (a ) < \calG_\ell^2 (a_\ell^- ) - v_\ell^* (a_\ell^-  -a ) .
 \ee
\begin{figure}[h]
\begin{center}
\includegraphics[width=\textwidth]{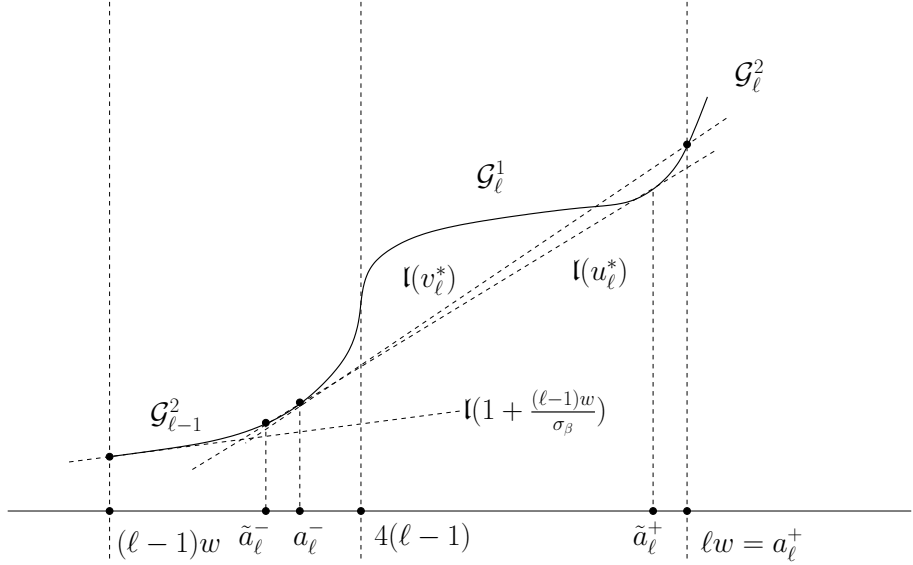}
\end{center}
\setlength{\abovecaptionskip}{-240pt}
\caption{The graph of $\calG_\ell^1$ sticks out below $\frl (v_\ell^* )$.
The transition slope $\tilde v_\ell^*$ satisfies
$1 +\frac{(\ell - 1)w}{\sigma_\beta}< \tilde v_\ell^* < v_\ell^* <
1 +\frac{\ell w}{\sigma_\beta}$.
}
\label{fig:Graphs}
\end{figure}
\noindent
 Obviously, there are optimal $\ell$-stacks of type 1 iff the graph of
 $\calG^1_\ell$
 {sticks out below} $\frl (v_\ell^* )$.
 \begin{proposition}
 \label{prop:vlstar-cond}
  For any $\ell < \ell^*$ the  graph of $ \calG_\ell^1$ {sticks out below} $\frl (v_\ell^* )$ iff
  \be{eq:vlstar-cond}
  v_\ell^* < 1 + \frac{\ell w}{\sigma_\beta} .
  \ee
  Equivalently, this happens iff
  \be{eq:vlstar-cond-1}
  \calG_{\ell -1}^2 (a ) > \calG_\ell^2 (\ell w) - \lb \ell w - a\rb \lb 1 + \frac{\ell w}{\sigma_\beta}\rb ,
  \ee
  for any $a\in Range (\calL_{\ell -1}^2 ) = \left[ (\ell -1) w , 4 (\ell -1 )\right]$.
  \end{proposition}
  \begin{proof}
  Note first of all that in view of \eqref{eq:v-l-slope}, the condition \eqref{eq:vlstar-cond}
  necessarily implies that
   $a_\ell^- = \ell w$.
   Consequently the fact that \eqref{eq:vlstar-cond} and \eqref{eq:vlstar-cond-1} are
   equivalent is straightforward,
   since by construction $\frl (v_\ell^* )$ supports both graphs.

   Note that since  $\frac{\dd^+}{\dd a} \calG_{\ell -1}^2 (4 (\ell -1 ))=\infty$ and
   since {$\calL_{\ell -1}^2 ( 4(\ell - 1)) = \calL_\ell^1 ( 4(\ell -1 ))$}  the graph
   of $\calG_\ell^1$ has
   to stay above $\frl (v_\ell^* )$ for values of $a\in Range (\calL_\ell^1 )$ which are sufficiently
   close to $4 (\ell -1 )$.

   Let us compute the second derivative
   \be{eq:2d-El1}
   \frac{\dd^2}{\dd a} \calG^1_\ell (a )
   \stackrel{\eqref{eq:st-L1}}{=} \frac{1}{\sigma_\beta} - \frac{1}{2}
   \sqrt{\frac{\ell w - (4(\ell -1)}{(a - 4 (\ell -1 ))^3}}.
   \ee
   The expression on the right hand side above is increasing (from $-\infty $) on
   $Range (\calL_\ell^1 )$. If it is non-positive on the whole interval, then the graph of $\calG_\ell^1$ is
   concave and it cannot stick out. Otherwise, the graph of $\calG_\ell^1$ is convex near
   the right end point $\ell w$ and hence its maximal derivative on the convex part is attained at
   $a = \ell w$ and equals to $1 + \frac{\ell w}{\sigma_\beta }$. Therefore,
   \eqref{eq:vlstar-cond}  is a necessary condition for the graph of
   $\calG_\ell^1$ to stick out.

   To see that it is also a sufficient condition recall once again
    that
   $v_\ell^* < 1 +\frac{\ell w}{\sigma_\beta }$ means that $\frl (v_\ell^* )$ supports
   $\calG_\ell^2$ at the left end point $a=\ell w$. {But then, $\calG_\ell^1$  goes
   below $\frl (v_\ell^* )$ for all values of $a\in Range (\calL_\ell^1 )$ which are
   sufficiently close to $\ell w$,
   because ``the union of $\calG_\ell^1$  and $\calG_\ell^2$'' is smooth at $a=\ell w$.}
   In particular, it should have a convex part.
  \end{proof}
\begin{remark} The argument above does not imply that for $\ell <\ell^* = \frac{4}{4-w }$ the
graph of $\calG_\ell^1$ sticks out if
 it has a convex part near $\ell w$. The latter is a necessary condition which gives the following
 upper bound on the maximal number of layers $\ell$ such that $\calG_\ell^1$ may stick out: Let
 us substitute $a = \ell w$ into \eqref{eq:2d-El1}:
 \be{eq:stick-ub}
 \frac{1}{\sigma_\beta} - \frac{1}{2} \sqrt{\frac{\ell w - 4(\ell -1)}{(\ell w - 4 (\ell -1 ))^3}} > 0 \
 \stackrel{\eqref{eq:non-dg}}{\Leftrightarrow} \ \ell < \ell^*\lb 1- \frac{\sigma_\beta}{8}\rb .
 \ee
\end{remark}
\begin{proposition}
 \label{prop:stick-out}
 If $w \leq 2\sigma_\beta$, then the graph of $\calG_\ell^1$ does not stick out for 
 any value of
 $\ell <\ell^* \lb 1- \frac{\sigma_\beta}{8}\rb$ (and hence
 stacks of type 1 are never optimal).

 If, however, $w > 2\sigma_\beta$, then there exists a
 number $k^*$, $1 \leq k^* < \ell^*
 \lb 1- \frac{\sigma_\beta}{8}\rb$ such that the graphs of
 $\calG_\ell^1$ stick out below
 $\frl (v_\ell^* )$-s for any $\ell = 1, \dots, k^*$, and they do not
 stick out for $\ell > k^*$.
\end{proposition}
\begin{proof}
The proof comprises two steps.
\smallskip

\noindent
\step{1} We claim that the graph of $\calG_1^1$ sticks out 
below $\frl (v_1^* )$ iff $w > 2\sigma_\beta$.

Indeed, recall that $\frl (v_1^* )$ is the line which passes through the origin and which is tangent to
the graph of $\calG_1^2$. Since the latter is convex and increasing, $v_1^* < 1+ \frac{w}{\sigma_\beta}$
iff
\[
 \calG_1^2 (w ) < w \lb 1+ \frac{w}{\sigma_\beta}\rb \ \Leftrightarrow\
 2w +\frac{w^2}{2\sigma_\beta} < w + \frac{w^2}{\sigma_\beta}\
 \Leftrightarrow\  w > 2\sigma_\beta,
\]
so the claim follows from Proposition~\ref{prop:vlstar-cond}.
\smallskip

\noindent
\step{2} We claim that for any $1\leq m <\ell$, if the graph of $\calG_\ell^1$ sticks out below
$\frl (v_\ell^* )$, then the graph of $\calG_m^1$ sticks out below
$\frl (v_m^* )$.

Assume that  \eqref{eq:vlstar-cond-1} holds. First of all take $a= (\ell -1 )w$.
Recall that $\calG_\ell^2 (\ell w ) = 2\ell w + (\ell w)^2/2\sigma_\beta$. Therefore,
\be{eq:G-l-diff}
\calG_{\ell -1}^2\lb (\ell -1 )w\rb - \calG_\ell^2 (\ell w ) + w\lb 1+ \frac{\ell w}{\sigma_\beta}\rb =
\frac{w^2}{2\sigma_\beta} - w  >0 .
\ee
Furthermore, if we record the range $a\in Range (\calL_{(\ell-1)}^2 )$ as $a = (\ell -1)w  +c$;
$c\in [0, (4- w)(\ell -1 )]$, then, in view of \eqref{eq:G-l-diff}, the necessary and sufficient
condition \eqref{eq:vlstar-cond-1} for the graph of $\calG_\ell^1$ to stick out reads as:
\be{eq:c-comp}
\begin{split}
 &\calG_{\ell -1}^2\lb (\ell -1 )w +c \rb - \calG_\ell^2 (\ell w ) + (w-c )
 \lb 1+ \frac{\ell w}{\sigma_\beta}\rb\\
 &\quad  = \lb \frac{w^2}{2\sigma_\beta} - w\rb  + \int_0^c \lb \frac{\dd}{\dd \tau} \calG_{\ell -1}^2
 \lb  (\ell -1 )w +\tau \rb - \lb 1+ \frac{\ell w}{\sigma_\beta}\rb\rb \dd\tau\\
 &\quad
 \stackrel{\eqref{eq:st-L2}}{=}
 \lb \frac{w^2}{2\sigma_\beta} - w\rb
 + \int_0^c
 \lb
 \frac{1}{r^{2, \ell -1} ( (\ell -1 )w +\tau )}
 - \lb 1+ \frac{\ell w}{\sigma_\beta}\rb
 \rb \dd\tau > 0  .
\end{split}
\ee
for any $c \in [0,(4- w)(\ell -1 )]$.

Now for any $k\in \bbN$ (and $\tau \in [0,(4-w )k]$),
\[
 {r^{2, k } ( k w +\tau )}
 \stackrel{\eqref{eq:rad-2a}}{=} \sqrt{\frac{4 k - (k w +\tau )}{ (4- w)k }}
 \stackrel{\eqref{eq:non-dg}}{=} \sqrt{1 - \frac{\tau }{ (4- w) k  }} .
\]
That means that for a fixed $\tau$ the function
$k\mapsto r^{2, k } ( k w +\tau )$ is
increasing in $k$. Therefore, \eqref{eq:c-comp} implies that,
\be{eq:c-comp-m}
\lb \frac{w^2}{2\sigma_\beta} - w\rb
 + \int_0^c
 \lb
 \frac{1}{r^{2, m -1} ( (m -1 )w +\tau )}
 - \lb 1+ \frac{m w}{\sigma_\beta}\rb
 \rb \dd\tau > 0 ,
 \ee
 for any $m=1, \dots , \ell$ and, accordingly, any
 $c\in [0, (m-1)(4-w )]$. Consequently,
 \eqref{eq:vlstar-cond-1} holds for any $m\leq \ell$.
\end{proof}
Assume now that $\ell <\ell^*$ and that the graph of $\calG^1_\ell$ sticks out below
$\frl (v_\ell^* )$. This means that there exits a range of slopes
$(\tilde v_\ell^* , v_\ell^* )$ such that for any $v\in (\tilde v_\ell^* , v_\ell^* )$,
 one can find $a = a(v )$, such that $\calL_\ell^1 (a(v ))$
is  the unique solution to the
variational problem \eqref{eq:VP-v}. By \eqref{eq:e1-deriv},
\[
 v = \frac{1}{r^{1, \ell} (a (v ))} + \frac{a (v )}{\sigma_\beta } \geq 1 + \frac{4 (\ell -1)}{\sigma_\beta}.
\]
Hence, in view of Proposition~\ref{prop:vlstar-cond} (and in view
of the fact that by Proposition~\ref{prop:stick-out} the graph of
$\calG_{\ell-1}^1$ has to stick out as well and consequently
$v_{\ell -1}^* <1 +\frac{(\ell -1 )w}{\sigma_\beta}$ ),
\[
 \tilde v_\ell^* > 1 + \frac{(\ell -1 )w}{\sigma_\beta} > v_{\ell - 1}^* .
\]
Which means that in the range of slopes $[1+ \frac{(\ell -1)w}{\sigma_\beta}, \tilde v_\ell^*)$
the $(\ell-1)$-stacks of type 2 continue to be optimal.

The structure of solutions and first order transitions in terms of optimal layers and optimal
areas is summarized in Theorem~\ref{thm:k-star-slopes} and depicted on Figure~\ref{fig:Trans}.


\begin{bigthm}
\label{thm:k-star-slopes}
If $w \leq 2\sigma_\beta$, then solutions
to the variational problem \eqref{eq:VP-v} are as
described in Theorem~\ref{thm:VP-vl-2}.

 If, however, $w > 2\sigma_\beta$, then there exists a number
 $1 \leq k^* < \ell^*
 \lb 1- \frac{\sigma_\beta}{8}\rb$ such that
 the following happens: For every $\ell = 1, \dots, k^*$ there exists a slope $\tilde v_\ell^*$;
 \[
  1 + \frac{(\ell -1)w}{2\sigma_\beta} < \tilde v_\ell^* < v_\ell^* < 1 + \frac{\ell w}{2\sigma_\beta},
 \]
such that

\noindent
1. The empty stack $\calL_0^2$ is the unique solution to \eqref{eq:VP-v}
for $v\in [0,\tilde{v}_1^* )$.

\noindent
2. For every $\ell = 1, \dots, k^*$ there is an area $\tilde a_\ell^-\in (4(\ell -1), \ell w)$
and a continuous increasing bijection
$\tilde a_\ell :  [\tilde v_\ell^* ,1+ \frac{\ell w}{2\sigma_\beta}]\mapsto
[ \tilde a_\ell^- , \ell w ]$
such that the $\ell$-stack $\calL_\ell^1 \lb \tilde a_\ell (v) \rb$ of type 1 is the
unique solution to \eqref{eq:VP-v} for every
$v\in (\tilde v_\ell^* ,1+ \frac{\ell w}{2\sigma_\beta})$.

\noindent
3.  For every $\ell <  k^*$
the $\ell$-stack $\calL_\ell^2 (a_\ell (v ))$ of type 2 is the unique solution
to \eqref{eq:VP-v} for every $v\in (1 + \frac{\ell w}{\sigma_\beta} , \tilde v_{\ell+1}^* )$,
where $a_\ell$ is the bijection described in Theorem~\ref{thm:VP-vl-2}.

\noindent
4.  The stack $\calL_{k^*}^2 (a_{k^*} (v ))$ is the unique solution to \eqref{eq:VP-v}
for every $v\in [ 1+ \frac{k^* w}{\sigma_\beta} , v_{k^* +1} )$.

\noindent
5.  Finally, for every $\ell > k^*$ the transition slope
$v^*_\ell \geq 1+ \frac{\ell w}{\sigma_\beta }$, and the stack $\calL_\ell^2 (a_\ell (v))$
of type 2 is the unique solution to \eqref{eq:VP-v} for every $v \in (v_\ell^* , v^*_{\ell +1})$.
\end{bigthm}

\vskip -0.5in
\begin{figure}[h]
\begin{center}
\includegraphics[width=\textwidth]{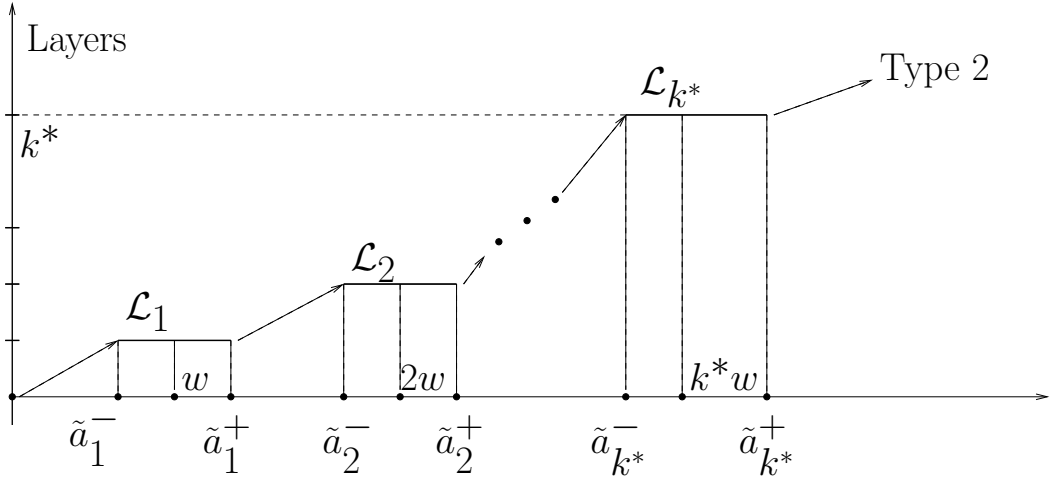}
\end{center}
\setlength{\abovecaptionskip}{-275pt}
\caption{ For $\ell =1, \dots,  k^* <\ell^*$ families
$\calL^1_\ell$ of type one are optimal for
$a\in (\tilde a_\ell^- , \ell w )$, whereas
families $\calL^2_\ell$ of type 2 are optimal for $a\in (\ell w , \tilde{a}_\ell^+ )$.
First order
transitions - jumps in terms of number of optimal layers from $\ell-1$ to $\ell$, and in terms of sizes
of optimal areas  from $\tilde{a}_{\ell-1}^+$ to $\tilde{a}_{\ell}^-$  (where we set $\tilde{a}_0^+ = 0$) -
occur at transition slopes
$\tilde{v}_{\ell}^*$.
For $\ell >k^*$ only families of type two are optimal,
and first order transitions occur  as described in Theorem~\ref{thm:VP-vl-2}.
}
\label{fig:Trans}
\end{figure}




\section{Low temperature level lines}
\label{sec:level-lines}
 The main contribution to the event $\left\{  \Xi_N \geq  \rho_\beta
N^{3} + A N^{2}\right\}  $ comes from bulk fluctuations and creations of
macroscopic size facets (large contours - see below) of the interface $\Gamma
$. In order to formulate the eventual reduced model let us first of all
collect the corresponding notions and facts from \cite{IS08}. \smallskip

\subsection{Bulk Fluctuations.}
For each $\beta$ fixed bulk fluctuations are governed by local limit results for
sums
of Bernoulli random variables, as the linear size of the
system $N\to\infty$.
Let us record a more
quantitative  version of \eqref{eq:bulk-ll}:
For every $K$ fixed
\begin{equation}
\label{eq:BulkLL}
\log\bbP_{N , \beta}\lb \Xi_N  =  \rho_\beta N^3 +A N^2\big| \Gamma
\rb
 =
-N{\frac{\lb \delta_\beta  -\frac{\alpha (\Gamma )}{N^2} \rb^2}{2D_\beta}}
%
+ O \left(  \log N\right)  ,
\end{equation}
uniformly
in
$A\leq K$ and $ |\alpha(\Gamma)|\leq KN^{2}$. \smallskip

\subsection{Contours and their Weights.} There is a natural contour
parametrization of surfaces $\Gamma$. Namely, given an interface $\Gamma$ and,
accordingly, the height function $h_{\Gamma}$ which, by definition, is
identically zero outside $\Lambda_{N}^{\circ}$, define the following
semi-infinite subset $\widehat{\Gamma}$ of $\mathbb{R}^{3}$,
\[
\widehat{\Gamma}\, =\, \bigcup_{\overset{(x,y, k)}{k < h_{\Gamma}(x, y)}}\,
\left(  (x,y, k) +\widehat{C}\right)  ,
\]
where $\widehat{C} = [-1/2, 1/2]^{3}$ is the unit cube. The above union is
over all $(x, y)\in\mathbb{Z}^{2}$ and $k\in1/2 +\mathbb{Z}$.

Consider now the level sets of $\Gamma,$ i.e. the sets
\[
H_{k}\, =\, H_{k}\left(  \widehat\Gamma\right)  \, =\, \left\{  \left(
x,y\right)  \in\mathbb{R}^{2}:\left(  x,y,k\right)  \in\widehat\Gamma\right\}
,\ k=-N,\, -N+1,\dots, \, N.
\]
We define
\textit{contours} as the connected components of sets $\partial H_{k}$,
subject to south-west splitting rules,
{
see Section~2.1 in \cite{IST15} or \cite{DKS}.
}
The length $\left\vert {\gamma}\right\vert $ of a contour is
defined in an obvious way. Since, by construction all contours are closed
{
self-avoiding
}
polygons composed of
the nearest neighbor bonds of $\Lambda_{N}^{*}$, the
notions of interior $\mathrm{int}(\gamma)$ and exterior $\mathrm{ext}%
(\gamma)$ of a contour $\gamma$ are well defined. A contour ${\gamma}$\textit{
}is called a $\oplus$-contour ($\ominus$-contour), if the values of the
function $h_{\Gamma}$ at the immediate exterior of $\gamma$ are smaller
(bigger) than those at the immediate interior of $\gamma$.

Alternatively, let us orient the bonds of each contours $\gamma\subseteq
\partial H_{k}$ in such a way that when we traverse $\gamma$ the set $H_{k}$
remains to the right. Then $\oplus$-contours are those which are clockwise
oriented with respect to their interior, whereas $\ominus$-contours are
counter-clockwise oriented with respect to their interior.

Let us say that two oriented contours $\gamma$ and $\gamma^{\prime}$ are
compatible, $\gamma\sim\gamma^{\prime}$, if

\begin{enumerate}
\item Either $\mathrm{int} (\gamma)\cap\mathrm{int} (\gamma^{\prime
})=\emptyset$ or $\mathrm{int} (\gamma)\subseteq\mathrm{int} (\gamma^{\prime
})$ or $\mathrm{int} (\gamma^{\prime})\subseteq\mathrm{int} (\gamma)$.

\item Whenever $\gamma$ and $\gamma^{\prime}$ share a bond $b$, $b$ has the
same orientation in both $\gamma$ and $\gamma^{\prime}$.
\end{enumerate}

A family $\Gamma= \left\{  {\gamma}_{i}\right\}  $ of oriented contours is
called consistent, if contours of $\Gamma$ are pair-wise compatible. It is
clear that the interfaces $\Gamma$ are in one-to-one correspondence with
consistent families of oriented contours. The height function $h_{\Gamma}$
could be reconstructed from a consistent family $\Gamma=\left\{
\gamma\right\}  $ in the following way: For every contour $\gamma$ the sign of
$\gamma$, which we denote as $\mathrm{sign} (\gamma)$, could be read from it
orientation. Then,
\be{eq:h-fumction}
h_{\gamma}(x,y )\, =\, \mathrm{sign}(\gamma) \chi_{\mathrm{int}(\gamma)}(x,y )
\quad\text{and}\quad h_{\Gamma}\, =\, \sum_{\gamma\in\Gamma} h_{\gamma} ,
\ee
where $\chi_{A}$ is the indicator function of $A$.

In this way
the weights $w_{\beta}(\Gamma)$ which appear
in \eqref{eq:GammaWeight-SOS} could be recorded as follows:
Let $\Gamma=\left\{  \gamma\right\}  $ be a consistent
family of oriented (signed) contours, Then,
\begin{equation}
w_{\beta}(\Gamma)\,=\,\mathrm{exp}\left\{  -\beta\,\sum_{\gamma\in\Gamma
}\left\vert \gamma\right\vert \right\}  . \label{Gammaweight}%
\end{equation}
In the sequel we shall assume that $\beta$ is sufficiently large. By
definition the weight of the flat interface \textbf{is} $w_{\beta}(\Gamma
_{0})=1$. \smallskip


\subsection{Absence of Intermediate and {Negative} Contours.}
In the sequel a claim that a certain property holds
{\em for all $\beta$ sufficiently large} means
 that one can
find $\beta_0$, such that it holds for all $\beta \geq \beta_0$.

For all $\beta$
sufficiently large the following rough apriori bound holds (see (6.2) in
\cite{IS08}): There exist positive combinatorial (that is independent of
$\beta$) constant $\nu$  such
for every $b_{0}>0$ fixed,
\begin{equation}
\log{\mathbb{P}}_{N, \beta } \left(  |\alpha(\Gamma)|\,>\,bN^{2}\right)
\,\lesssim\,-\nu\beta  N\sqrt{b}, \label{c3nu}%
\end{equation}
uniformly in $b\geq b_{0}$ and in $N$ large enough. A comparison with
\eqref{eq:BulkLL} reveals that we may fix $K_\beta = K_\beta (A)$
and ignore $\Gamma$ with $\alpha
(\Gamma)\geq K_\beta  N^{2}$. Now let the interface $\Gamma$ with
$\alpha(\Gamma)\leq K_\beta  N^{2}$ be given by a consistent collection
of contours, and assume that $\gamma\sim\Gamma$. Of course $\alpha(\Gamma
\cup\gamma)=\alpha(\Gamma)+\alpha(\gamma)$. Then \cite{IS08} there exists a
constant $c_\beta  = c_\beta  (
{
A
}
)$ such that
\begin{equation}
\label{eq:IntermBound}-\log{\mathbb{P}_{N , \beta}}\left(  \Gamma\cup\gamma
\,\Big\vert\,\Xi_N
{
\geq
}
\rho_\beta N^{3}+ A N^{2}\right)  \lesssim c_\beta  \frac
{|\gamma|^{2}}{N}
{
\1_{{\rm sign}(\gamma ) = 1}
}
-\beta|\gamma| .
\end{equation}
Consequently, there exists $\epsilon_\beta = \epsilon_\beta (A) >0$, such that we can
rule out all
contours $\gamma$ with
$\epsilon^{-1}_\beta\log N \leq|\gamma|\leq\epsilon_\beta N$,
{
 and all negative contours $\gamma$ with $|\gamma|> \epsilon_\beta N$.
 }

It is not difficult to see that \eqref{eq:p-beta} implies that
\be{eq:eps-beta}
\liminf_{\beta\to\infty} \frac{1}{\beta}\log \epsilon_\beta (A) > -\infty
\ee
for any $A$ fixed.
\smallskip


\noindent\textbf{Definition. }
The notion of big and small contours depends on the running value of
excess area $A$, on the linear size of the system $N$
and on the inverse temperature $\beta$.
Namely, a
contour $\gamma$ is said to be large, respectively small,  if
\be{eq:large-small}
|\gamma|\geq \epsilon_\beta (A) N\quad\text{and, respectively, if}\quad
|\gamma| \leq \frac{1}{\epsilon_\beta (A )}\log N .
\ee
Since we already know that intermediate contours
{
and large $\ominus$-type contours
}
could be ignored,
let us use $\hat{\bbP}_{N , \beta}$ for the restricted ensemble which
contains only  $\oplus$-type large  or  small contours.
\smallskip

\subsection{Irrelevance of Small Contours.}
Let $\Gamma= \Gamma^{l}\cup\Gamma^{s}$ is a compatible collection of contours
with $\Gamma^{l}$ being the corresponding set of large contours of $\Gamma$
and, accordingly, $\Gamma^{s}$ being the collection of small contours of
$\Gamma$. Clearly,
\be{eq:large-length}
\log
{
\bbP_{N, \beta}
}
\lb \abs{\Gamma^{l}} \geq cN\rb \leq -\beta c N
\lb 1-\smo{N}(1)\rb ,
\ee
uniformly in $c$ and all $\beta$ sufficiently large.  Hence, again
by comparison with \eqref{eq:BulkLL} we can ignore all collections
of large contours with total length $\abs{\Gamma^{l}} \geq K_\beta N$.

On the other hand, elementary low density percolation arguments and the
$\pm$-symmetry of height function imply that
\begin{equation}
\label{eq:small}\log {\mathbb{P}}_{N, \beta}\left(  \left\vert \alpha(\Gamma^{s}
)\right\vert \geq b ~\bigm| \Gamma^{l} \right)  \lesssim-
\frac{{\rm e}^{4\beta}b^{2}}%
{N^2}\wedge\frac{\epsilon_\beta b}{\log N  } ,
\end{equation}
uniformly in $\Gamma^{l}$ and in $b\lesssim K_\beta N^{2}$. Again, a comparison with
\eqref{eq:BulkLL} implies that we may restrict attention to the case of
$\left\vert \alpha(\Gamma^{s} )\right\vert \lesssim N^{3/2}$. Such corrections
to the total value of $\Xi$ are invisible on the scales \eqref{anot} we
work with and, consequently, the area shift induced by small contours may be
ignored. \smallskip

\subsection{The Reduced Model of Big Contours.} In the sequel we shall
employ the following notation: $\mathcal{C}$ for clusters of non-compatible
\emph{small} contours and $\Phi_{\beta}(\mathcal{C})$ for the corresponding
cluster weights which shows up in the cluster expansion representation of
partition functions. Note that although the family of clusters $\mathcal{C}$
does depend on the running microscopic scale $N$, the weights $\Phi_{\beta
}(\mathcal{C})$ remains the same. By usual cluster expansion estimates, for
all $\beta$ sufficiently large
\begin{equation}
\label{eq:clusterweight}\sup_{\mathcal{C}\neq\emptyset}
\text{e}^{
{
2 \beta \lb {\rm diam}_\infty \lb
\mathcal{C}\rb +1\rb
}
}
\left\vert \Phi_{\beta}(\mathcal{C}
)\right\vert \lesssim 1 .
\end{equation}
\smallskip

\noindent We are ready now to describe the reduced model which takes into
account only large contours: The probability of a compatible collection
$\Gamma = \left\{  \Gamma_{1} , \dots,\Gamma_{k}\right\}  $ of \emph{large}
contours is given by
\begin{equation}
\label{eq:Plarge}
\mathbb{Q}_{N, \beta} \left(  \Gamma\right) \,
{
\eqvs
}
\,
\mathrm{exp}%
\left\{  -\beta\sum\left\vert \Gamma_{i}\right\vert -
{
\sum_{ \mathcal{C}\subset\Lambda_{N}}
\1_{\lbr \mathcal{C}\not \sim \Gamma\rbr}
}
\Phi_{\beta}(\mathcal{C})\right\}  ,
\end{equation}
The conditional distributions of $\Xi^{v}_N$ and $\Xi^{s}_N$ given such
$\Gamma^l$ are still $\mathrm{Bin}\left(  \left\vert V_{N} (\Gamma)\right\vert ,
p^{v}\right)  $ and $\mathrm{Bin}\left(  \left\vert S_{N} (\Gamma)\right\vert
, p^{s}\right)$, and we shall use $\bbQ_{N, \beta}$ for the corresponding
joint distribution.

For future references let us reformulate the bulk fluctuation bound
\eqref{eq:BulkLL}  in terms of the reduced measure $\bbQ_{N, \beta}$:
For every $K$ fixed
\begin{equation}
\label{eq:BulkLL-Q}
\log\bbQ_{N , \beta}\lb \Xi_N   \geq \rho_\beta N^3 +A N^2\big| \Gamma
\rb
 =
-N{\frac{\lb \delta_\beta  -\frac{\alpha (\Gamma )}{N^2} \rb^2}{2D_\beta}}
%
+ O \left(  \log N\right)  ,
\end{equation}
uniformly
in
$A\leq K$ and $ |\alpha(\Gamma)|\leq KN^{2}$. \smallskip

The notation  for conditional reduced
measures is
\be{eq:QNA}
\bbQ_{N, \beta }^A = \bbQ_{N, \beta }\lb\, \cdot\, \big|\, \Xi_N\geq
\rho_\beta N^3 + AN^2\rb .
\ee
From now on we shall concentrate on proving {\bf A2} of Theorem~\ref{thm:A}
for $\bbQ_{N, \beta}^A$-measures instead of $\bbP_{N, \beta}^A$-measures.
Specifically, we shall prove:
\begin{bigthm}
\label{thm:B}
Assume that bulk occupation probabilities $p_v$ and $p_s$ satisfy
\eqref{eq:p-beta}. Then there exists $\beta_0 <\infty$ such that for any
$\beta >\beta_0$ the following holds: Let  $0 < A_1 (\beta ) <A_2 (\beta )
< A_3 (\beta )  <\dots$ and, respectively
$ 0 < a_1^- (\beta ) < a_1^+ (\beta ) <a_2^- (\beta )
< a_2^+ (\beta ) <\dots $
be the sequences of  as described in
part {\bf A1} of Theorem~\ref{thm:A}.

Then, for any $A\in [0, A_1)$ the set $\Gamma$ of large contours is empty
with $\bbQ_{N, \beta}^A$-probability tending to one. On the other hand,
for any $\ell = 1, 2, \dots$ and for any $A\in (A_\ell , A_{\ell +1})$,
the set $\Gamma$ contains,
 with $\bbQ_{N, \beta}^A$-probability
tending to
one,
exactly $\ell$ large contours $\Gamma = \lbr \Gamma_1, \dots , \Gamma_\ell\rbr$.
Moreover, the rescaled contours from  $\Gamma$
 concentrate
 in Hausdorff distance $\sfd_{\sfH}$  near the optimal stack
 $\lbr \gamma_1^* , \dots , \gamma_\ell^*\rbr $, as described
 in part {\bf A1} of
 Theorem~\ref{thm:A}, in the
 following sense: For any $\epsilon > 0$ ,
 \be{eq:Hausd-stack-Q}
 \lim_{N\to\infty}\bbQ^A_{N, \beta} \lb
 {
 \sum_{i=1}^{\ell-1}
 \sfd_{\sfH} \lb \frac{1}{N}\Gamma_i , \gamma^*_i\rb
 + \min_{x\, :\, x+ \gamma_\ell^^ \subset \bbB_1} \sfd_{\sfH} \lb
 \frac{1}{N}\Gamma_\ell , x+ \gamma^*_\ell\rb
 \geq \epsilon
 }
 \rb = 0.
 \ee
\end{bigthm}

\section{Proofs}
\label{sec:Proofs}

\subsection{Surface tension.}
Let us say that a nearest-neighbor path $\gamma$ on $\mathbb{Z}^{2}$
is admissible if it can be a realized as a portion of a level line
of the height function $h_\gamma$ in \eqref{eq:h-fumction}.
Following \eqref{eq:Plarge} we associate with $\gamma$ its {\em free} weight
\begin{equation}
\label{eq:gammaweight}
w_{\beta}^\sff (\gamma) = \text{e}^{-\beta\left\vert
\gamma\right\vert - \sum_{\mathcal{C}\;\not \sim \;\Gamma} \Phi_{\beta
}(\mathcal{C}) }.
\end{equation}
We say that an admissible $\gamma = \lb \gamma_0, \dots , \gamma_n\rb$ is
$\gamma: 0\to  x$ if its end-points satisfy $\gamma_0 = 0$ and $\gamma_n =x$.
The corresponding two-point function and the surface tension are
\begin{equation}
\label{eq:twopoint-st}
G_{\beta}(x)\triangleq\sum_{\gamma: 0\to  x} w_{\beta
}^\sff (\gamma)\quad\text{and}\quad \tau_{\beta}(x) = -\lim_{n \to\infty}\frac
1{n}\log G_{\beta}(\lfloor n x\rfloor).
\end{equation}
Recall that we are considering only sufficiently large values of $\beta$. In
particular, \eqref{eq:clusterweight} applies, and
the surface tension $\taub$ in \eqref{eq:twopoint-st} is well defined.

In the notation of \eqref{eq:gammaweight}
given a large level line $\Gamma\subset B_N$ we define its {\em free} weight
\begin{equation}
\label{eq:weight-free}
w_\beta^{\sff} (\Gamma ) =
{\rm e}^ {  -\beta \left\vert \Gamma\right\vert - \sum_{\mathcal{C}\not \sim \Gamma }
\Phi_{\beta}(\mathcal{C})} .
\end{equation}
In this way the measure $\bbQ_N$ in \eqref{eq:Plarge} describes a gas of
 large level lines which interact between each other and with the boundary $\partial B_N$.
 The statement below  is well understood {(see e.g. [DKS]),}  and it holds for all sufficiently low temperatures:
 \begin{lemma}
  \label{lem:shape-lim}
  Let $\eta \subset \bbB_1$ be a rectifiable Jordan curve. Given any sequence of positive numbers
  $\epsilon_N$ such that $\lim_{N\to \infty}\epsilon_N = 0$ and
  $\lim_{N\to \infty}N\epsilon_N = \infty$,
  the following holds:
  \be{eq:shape-lim}
  \lim_{N\to\infty} \frac{1}{N}  \log \lb \sum_{\Gamma\subset B_N}
  w_\beta^{\sff} (\Gamma)\1_{\lbr \dd_{\sfH}\lb \frac{1}{N}
  \Gamma , \eta \rb \leq \epsilon_N\rbr}\rb  = -\taub (\eta ) .
  \ee
 \end{lemma}

\subsection{Lower bounds on $\bbQ_{N, \beta }\lb \Xi_N  \ge
\rho_\beta N^3 + AN^2\rb$.} Let us apply part {\bf A.1} of Theorem~\ref{thm:A}
for the surface tension $\taub$ defined in \eqref{eq:twopoint-st}
 and bulk occupation probabilities $p_v (\beta ), p_s (\beta )$ which
 satisfy \eqref{eq:p-beta}. Assume that
$A\in \lb A_\ell (\beta ) , A_{\ell +1} (\beta )\rb$ for some $\ell =0, 1, \dots$.
Let $a^* = 0$ if $\ell = 0$ and $a^* \in (a_\ell^- , a_\ell^+ )$ be the optimal
rescaled area as described in Theorem~\ref{thm:A}. Then,
\begin{proposition}
\label{prop:lb-astar}
In the notation of
\eqref{eq:quantities} and
{
\eqref{eq:VP-delta}
}
the following lower bound holds:
\be{eq:lb-astar}
\begin{split}
\liminf_{N\to\infty} \frac{1}{N} \log \bbQ_{N, \beta }\lb \Xi_N \geq
\rho_\beta N^3 + AN^2\rb &\geq
{
 - \min_{\calL} \calE_\beta \lb \calL ~|\delta_\beta \rb =
}
\\
 & = -
\lb\frac{\lb \delta_\beta - a^*\rb^2}{2D_\beta} + \taub (a^* )\rb .
\end{split}
\ee
\end{proposition}
\begin{proof}
If $a^*=0$, then the claim follows from \eqref{eq:BulkLL-Q}. \\
Assume that $a^* \in (a_\ell^- , a_\ell^+ )$ for some $\ell\geq 1$.  Let
$\gamma_1\subseteq \gamma_2 \subseteq \dots \subseteq \gamma_\ell\subset \bbB_1$ be the optimal
stack as described in Theorem \ref{thm:k-star-slopes}.  Pick a sequence $\epsilon_N$ which satisfies
conditions of Lemma~\ref{lem:shape-lim} and, for $J=1, \dots , \ell$ define tubes
\[
A_N^j =   (1-  (1+ 3 (j-1))\epsilon_N )N\mathring{\gamma}_j\setminus
(1-  (2+ 3 (j-1))\epsilon_N )N\mathring{\gamma}_j .
\]
Lemma~\ref{lem:shape-lim} implies that for any $j =1, \dots , \ell$
\be{eq:j-tube}
\lim_{N\to\infty} \frac{1}{N}  \log \lb \sum_{\Gamma_j\subset A_N^j} w_\beta^{\sff} (\Gamma_j )
 \rb  = -\taub (\gamma_j ) .
  \ee
By construction $A_N^j$-s are disjoint and there exists $c_1 = c_1 (\beta , a^* ) >0$ such that
\be{eq:dist-AN}
\min_{1\leq j \ell}\dd_{\sfH}  (A_N^{j-1} , A_N^{j} ) \geq c_1 N\epsilon_N \,
\ee
where we put $A_N^{0} =\partial B_N$. Hence, in view of \eqref{eq:clusterweight},
\be{eq:tube-correction}
\mathrm{exp}
\left\{  -\beta\sum_{j=1}^\ell \left\vert \Gamma_{j}\right\vert - \sum
_{\substack{{\mathcal{C}\not \sim \Gamma }\\{\mathcal{C}\subset\Lambda_{N}}}}
\Phi_{\beta}(\mathcal{C})\right\} \geqs
{\rm e}^{-c_2 \ell N {\rm e}^{-2c_1 N\epsilon_N}}
\prod_1^\ell w_\beta^{\sff} (\Gamma_j )
\ee
for any collection $\Gamma = \lb\Gamma_1  , \dots , \Gamma_\ell \rb$ of level lines
satisfying
$\Gamma_j \subset A_N^{j}$ for $j=1, \dots , \ell$.

Note also that, for any $j=1, \dots \ell$,  if a large level line $\Gamma_j \subset A_N^j$, then
\be{eq:tube-area}
N^2 \sfa (\gamma_j ) (1- 3\ell\epsilon_N )^2 \leq \sfa (\Gamma_j ) \leq N^2 \sfa (\gamma_j ) .
\ee
Hence, by \eqref{eq:BulkLL-Q},
\begin{equation}
\label{eq:Bulk-lb}
\frac{1}{N} \log\bbQ_{N , \beta}\lb \Xi_N  \geq  \rho_\beta N^3 +A N^2\big| \Gamma
\rb
\geq
-{\frac{\lb \delta_\beta  - a^*  \rb^2}{2D_\beta}} - \frac{6\delta_\beta^2 \ell\epsilon_N}{D_\beta} ,
\end{equation}
for any collection $\Gamma = \lb\Gamma_1 , \dots , \Gamma_\ell \rb$ of such level lines.

Putting \eqref{eq:j-tube}, \eqref{eq:tube-correction} and \eqref{eq:Bulk-lb} together
(and recalling that $\epsilon_N$ was chosen to satisfy conditions of
Lemma~\ref{lem:shape-lim})
we deduce \eqref{prop:lb-astar}.
\end{proof}
\subsection{Strategy for proving upper bounds.}
Below we explain the strategy which we employ for proving
\eqref{eq:Hausd-stack-Q}. For the rest of the section let us assume that a sufficiently large
$\beta$ and an excess area value $A \in \lb A_\ell (\beta ), A_{\ell +1} (\beta )\rb$ are fixed.
\smallskip

\noindent
\step{1} (Hausdorff distance and Isoperimetric rigidity).
We shall employ the same notation $\dd_\sfH$ for two and three dimensional Hausdorff
distances.

Given a family $\calL$ of compatible loops define the height function $h[\calL] :\bbB_1\mapsto \bbN$,
\be{eq:height-f}
h[\calL] (y ) = \sum_{\gamma\in\calL}
\1_{\lbr y \in \mathring{\gamma}\rbr} .
\ee
In the case of  the optimal stack (recall that we fixed $\beta$ and $A$, so the latter notion is well defined)
$\calL^*  = \lb \gamma_1^* , \dots , \gamma_\ell^*\rb $ as described in
Theorem \ref{thm:k-star-slopes}, we say that $x$ is admissible if $x+\gamma^*_\ell\subset \bbB_1$.
If $x$ is admissible, then
we use $h^*_x = h^*_x (A , \beta )$ for the height function of
$\lb \gamma_1^* , \dots , x+ \gamma_\ell^*\rb $.
Of course if $\calL^*$ is 
{of type-2,}  the only admissible $x$ is $x=0$.

We can think about $h[\calL ]$ in terms of its epigraph, which is a three dimensional subset
of {$\bbB_1\times\bbR_{{+}}$.}  In this way the notion of Hausdorff distance
$\dd_{\sfH} \lb h[\calL ] , h^*_x\rb$ is well defined.

{
We will need the qualitative stability properties of the minima of the functional
$$
{
\calE_\beta \lb \calL~|~ \delta_\beta \rb =
}
\sum_{\gamma\in\calL} \taub (\gamma ) + \frac{(\delta_\beta - \sfa (\calL ))^2}{2 D_\beta }.
$$
Namely, we claim that  for every $\nu >0$ there
exists $\rho = \rho_\beta  (\nu , A) >0 $ such
that
\be{eq:Iso-Stack}
{
\calE_\beta \lb \calL~|~ \delta_\beta \rb
}
>  \taub (a^* ) + \frac{(\delta_\beta - a^*)^2}{2 D_\beta } + \rho ,
\ee
whenever {$\min_{x} \dd_{\sfH} \lb h[\calL ], h^*_x \rb >\nu$.} Such stability properties are known for the Wulff and
constrained Wulff isoperimetric problems. That is, if  $\sf W$ is the Wulff loop with area $a\left(  \sf W\right)  $ inside,
see Section 3.2, and $\gamma$ is a loop with the same area, $a\left(
\gamma\right)  =a\left(  \sf W\right)  ,$ which satisfies $\min_{x}\dd_{\sfH}\left(
\gamma,\sf W+x\right)  >\nu,$ then for some $\rho=\rho\left(  \nu\right)  >0$ we
have%
\[
\tau\left(  \gamma\right)  \geq\tau\left(  \sf W\right)  +\rho.
\]
That is the content of the \textit{generalized Bonnesen inequality}, proven in
\cite{DKS}, see relation 2.4.1 there. The same stability property holds for
the constrained case, that is when we impose the additional constrain that the
loop $\gamma$ has to fit inside a square, and where the role of the Wulff
shape $\sf W$ is replaced by the Wulff plaquette $\sf P,$ see \cite{SS} for more
details. We will show now that the above quantitative stability of the
surface tension functional implies the quantitative stability of the
functional $\mathcal{E}_{\beta}\left(  \mathcal{\ast}{\Huge |}\delta_{\beta
}\right)  .$
}

{
Let us prove \eqref{eq:Iso-Stack}. Suppose $\dd_{\sfH}$ is bigger than $\nu.$ There can be
several reasons for that. The simplest is that the number of levels in the
epigraph of $h\left[  \mathcal{L}\right]  $ is different from that for
$h^{\ast}.$ (In which case the Hausdorff distance in question is at least $1$.)
In order to estimate the discrepancy $\rho$ in that case we have to look for
the minimizer of the functional $\mathcal{E}_{\beta}\left(  \mathcal{L}%
{\Huge |}\delta_{\beta}\right)  $ under additional constraint on $\mathcal{L}$
to have a different number of levels than the optimal stack $h^{\ast}.$ Let
$\mathcal{E}_{\beta}^{w}\left(  \delta_{\beta}\right)  $ be the
\textit{minimal value} of the functional $\mathcal{E}_{\beta}\left(
\mathcal{L}{\Huge |}\delta_{\beta}\right)  $ over these `wrong' loop
configurations. The function $\mathcal{E}_{\beta}^{w}\left(  \delta_{\beta
}\right)  ,$ as a function of $\delta_{\beta},$ is piecewise continuous, as is
the true optimal value $\mathcal{E}_{\beta}\left(  \delta_{\beta}\right)
=\tau_{\beta}\left(  a^{\ast}\right)  +\frac{\left(  \delta_{\beta}-a^{\ast
}\right)  ^{2}}{2D_{\beta}}$ (where $a^{\ast}=a^{\ast}\left(  \delta_{\beta
}\right)  $ is the optimal area corresponding to the excess value
$\delta_{\beta}).$ The difference $\mathcal{E}_{\beta}^{w}\left(
\delta_{\beta}\right)  -\mathcal{E}_{\beta}\left(  \delta_{\beta}\right)  $ is
continuous and non-negative, and it vanishes precisely at the values
$\delta_{\beta}$ corresponding to critical values $A_{\ell}\left(
\beta\right)  $ of the parameter $A.$ Therefore the difference $\mathcal{E}%
_{\beta}^{w}\left(  \delta_{\beta}\right)  -\mathcal{E}_{\beta}\left(
\delta_{\beta}\right)  \equiv\rho_{0}\left(  A,\beta\right)  $ is positive for
our fixed value $A\in\left(  A_{\ell}\left(  \beta\right)  ,A_{\ell+1}\left(
\beta\right)  \right)  .$
}

{
Let now the number of levels $l$ in the epigraph of $h\left[  \mathcal{L}%
\right]  $ is $l\left(  h^{\ast}\right)  $ -- i.e. the same as that for
$h^{\ast}.$  Let $\mathcal{L}_{a}$ be the minimizer of $\mathcal{E}_{\beta
}\left(  \mathcal{L}{\Huge |}\delta_{\beta}\right)  $ over all the families
with $l\left(  h^{\ast}\right)  $ levels and with $a\left(  \mathcal{L}%
\right)  =a.$ If our loop collection $\mathcal{L}$ is far from the optimal
stack $\mathcal{L}_{a^{\ast}}$ -- i.e. if $\min_{x}\dd_{\sfH}\left(  h\left[
\mathcal{L}\right]  ,h_{x}^{\ast}\right)  >\nu$ -- as well as from all other
stacks: $\min_{x}\dd_{\sfH}\left(  h\left[  \mathcal{L}\right]  ,h\left[
\mathcal{L}_{a}\right]  _{x}\right)  >\frac{\nu}{10},$ then our claim (5.12)
follows just from the stability properties of the surface tension functional
$\tau\left(  \ast\right)  .$ If, on the other hand, $\mathcal{L}$ is far from
the optimal stack $\mathcal{L}_{a^{\ast}}$, i.e.  $\min_{x}\dd_{\sfH}\left(
h\left[  \mathcal{L}\right]  ,h_{x}^{\ast}\right)  >\nu,$ but it is close to
some other stack $\mathcal{L}_{\bar{a}}$ from `wrong area stacks': $\min
_{x}\dd_{\sfH}\left(  h\left[  \mathcal{L}\right]  ,h\left[  \mathcal{L}_{\bar{a}%
}\right]  _{x}\right)  <\frac{\nu}{10},$ then we are done, since in that case
$\dd_{\sfH}\left(  h\left[  \mathcal{L}_{a^{\ast}}\right]  ,h\left[  \mathcal{L}%
_{\bar{a}}\right]  \right)  >\frac{9\nu}{10},$ and we already know the
minimizer of $\mathcal{E}_{\beta}\left(  \mathcal{\ast}{\Huge |}\delta_{\beta
}\right)  $ over the set $\left\{  \mathcal{L}:\min_{x}\dd_{\sfH}\left(  h\left[
\mathcal{L}\right]  ,h\left[  \mathcal{L}_{\bar{a}}\right]  _{x}\right)
<\frac{\nu}{10}\right\}  $ to be far from $\mathcal{L}_{a^{\ast}}.$
}

\smallskip

\noindent
\step{2} (Upper bound on the number of large contours, compactness considerations
{and skeleton calculus}).
In principle there should be a nice way to formulate and prove large deviation results
 using compactness of the space of closed connected subsets of $\bbB_1$ endowed with
 Hausdorff distance, see \cite{Cerf}. However, as in the latter work, we shall eventually
 resort to skeleton calculus developed in \cite{DKS}.

By
\eqref{eq:Plarge} and \eqref{eq:clusterweight},
\be{eq:limsup-lengths}
\limsup_{N\to\infty}
\frac{1}{N} \log \bbQ_{N, \beta}\lb\sum \abs{\Gamma_i } \geq KN\rb \leq
-K\beta \lb 1 - \bigof{{\rm e}^{-{4} \beta}} \rb .
\ee
In view of Proposition~\ref{prop:lb-astar} we can reduce attention to families $\Gamma$ of
large contours which satisfy $\sum \abs{\Gamma_i } \leq K_\beta  ( A )N$ for some $K_\beta  ( A )$
large enough. By \eqref{eq:large-small} this means that we can restrict attention to
collections $\Gamma$ of large contours whose cardinality is at most
{
\be{eq:card-large}
\#
\lbr \Gamma_i \in \Gamma
\rbr \leq
n_\beta (A) =  K_\beta  ( A )/\epsilon_\beta (A)
\ee
}
Given a compatible collection
$\Gamma$ of large contours we can talk about the rescaled surface
$h_N [\Gamma ]:= h\left[\frac{1}{N}\Gamma \right]$ and,
accordingly,
consider events
\be{eq:nu-event}
\min_{x} \dd_{\sfH} \lb h_N [\Gamma ], h^*_x (A, \beta )  \rb >\nu .
\ee

{Fix a sequence  $\epsilon_N$ satisfying conditions of Lemma~\ref{lem:shape-lim} and
 set $\ell_N = N\epsilon_N$. Consider a large contour $\Gamma_i \in \Gamma$. By
\eqref{eq:card-large} there are at most $n_\beta (A)$ such contours, and each of them
has a bounded length $\abs{\Gamma_i}\leq KN$. In view of south-west splitting rules
we can view $\Gamma_i$ as a parameterized nearest neighbor loop
$\Gamma_i = \lbr u_0, \dots , u_n = u_0\rbr$ with $n\leq KN$.
\smallskip

\noindent
{\bf Definition}.  The $\ell_N$-skeleton
$\gamma_i$ of $\Gamma_i$ is defined as follows: Set $N\sfu_0 = u_0$ and $\tau_0= 0$.
Then, given $k=0, 1, \dots$ with $\sfu_k$ and $\tau_k$ already defined set
\be{eq:skel-step}
\tau_{k+1} = \min \lbr m >\tau_k ~:~ \abs{u_m - N \sfu_k}_1 > \ell_N\rbr \quad {\rm and}\quad
N\sfu_{k+1} = u_{\tau_{k+1}} ,
\ee
provided $\lbr m >\tau_k ~:~\abs{u_m - \sfu_k} > \ell_N\rbr\neq\emptyset$. Otherwise stop
and set $\gamma_i \subset \bbB_1$ to be the polygonal approximation through the vertices
$\sfu_0, \dots , \sfu_{k} , \sfu_0$. \qed

We write $\Gamma_i \stackrel{\ell_N}{\sim} \gamma_i$
 and, accordingly, $\Gamma \stackrel{\ell_N}{\sim} \calS$, if
 $\calS = \lb\gamma_1, \gamma_2 , \dots \rb$ is a collection of $\ell_N$ skeletons
 compatible with family $\Gamma$ of large contours. Since $\abs{\Gamma_i}\leq KN$
 any compatible $\ell_N$ skeleton $\gamma_i$ has at most $\frac{K}{\epsilon_N}$ vertices.
 Therefore, there are at most
 \be{eq:skel-number}
 \lb \lb {N^2}\rb ^{ \frac{K}{\epsilon_N}}\rb^{n_\beta (A )} = {\rm exp}\lbr
 \frac{2K K_\beta  ( A )\log N}{\epsilon_N \epsilon_\beta (A)}\rbr  = {\rm e}^{\smo{N} (1) N}
 \ee
 different collections of $\ell_N$-skeletons to consider. Thus the entropy of the skeletons does not present an issue, and it would suffice to give
 uniform upper bounds on
 \be{eq:upper-target}
 \bbQ_{N , \beta}^A\lb
 \min_{x} \dd_{\sfH} \lb h_N [\Gamma ], h^*_x (A, \beta )  \rb >\nu ;
 \Gamma \stackrel{\ell_N}{\sim} \calS
 \rb
 \ee
 for fixed $\ell_N$-skeleton collections $\calS$.
}

If $\gamma \subset \bbB_1$ is a closed polygon -- for instance an $\ell_N$-skeleton
of some large contour -- then its surface tension
$\taub (\gamma )$ and its signed area $\sfa (\gamma )$ are  well defined.
Accordingly one defines $\taub (\calS ) = \sum_i \taub (\gamma_i )$ and
$\sfa\lb \calS\rb = \sum_i \sfa (\gamma _i )$ for finite collections
$\calS$
of polygonal lines.
{
We apply now the isoperimetric rigidity
bound \eqref{eq:Iso-Stack}: For every $\nu >0$ there
exists $\rho = \rho_\beta  (\nu , A) >0 $ such
that for all $N$ sufficiently large the following holds:
\be{eq:Iso-Stack-skel}
{
\calE_\beta \lb \calS~|~ \delta_\beta \rb =
}
\sum_{\gamma\in\calS} \taub (\gamma ) + \frac{(\delta_\beta - \sfa (\calS ))^2}{2 D_\beta }
>  \taub (a^* ) + \frac{(\delta_\beta - a^*)^2}{2 D_\beta } + \rho ,
\ee
whenever $\calS$ is an $\ell_N$-skeleton; $\calS\stackrel{\ell_N}{\sim}\Gamma$  of
a collection $\Gamma$ of large contours which satisfies \eqref{eq:nu-event}.
}

{
For $n \leq n_\beta (A )$ (see \eqref{eq:card-large}) and $\rho >0$
consider the collection $\frS_N (\rho )$ of  families of $\ell_N$-skeletons
$\calS = (\gamma_1 , \dots , \gamma_n )$ which satisfy  \eqref{eq:Iso-Stack-skel}.
}

Then  \eqref{eq:Hausd-stack-Q} would be a consequence
of the following statement:
\begin{theorem}
 \label{thm:UB-L-net}
There exists a positive function $\alpha$ on $(0, \infty )$ such that the following happens:
Fix $\beta$ sufficiently large and let $A$ be as in the conditions of Theorem~\ref{thm:B}.
Then for any $\rho >0$ fixed,
\be{eq:UB-L-net}
 \max_{\calS \in \frS_N (\rho )} \frac{1}{N}  \log \bbQ_{\beta , N}^A
 \lb
 \Gamma \stackrel{\ell_N}{\sim} \calS
 \rb  <  - \alpha (\rho ) ,
 \ee
 as soon as $N$ is sufficiently large.
 \end{theorem}
Similar upper bounds were derived in \cite{DKS} for collections consisting
of one large and several small skeletons.
Here we have somewhat more delicate situation, since we need to control the weights of stacks  of almost optimal
contours, which are interacting. This  requires additional tools and efforts.

%
\smallskip

\noindent
\step{3} (Refined estimates in terms of graph structure of $\Gamma$).

Let us elaborate on the upper bounds derived in \cite{DKS}. Consider the ensemble
of (single) large microscopic loops $\Gamma$ with weights $w_\sff^\beta (\Gamma )$ as in
\eqref{eq:weight-free}.  Given a (polygonal) skeleton $\gamma\subset \bbB_1$ define
\[
w^\beta_\sff \lb \Gamma \stackrel{\ell_N}{\sim} \gamma\rb :=
\sum_{\Gamma \stackrel{\ell_N}{\sim} \gamma} w_\sff^\beta (\Gamma ).
\]
More generally, given a function $F( \Gamma_1, \Gamma_2,...)$ we put
\[
 \bigotimes_i w_\beta^{\sff} \lb F( \Gamma_1, \Gamma_2,...)  \rb
 :=
 \int \bigotimes_i w_\beta^{\sff} \lb  d\Gamma_i  \rb F( \Gamma_1, \Gamma_2,...).
\]
Upper bounds
derived in \cite{DKS} imply that here exists a positive non-decreasing
function $\alpha_0$ on $(0,\infty)$ such that the following happens: Fix $a_0 >0$.
Given a closed polygon $\gamma$, 
define
its excess surface tension
\be{eq:Omega-diff}
\Omega_\beta (\gamma ) = \taub (\gamma ) - \taub \lb \sfa (\gamma )\rb .
\ee
Then, for all $N$ and $\beta$ sufficiently large,
\be{eq:DKS-one}
\frac{1}{N}\log w^\beta_\sff \lb \Gamma \stackrel{\ell_N}{\sim} \gamma\rb
< -\lb  \taub \lb \sfa ( \gamma ) \rb +  \alpha_0 \lb \Omega_\beta (\gamma )\rb\rb
\lb 1 -  \smo{N} (1) \rb
\ee
uniformly in  $\sfa (\gamma ) > a_0$ . This estimate,
is explained in the beginning of Subsection~\ref{sub:DKS}.

{
Should we be able to bound  $\bbQ_{\beta , N}
\lb \Gamma\stackrel{\ell_N}{\sim} \calS\rb $
by  product weights
\[
 \bigotimes w_\beta^{\sff} \lb \Gamma\stackrel{\ell_N}{\sim} \calS\rb
 := \prod_{\gamma_{i}\in\mathcal{S}}
 w^\beta_\sff \lb \Gamma_i \stackrel{\ell_N}{\sim} \gamma_i \rb ,
 \]
then \eqref{eq:DKS-one} and   Proposition~\ref{prop:lb-astar}  would readily
imply \eqref{eq:UB-L-net}.
However, due to cluster sharing in \eqref{eq:Plarge} and due to confined geometry of
clusters;
$ \calC \subset B_N$, contours in $\Gamma$ do interact both between
each other and with
$\partial B_N$, which a priori may lead to a modification of surface tension.
Therefore, one should proceed with care.
}

A compatible collection $\Gamma = \lbr \Gamma_v\rbr_{v\in \calV}$ of large level lines has a natural
graph structure: Namely let us say that $\Gamma_u \sim \Gamma_v$ if there is a continuous path
in $\bbR^2$ which connects between $\Gamma_u$ and $\Gamma_v$ without hitting any other element of
$\Gamma$. This  {notion of hitting}
is ambiguous because by construction different $\Gamma_v$-s may share
bounds or even coincide. We resolve this ambiguity as follows: If there is a strict inclusion
$\mathring{\Gamma}_u \subset \mathring{\Gamma}_v$, then any path from the infinite component
of $\bbR^2\setminus \Gamma_v$
to $\Gamma_u$   by definition hits $\Gamma_v$. If $\Gamma_{v_1}= \dots = \Gamma_{v_k}$,
then we fix an ordering and declare that any path from the infinite
component of $\bbR^2\setminus \Gamma_{v_1}$ to
$\Gamma_{v_j}; \, j>1, $ hits $\Gamma_{v_i}$ for any $i < j$.

In this way we label collections $\Gamma$ of large level lines by
finite graphs $\calG = \lb\calV , \calE\rb $.
If $\calS = \lbr \gamma_u\rbr$
is a family of $\ell_N$-skeletons of $\Gamma$ (meaning that
$\Gamma_u \stackrel{\ell_N}{\sim} \gamma_u$ for every $u\in \calV$, then by
definition $\calS$ has the same graph structure.

 We write $\Gamma \in \calG$ if $\calG$ is the above graph of $\Gamma$.
The chromatic number of this graph plays a role. In
Subsection~\ref{sub:chromatic-bound} we show how, once the chromatic number is
under control,
to reduce
complex many-body interactions in
$\bbQ_{\beta , N}
\lb \Gamma\stackrel{\ell_N}{\sim}\calS\rb $
 to upper bounds on pairs of interacting contours.
\smallskip

\noindent
\step{4} (Entropic repulsion versus interaction). In Subsection~\ref{sub:two}
we
formulate  decoupling upper bounds for two interacting contours.
In view of these bounds \eqref{eq:clusterweight} implies
 that
at sufficiently low temperatures entropic repulsion always
beats 
{our weak}
interactions. The proof  is relegated
to Section~\ref{sec:Two}.
\smallskip

\noindent
\step{5} (Bounds on chromatic numbers).
 In Subsection~\ref{sub:kappa}
 we derive exponential upper bounds on chromatic numbers, which enable
 reduction to  the decoupling estimates for pairs of contours.
\smallskip

\subsection{A chromatic number upper bound on a collection of large contours.}
\label{sub:chromatic-bound}
Let $\calG = \lb\calV , \calE\rb $ be a finite graph, for instance
associated to a collection $\Gamma = \lbr \Gamma_v\rbr_{v\in\calV} $
of large level lines.  Let $\calS = \lbr  \gamma_v\rbr_{v\in\calV} $ be a
collections of polygonal lines.
We wish to derive an upper bound on $\bbQ_{N ,\beta}$-probabilities
(see \eqref{eq:Plarge})
\be{eq:Graph-sets}
\bbQ_{N, \beta}\lb \1_{\lbr \Gamma\in \calG\rbr} \prod_{v\in \calV}
\1_{\lbr \Gamma_v\stackrel{\ell_N}{\sim} \gamma_v \rbr}\rb \eqvs
\sum_{\Gamma}
{\rm e}^{-\beta \sum_v \abs{\Gamma_v}
- \sum_{\mathcal{C}\not \sim \Gamma }^*
\Phi_{\beta}(\mathcal{C}) }
\1_{\lbr \Gamma\in \calG\rbr} \prod_{v\in \calV}
\1_{\lbr \Gamma_v\stackrel{\ell_N}{\sim} \gamma_v \rbr}.
\ee
Above $\sum^*$ restricts summation to connected clusters $\mathcal{C}\subset B_N$. Since we are
trying
to derive upper bounds in terms of surface tension $\taub$ which was introduced in
\eqref{eq:shape-lim}
in terms of infinite volume weights, it happens to be convenient to
augment $\calV$ with an additional
root vertex $0$ which corresponds to $\Gamma_0 = \partial B_N$, and connect it
to other vertices of $\calV$ using exactly the
same
rules as specified above (under the convention that if $\Gamma$ contains other copies
of $\partial B_N$, then
$\Gamma_0$ is the external one in the ordering of these copies).
Let $\calG_0 = \lb \calV_0 , \calE_0\rb$ to
be the
augmented graph,
and let $\hat\calG_0$ to be the line graph of $\calG_0$. That is vertices of $\hat\calG_0$ are
undirected
edges $e = (u,v )$ of $\calG_0$, and we say that $e$ and $g$ are neighbors in $\hat\calG_0$ if they
are adjacent in $\calG_0$. Let $\kappa_{\calG}$ be the chromatic number of $\hat\calG_0$, and consider a
disjoint decomposition $\hat\calG_0 = \cup_{i=1}^{\kappa_{\calG}}\hat\calG_i$.  By definition each class
$\hat\calG_i$ contains pair-wise non-adjacent edges.

Now, if $\Gamma\in \calG$, then,
\be{eq:Phi-weight-bound}
\sumtwo{\mathcal{C}\not \sim \Gamma }{\mathcal{C}\subset B_N}
\Phi_{\beta}(\mathcal{C})
\geq \sum_{v\in \calV} \Phi_{\beta}(\mathcal{C})
\1_{\lbr \mathcal{C}\not \sim \Gamma_v \rbr}
 - \sum_{e\in \calE_0}\abs{\Phi_{\beta}(\mathcal{C})}
 \1_{\lbr \mathcal{C}\not \sim e\rbr } ,
\ee
where  we write $\1_{\lbr \mathcal{C}\not \sim e\rbr } = \1_{\lbr \mathcal{C}\not \sim \Gamma_u \rbr}
\1_{\lbr \mathcal{C}\not \sim \Gamma_v \rbr}$ for an undirected edge  $e = (u, v) \in \calE_0$.

We arrive to the following upper bound in terms of product free weights defined in \eqref{eq:weight-free}:
\be{eq:UB-product}
\bbQ_{N, \beta}\lb \1_{\lbr \Gamma\in  \calG\rbr} \prod_{v\in \calV}
\1_{\lbr \Gamma_v\stackrel{\ell_N}{\sim} \gamma_v  \rbr}\rb \leqs
\bigotimes_{v\in \calV} w_\beta^{\sff} \lb
1_{\lbr \Gamma\in  \calG\rbr} \prod_{v\in \calV}
\1_{\lbr \Gamma_v\stackrel{\ell_N}{\sim} \gamma_v  \rbr}
{\rm exp}\lbr \sum_{i=1}^{\kappa_\calG}
\sumtwo{e\in \hat \calG_i}{\calC\not\sim e}  \abs{\Phi_{\beta}(\mathcal{C})}\rbr
\rb .
\ee
By the (generalized) H\"{o}lder inequality,
\be{eq:Gen-Holder}
\begin{split}
&\log \bbE_{N, \beta}\lb \1_{\lbr \Gamma\in  \calG\rbr}
\prod_{v\in \calV} \1_{\lbr\Gamma_v\stackrel{\ell_N}{\sim} \gamma_v  \rbr}\rb
\\
&\quad \leqs \frac{1}{\kappa_\calG} \sum_{i=1}^{\kappa_\calG}
\log \lb \bigotimes_{v\in \calV} w_\beta^{\sff}
\lb
1_{\lbr \Gamma\in  \calG\rbr} \prod_{v\in \calV}
\1_{\lbr \Gamma_v\stackrel{\ell_N}{\sim} \gamma_v  \rbr}
{\rm exp}\lbr \kappa_\calG \sumtwo{e\in \hat \calG_i}{\calC\not\sim e}
\abs{\Phi_{\beta}(\mathcal{C})}\rbr
\rb
\rb
.
\end{split}
\ee
For each $i = 1, \dots , \kappa_\calG$ we can relax constraints and write
\be{eq:relax-pairs}
1_{\lbr \Gamma\in  \calG\rbr} \leq
\prod_{e = (u, v )\in \hat \calG_i} \1_{\lbr \Gamma_u \sim\Gamma_v\rbr} .
\ee
Above $\Gamma_u \sim\Gamma_v$ just means that $\Gamma_u$ and $\Gamma_v$ are
two compatible large level lines.

Let us say that $u\not\in \hat \calG_i$  if no edge of
{$\hat \calG_i$}  contains $u$ as a vertex.
Each of the $i =1, \dots , \kappa_\calG$ summands on the right hand side of \eqref{eq:Gen-Holder}
is bounded above (under notation convention $w_\beta^{\sff} (\Gamma_0 )=1$ for the auxiliary
vertex $\Gamma_0 =\partial B_N$)
by
\be{eq:term-i}
\sum_{u\not\in \hat \calG_i} \log w_\beta^{\sff} (\Gamma_u\stackrel{\ell_N}{\sim} \gamma_u )
+ \sum_{(u,v) \in \hat\calG_i }
\log
w_\beta^{\sff}\bigotimes w_\beta^{\sff}
\lb \1_{\lbr \Gamma_u \sim \Gamma_v \rbr} \1_{\lbr \Gamma_v\stackrel{\ell_N}{\sim} \gamma_v  \rbr}
\1_{\lbr \Gamma_u\stackrel{\ell_N}{\sim} \gamma_u  \rbr}
{\rm e}^{\kappa_\calG \sum_{\calC\not\sim e} \abs{\Phi_{\beta}(\mathcal{C})}}
\rb
.
\ee

In order to apply  \eqref{eq:term-i} we need, first of all, to control the chromatic number $\kappa_\calG$.
After that we shall be left with studying only the case of two compatible contours, and  we shall
need to show that in the latter case at all sufficiently  low temperatures entropic repulsion which is triggered
by compatibility constraint $\Gamma_u \sim \Gamma_v$ wins over the attractive  potential
$\kappa_\calG \sum_{\calC\not\sim e} \abs{\Phi_{\beta}(\mathcal{C})}$.

\subsection{Entropic repulsion versus interaction. }
\label{sub:two}
In this subsection we formulate  upper bounds on probabilities related to  two compatible
interacting large contours.

{
\begin{theorem}
 \label{thm:entropic}
 Assume that a number $\chi > 1/2$ and a sequence $\lbr \kappa_\beta\rbr$ are
 such that
 \begin{equation}
\label{eq:chi-cond}\limsup_{\beta\to\infty}
\sup_{\mathcal{C}\neq\emptyset}
\kappa_\beta \text{e}^{
{
\chi \beta \lb {\rm diam}_\infty \lb
\mathcal{C}\rb +1\rb
}
}
\left\vert \Phi_{\beta}(\mathcal{C}
)\right\vert < 1 .
\end{equation}
Fix $a_0>0$. Recall the definition of excess surface tension $\Omega_\beta$ in
\eqref{eq:Omega-diff}.
Then,
\be{eq:entropic}
\begin{split}
&\frac{1}{N}
\log\lb
w_\beta^{\sff}\bigotimes w_\beta^{\sff}
\lb
\1_{\lbr \Gamma_u \sim \Gamma_v \rbr}
\1_{\lbr \Gamma_v\stackrel{\ell_N}{\sim} \gamma_v \rbr}
\1_{\lbr \Gamma_u\stackrel{\ell_N}{\sim} \gamma_u \rbr}
{\rm e}^{\kappa_\beta \sum_{\calC\not\sim e} \abs{\Phi_{\beta}(\mathcal{C})}}
\rb
\rb
\\
&\qquad
\leq -[ \taub (\sfa (\gamma_u ) ) + \taub (\sfa (\gamma_v)  ) +
\alpha_0 \lb  \Omega_\beta
(\gamma_u )\rb  +
 \alpha_0 \lb \Omega_\beta  ( \gamma_v )\rb
]
\lb 1- \smo{N} (1)\rb,
\end{split}
\ee
uniformly in $\beta$ and $N$ sufficiently large and uniformly in closed polygonal
lines $\gamma_v , \gamma_u$ satisfying $\sfa (\gamma_v ) , \sfa (\gamma_u ) \geq a_0$.
The function  $\alpha_0$ is the  function  appearing in \eqref{eq:DKS-one}.
\end{theorem}
We sketch the  proof of Theorem~\ref{thm:entropic} in the concluding Section~\ref{sec:Two}.
}

\subsection{Upper bound on $\kappa_\calG$.}
\label{sub:kappa}
In this  Subsection we shall show that
for all $\beta$ sufficiently large one can restrict attention
to graphs $\calG$ satisfying $\kappa_\calG \leq \kappa_\beta$ where the
sequence $\lbr \kappa_\beta\rbr$ complies with \eqref{eq:chi-cond}.

We start with a simple general combinatorial observation.

Let $G_{N}$ be a graph with no loops and double edges, having $N$ vertices.
Its edge coloring is called proper, if at every vertex all the bonds entering
it have different colors. The minimal number of colors needed for creating a
proper edge coloring will be denoted by $\varkappa\left(  G_{N}\right)  .$ It
is called the \textit{edge chromatic number. }We need the upper bound on
$\varkappa\left(  G_{N}\right)  .$ Evidently, the complete graph with $N$
vertices has the highest edge chromatic number, so it is sufficient to
consider only the case of $G_{N}$ being a complete graph.

\begin{theorem}
\label{coloring simplex} Let $G_{N}$ be a complete graph with $N$ vertices.
Then
\[
\varkappa\left(  G_{N}\right)  =\left\{
\begin{array}
[c]{cc}%
N\ \ \ \ \  & \text{if }N\text{ is odd,}\\
N-1 & \text{if }N\text{ is even.}%
\end{array}
\right.
\]

\end{theorem}

\begin{proof}
Let us produce a proper edge coloring of $G_{N}$ by $N$ colors, for $N$ odd.
To do this, let us draw $G_{N}$ on the plane as a regular $N$-gon $P_{N}$,
with all its diagonals. Let us color all the $N$ sides $s$ of $P_{N}$ using
all the $N$ colors. What remains now is to color all diagonals. Note, that
every diagonal $d$ of $P_{N}$ is parallel to precisely one side $s_{d}$ of
$P_{N},$ because $N$ is odd. Let us color the diagonal $d$ by the color of the
side $e_{d}.$ Evidently, the resulting coloring $\mathcal{C}_{N}$ is proper.

Every vertex of $G_{N}$ has $N-1$ incoming edges, while we have used $N$
colors to color all the edges. Therefore at each vertex $v$ exactly one color
$c_{v}$ is missing -- it is the color of the edge opposite $v.$ By
construction, all $N$ missing colors $c_{v}$ are different. Let us use this
property to construct a proper edge coloring of $G_{N+1}$ by $N$ colors.
First, we color $G_{N}\subset G_{N+1}$ by the coloring $\mathcal{C}_{N}.$ Let
$w$ be the extra vertex, $w=G_{N+1}\setminus G_{N}.$ Let us color the bond
$\left(  v,w\right)  $ by the color $c_{v}.$ Evidently, the resulting coloring
$\mathcal{C}_{N+1}$ is again proper.

Finally, for $N$ odd the constructed coloring $\mathcal{C}_{N}$ is best
possible: there is no coloring using $N-1$ colors. Indeed, suppose such a
coloring does exist. That means that at each vertex of $G_{N}$ all $N-1$
colors are present. Therefore, the bonds of the first color, say, define a
partition of the set of all vertices into pairs. That, however, is impossible,
since $N$ is odd. This nice last argument is due to O. Ogievetsky.
\end{proof}


{

 Next let us record  a (straightforward) consequence of \eqref{eq:BulkLL-Q} as follows:
 Fix $A$. Then,
 \be{eq:A-Gamma}
 \lim_{N\to\infty}\lb  \frac{1}{N} \log
 \lb
 \bbQ_{N , \beta}
 \lb
 \Xi_N \geq  \rho_\beta N^3 +AN^2\, \big|\, \Gamma\stackrel{\ell_N}\sim \calS
 \rb
 \rb
  + \frac{\lb \delta_\beta - \sfa\lb \calS\rb \rb^2}{2D_\beta}\rb = 0 .
 \ee
 uniformly in collections $\calS$ of  closed polygons with
 $0 \leq \sfa \lb \calS\rb \leq \delta_\beta$
 (recall \eqref{eq:quantities} for the definition of $\delta_\beta$).
}

We proceed  with the following two observations concerning the variational
problem \eqref{eq:VP-delta}:

\begin{lemma}
\label{lem:difference}
Fix $\beta$ sufficiently large and consider \eqref{eq:VP-delta}.
Given $\delta >0$ let $a^*_\beta (\delta )$ be the optimal area.
Then,
\be{eq:difference}
\limsup_{\delta\to\infty}\lb \delta - a^*_\beta (\delta )\rb :=\xi_\beta < \infty.
\ee
\end{lemma}
Furthermore,
\begin{lemma}
 \label{lem:a-bound}
 In the notation of Lemma~\ref{lem:difference},
 \be{eq:a-bound}
 \min_{\ell}\lb \taub \lb \calL_\ell^2{(a)}  \rb + \frac{(\delta - a )^2}{2D_\beta}\rb
 - \lb \taub \lb a_\beta^* \rb + \frac{(\delta - a^*_\beta )^2}{2D_\beta}\rb
 \geq \frac{(a-a^*_\beta )^2}{2D_\beta} ,
 \ee
 for all $\delta$ and all $a\in [0, \delta]$
\end{lemma}
\begin{proof}[Proof of Lemma~\ref{lem:difference}]
 We may consider only sufficiently large values of $\delta$, such that
 solutions to \eqref{eq:VP-delta} are given by optimal stacks of type 2.
 By the second of \eqref{eq:st-L2},
 \be{eq:rad-eq}
 \frac{\delta -a^*_\beta (\delta )}{D_\beta} = \frac{\taub (\sfe )
 }{
 r^{2, *} ( a^*_\beta (\delta ))} ,
 \ee
 where $r^{2, *} (a )$ is the radius of the optimal stack of type-2 at
 given area $a$. Hence, it would be enough to check that
 \be{eq:rad-star-a}
 \liminf_{a\to\infty} r^{2, *} (a ) >0 .
 \ee
 Clearly, if  $\ell <m$ and $a\in [m w ,4\ell ]$ (that is if area $a$ can be
 realized by both $\ell$ and $m$ stacks of type-2), then
 $r^{2, \ell} (a ) < r^{2, m } (a )$. Which means that
 the map $a\mapsto r^{2, *} (a )$ has the following structure: There is a
 sequence $w= \hat{a}_1 < \hat{a}_2 < \dots $ of (transition)  areas such
 that:

 \noindent
 (a) On each of the intervals $\lb \hat{a}_\ell , \hat{a}_{\ell +1}\rb$
 the optimal radius $r^{2, *} (a ) = r^{2 , \ell } (a)$ and it is decreasing.

 \noindent
 (b) At transition points $r^{2 , \ell } (\hat{a}_{\ell +1} ) <
 r^{2 , \ell +1} (\hat{a}_{\ell +1} )$. Hence we need to show that
 \be{eq:rad-star-a-l}
\liminf_{\ell\to\infty} r^{2 , \ell } (\hat{a}_{\ell +1} ) > 0.
 \ee
 Fix $\ell$ and define $r_\ell = r^{2 , \ell } (\hat{a}_{\ell +1} )$ and
 $\rho_\ell = r^{2 , \ell+1 } (\hat{a}_{\ell +1} )$. Then,
 \be{eq:area-rho}
 \hat{a}_{\ell +1} = \ell\lb 4- (4-w )r_\ell^2\rb  =
 (\ell +1) \lb 4- (4-w )\rho_\ell^2\rb .
 \ee
 By definition, $\tau \lb \calL^2_\ell ( \hat{a}_{\ell +1} )\rb =
 \tau \lb \calL^2_{\ell+1} ( \hat{a}_{\ell +1} )\rb$. Which reads
 (recall \eqref{eq:rad-2a} and the first of \eqref{eq:st-L2}) as
 \be{eq:tau-rho}
 \ell\lb 8 - 2r_\ell  (4- w)\rb = (\ell +1 ) \lb 8- 2\rho_\ell (4- w)\rb
 \ee
 Solving \eqref{eq:area-rho} and \eqref{eq:tau-rho} we recover
 \eqref{eq:rad-star-a}.
\end{proof}
\begin{remark}
 A slightly more careful analysis implies that under \eqref{eq:p-beta}
 there exists
 $\nu <\infty$ such that for all sufficiently large values of $\beta$,
 \be{eq:difference-nu}
\sup_{\delta}\lb \delta - a^*_\beta (\delta )\rb \leq {\rm e}^{\beta\nu} .
\ee
\end{remark}
\begin{proof}[Proof of Lemma~\ref{lem:a-bound}]
 By the second of \eqref{eq:st-L2} and then by \eqref{eq:rad-2a} the
 function $a\mapsto \taub \lb \calL_\ell^2 (a )\rb$ is convex for any
 $\ell\in\bbN$. Hence,
 \[
  \frac{\dd^2}{\dd a^2}\lb \taub \lb \calL_\ell^2 (a )\rb  +
  \frac{( \delta -a )^2}{2 D_\beta }\rb \geq \frac{1}{D_\beta}
 \]
uniformly in $a\in (\ell w , \delta\wedge 4 \ell )$. The same applies
at generic values $a\in (w, \delta )$ for  the function
\[
 \min_{\ell} \lb \taub \lb \calL_\ell^2 (a )\rb  +
  \frac{( \delta -a )^2}{2 D_\beta }\rb .
\]
But the latter attains its minimum at $a_\beta^* (\delta )$.
Hence the conclusion.
\end{proof}
{
Consider now a collection $\calS = \lbr \gamma_v\rbr_{v\in\calV}$ of closed polygonal
lines.
Given a number $\zeta <1$, such that $2\zeta w > \sfa\lb \bbB_1\rb= 4$, let us
split $\calS$  into a disjoint union,
\be{eq:L-decomp}
 \calS = \calS_0\cup\calS_1\cup\dots ,
\ee
where $\calS_0$ contains all the polygons
$\gamma$ of $\calS$ with area $\sfa (\gamma )\geq \zeta w$, whereas, for
$i=1, 2, \dots$
\be{eq:Li}
\calS_i = \lbr\gamma\in\calS \, : \, \sfa (\gamma ) \in \left[ \zeta^{i+1}w, \zeta^{i}w \right)\rbr  .
\ee
}
{
By construction, any compatible collection $\Gamma$ of large level lines, such that
$\Gamma\stackrel{\ell_N}{\sim}\calS_0$ is always an ordered stack. Given
numbers $d >0, m\in \bbN$ and a value $\beta$ of the inverse
temperature, let us say $\calS$ is  bad; $\calS\in \frB_{d , m} (\beta )$, if
either there exists $i  > d\beta$ such
that $\calS_i$ is not empty, or there exists $1\leq i \leq  d\beta$ such that
the cardinality
\be{eq:cardLi}
\abs{\calS_i} = \#\lbr \gamma\, :\, \gamma  \in \calS_i\rbr \geq m .
\ee
Alternatively, we may think in terms of graphs $\calG = \lb \calV , \calE\rb$ associated
to  bad
collections $\calS = \lbr\gamma_v\rbr_{v\in \calV}$.
}
\begin{proposition}
 \label{thm:bad-stacks}
{
 There exist $d<\infty$ and $m\in \bbN$  such that for all sufficiently large values of $\beta$
 the following holds:
 \be{eq:bad-stacks}
 \limsup_{N\to\infty} \frac{1}{N}
 \max_{\calS  \in \frB_{d , m} (\beta )}
 \log \bbQ_{N, \beta}^A
 \lb \1_{\lbr \Gamma\in  \calG\rbr} \prod_{v\in \calV}
 \1_{\lbr \Gamma_v\stackrel{\ell_N}{\sim} \calS\rbr}\rb
 < 0,
 \ee
for any excess area $A\geq 0$.
}

{
 In particular for all sufficiently large $\beta$ we may restrict
 attention to graphs $\calG$ with
 chromatic number $\kappa_\calG \leq \beta d m +2$,
 independently of the value of excess area $A$
 in $\bbQ_{N, \beta}^A$.
 }
\end{proposition}

\begin{proof}[Proof of Proposition~\ref{thm:bad-stacks}]
{
Note that the estimate $\kappa_\calG \leq \beta d m +2$ on the chromatic number is
obtained from \eqref{eq:bad-stacks} as follows.  For good skeleton collections  the union
$\mathcal{S}_{1}\cup\mathcal{S}_{2}\cup...$
contains at most $\beta d m $ loops, while the collection  of large
level lines $\Gamma^{(0)} \stackrel{\ell_N}{\sim} \mathcal{S}_{0}$ is an
ordered stack. Therefore we have to estimate the edge chromatic number of a graph
with at most $\beta d m +2$ vertices, which we do by using our
combinatorial theorem above.
}

{
In view of the lower bound \eqref{eq:lb-astar} it would be enough
to prove the following:
There exists $c = c_\beta >0$ such that
If
$\calS\in \frB_{d , m} (\beta )$
\be{eq:Bad-bound}
\limsup_{N\to\infty} \max_{\calS\in \frB_{d , m} (\beta )}
\lb
\frac{1}{N}
\log\lb \bbQ_{N, \beta }\lb \Gamma\stackrel{\ell_N}{\sim}
\calS \rb \rb  + \taub \lb \sfa \lb \calS\rb\rb \rb
< - c_\beta
\ee
Let $\calS = \lbr \gamma_u\rbr\in \frB_{d , m} (\beta )$ be a bad collection
of skeletons.
 Consider its decomposition \eqref{eq:L-decomp}.  For any
 $k\geq 0$ set $\calS^{(k)}_l = \calS_0\cup\dots \cup\calS_k$
 and $\calS_s^{(k)} = \calS_{k+1}\cup\dots $.
 We shall prove \eqref{eq:Bad-bound} by a gradual reduction procedure using the
 following identity: Given
 $k\geq 0$,  the decomposition $\calS = \calS^{(k)}_l \cup \calS^{(k)}_s$ induces
 the decomposition
 $\Gamma= \Gamma^{(k)}_l\cup \Gamma^{(k)}_s$ of any collection $\Gamma\stackrel{\ell_N}{\sim}\calS$
 of
 large level lines.
 Then,
 \be{eq:Cond-contour}
 \begin{split}
 &\bbQ_{N , \beta}\lb
 \Xi \geq \rho_\beta N^3 + AN^2\, ;\,
 \Gamma \stackrel{\ell_N}{\sim}\calS\rb \\
 &=
 \sum_{\Gamma\stackrel{\ell_N}{\sim}\calS}
 \bbQ_{N, \beta}\lb
 \Xi \geq \rho_\beta N^3 + AN^2\,  \ \Big|\, \Gamma^{(k)}_l\cup \Gamma^{(k)}_s \rb
 \bbQ_{\beta , N} \lb \Gamma^{(k)}_s \big| \Gamma^{(k)}_l\rb
 \bbQ_{\beta , N} \lb \Gamma^{(k)}_l\rb.
 \end{split}
 \ee
 The conditional probability
 $\bbQ_{N , \beta}\lb\, \cdot\, \big|\, \Gamma\rb$ is a straightforward modification
 of \eqref{eq:Plarge}: Given a splitting $\Gamma\cup\Gamma^\prime$ of a compatible family of
 large contours, or, alternatively a splitting $\calV\cup\calV^\prime$ of the
 vertices of the associated
 graph,
 \begin{equation}
\label{eq:Plarge-cond}
\mathbb{Q}_{N, \beta} \left(  \Gamma^\prime \, \big|\, \Gamma\right) \,
{
\eqvs
}
\,
\mathrm{exp}%
\left\{  -\beta\sum_{v\in \calV^\prime}\left\vert \Gamma_{v}\right\vert -
\sumtwo{ \mathcal{C}\subset\Lambda_{N}}{\mathcal{C} \sim \Gamma}
\1_{\lbr \mathcal{C}\not \sim \Gamma^\prime \rbr}
 \Phi_{\beta}(\mathcal{C})\right\}  .
\end{equation}
We shall rely on the following upper
bounds on conditional weight which holds for all $\beta$ and $N$
sufficiently large:
By
\eqref{eq:clusterweight} and \eqref{eq:twopoint-st},
\be{eq:Interm-bound1}
\frac{1}{N} \log
\lb
\sum_{\Gamma^\prime \stackrel{\ell_N}{\sim} \calS^\prime }
\bbQ_{N, \beta}
\lb
 \Gamma^\prime \big| \Gamma
 \rb
 \rb
 \leq
 - \lb
 1 - \bigof{{\rm e}^{-4\beta}}
 \rb
 \taub \lb \calS^\prime  \rb .
\ee
for any $\Gamma$ and any collection
$\calS^\prime$  of
closed polygonal lines.
}


Let us now turn to proving  {\eqref{eq:Bad-bound}}
(and consequently \eqref{eq:bad-stacks})
proper.

\noindent
\step{1}.
{
 Let us  explain how we choose $m$ in $\frB_{d, m}$.
 We can fix (independently of $\beta$
two numbers $r >1 $ and $c<\infty$ such that
\be{eq:rc-bound}
\taub\lb \calS^{(k)}_s\rb \geq r \taub\lb \sfa\lb \calS^{(k)}_s\rb\rb,
\ee
whenever  $\calS$ and a number $k\geq 0$ are
such that
\be{eq:areas}
\sfa\lb \calS^{(k)}_s\rb \geq c\zeta^k.
\ee
Indeed, by construction, the
areas of loops from   $\calS^{(k)}_s$ are bounded above by $\zeta^{k+1}$. Hence,
\[
 \taub \lb \calS^{(k)}_s \rb \geq 2 \taub (\sfe )
 \left\lfloor \frac{a}{\zeta^{k+1}} \right\rfloor
 \sqrt{ w_\beta  \zeta^{k+1}}
\]
whenever $\sfa\lb \calS^{(k)}_s \rb = a$. This should be compared with $\taub (a )$
which equals to $2 \taub (\sfe )\sqrt{a w_\beta } $ if $a\leq w_\beta $ and,
otherwise, bounded above
by $2 \taub (\sfe )\lb  \left\lfloor \frac{a}{w_\beta }  \right\rfloor +1\rb w_\beta $.
}

Note that \eqref{eq:areas} will hold if
$\abs{\calS_{k+1}} \geq \frac{c}{\zeta w}$.
We set $m = \frac{c}{\zeta w}$.
\smallskip

\noindent
{
\step{2}. The first consequence of \eqref{eq:rc-bound} is that we can rule out
collections $\calS$ with
$\sfa \lb \calS^{(0)}_s\rb \geq c$. Indeed, let $\calS$ be such a collection
of polygonal lines. By construction $\calS_0$ is compatible with  an ordered
stack of large contours, hence its graph is just a line segment.
Hence, the decoupling
bound \eqref{eq:term-i} and Theorem~\ref{thm:entropic} imply:
\be{eq:big-stack}
 \frac{1}{N}\log\lb \bbQ_{N, \beta}
 \lb \Gamma \stackrel{\ell_N}{\sim}\calS_0\rb\rb
 \leq
 -
 \lb 1 - \smo{N}(1)\rb
 \sum_{\gamma\in \calS_0}\lb \taub \lb \sfa (\gamma )
 \rb
 +\alpha_0 \lb \Omega_\beta (\gamma )\rb
 \rb
\ee
On the other hand, \eqref{eq:Interm-bound1} and  \eqref{eq:rc-bound} imply that
\be{eq:cond-not}
 \frac{1}{N}\log\lb
 \bbQ_{N, \beta}
 \lb \Gamma^{(0)}_s
 \stackrel{\ell_N}{\sim}
 \calS^{(0)}_s\, \big| \, \Gamma
 \rb
 \rb
 \leq -r
 \lb 1 - \bigof{{\rm e}^{-4\beta}}\rb
 \taub \lb \sfa\lb \calS^{(0)}_s\rb\rb.
\ee
for all compatible collections of large contours
$\Gamma \stackrel{\ell_N}{\sim} \calS_0$.
The last expression is strictly smaller than
{$ - \taub \lb \sfa\lb \calS^{(0)}_s\rb\rb$}
for all $\beta$ satisfying
\[
r \lb 1 - \bigof{{\rm e}^{-4\beta}}\rb >1 .
\]
Hence \eqref{eq:Bad-bound} .
}

{
There are two implications of the above computation
which hold for all
$\beta$ sufficiently large.
First of all pick $\nu_{1} > \nu$ (see \eqref{eq:difference-nu}).
Then
for any $A$ we can restrict attention
to skeleton collections $\mathcal{S}$ satisfying
\be{eq:calLnot-bound}
\sfa \lb \calS_0\rb \geq \delta_\beta - {\rm e}^{\nu_1\beta}.
\ee
Indeed, recall $\beta$-independent constant $c$ which was
defined via \eqref{eq:rc-bound}. In view of  Proposition~\ref{thm:bad-stacks} we
may restrict attention to
$\sfa\left(  \mathcal{S}_{s}^{\left(  0\right)  }\right) \leq c$. Which means that
if \eqref{eq:calLnot-bound}is violated, then
\[
 \sfa \lb \calS\rb \leq \sfa_\beta^* (\delta_\beta )   - {\rm e}^{\nu\beta} +c
\]
for all $\beta$ sufficiently large.
On the other hand, as we have already mentioned, $\mathcal{S}_{0}$ has
graph structure of an ordered stack or, in other words, one-dimensional segment,
 and \eqref{eq:big-stack} holds. Therefore, \eqref{eq:calLnot-bound} is secured by
Lemma~\ref{lem:a-bound} and \eqref{eq:difference-nu}
(and, of course, lower bound \eqref{eq:lb-astar}).
}

{
Next, assume  \eqref{eq:calLnot-bound}. Then, in view of the
upper bound on conditional
weights \eqref{eq:Interm-bound1},
and proceeding as in derivation of \eqref{eq:IntermBound}, one
can fix $d< \infty$ and
rule out
 skeletons  $\gamma$ with $\sfa ( \gamma  ) <\xi^{\beta d}w$.
}

 As a result we need to consider only bad collections
 $\calS$ which satisfy $\calS_i = \emptyset$ for any
 $i >\beta d$, but which still violate \eqref{eq:cardLi}.

 \noindent
 {
 \step{3}. We proceed by induction. Assume that $\calS$ is such that
 the cardinalities
 $\abs{\calS_1}, \dots , \abs{\calS_k} \leq m$.
 Then $\calS$ may be ignored if
 $\abs{\calS_{k+1}} >m$. Indeed, the latter would imply that
 $\sfa\lb \calS^{(k)}_s\rb \geq c\zeta^k$, and hence  \eqref{eq:rc-bound} holds.
 Consequently, as in the case of \eqref{eq:cond-not}
\be{eq:cond-k}
 \limsup_{N\to\infty}
 \frac{1}{N}\log\lb
 \bbQ_{N, \beta} \lb \Gamma^{(k)}_s \stackrel{\ell_N}{\sim}
 \calS^{(k)}_s\, \big| \, \Gamma\rb
 \rb
 \leq -r
 \lb 1 - \bigof{{\rm e}^{-4\beta}}\rb
 \taub \lb \sfa\lb \calS^{(k)}_s\rb\rb.
\ee
for all collection  $\Gamma \stackrel{\ell_N}{\sim} \calS_l^{(k)}$.
}

{
On the other hand, by induction assumption, the
chromatic number of $\calL^{(k)}_l$ is bounded above by
 {$km +2 \leq \beta d m  +2 := \kappa_\beta$.}
Hence, Theorem~\ref{thm:entropic} applies, and in view of
the decoupling
bound \eqref{eq:term-i}, we infer:
\be{eq:big-stack-k}
 \frac{1}{N}\log\lb \bbQ_{N, \beta}
 \lb \Gamma \stackrel{\ell_N}{\sim}\calS_l^{(k)}\rb\rb
 \leq
 -
 \lb 1 - \smo{N}(1)\rb
 \sum_{\gamma\in \calS_l^{(k)}}\lb \taub \lb \sfa (\gamma )
 \rb
 +\alpha_0 \lb \Omega_\beta (\gamma )\rb
 \rb
\ee
Together with \eqref{eq:cond-k} this implies \eqref{eq:Bad-bound}.
}
\end{proof}

\subsection{Proof of Theorem~\ref{thm:UB-L-net}}
{
Let us complete the proof of Theorem~\ref{thm:UB-L-net} and hence of
 Theorem~\ref{thm:B}. As we have seen in the previous subsection for
 each $\beta$ sufficiently large
 we may ignore skeleton collections
 with chromatic numbers exceeding
 {$\kappa_\beta =\beta d m  +2$.}
 If, however, the skeleton $\calS$ is {\em good}, that is
 if the chromatic number of $\calS$ is less or equal to $\kappa_\beta$,
 then \eqref{eq:chi-cond} is satisfied,
 and, in view of \eqref{eq:Gen-Holder}, \eqref{eq:term-i}
 and \eqref{eq:entropic} we conclude
 that
 \be{eq:good-L-bound}
 \frac{1}{N} \log \lb \bbQ_{N, \beta}
 \lb \Gamma\stackrel{\ell_N}{\sim} \calS\rb\rb  \leq -
\lb
 \sum_{\gamma_u\in \calS}
 \lb
 \taub
 \lb \sfa (\gamma_u )
 \rb
 + \alpha_0 \lb  \Omega_\beta (\gamma_u )\rb
 \rb
 \rb
 (1- \smo{N} (1)) .
 \ee
 In \eqref{eq:good-L-bound} there is the same
 correction term $\smo{N}(1) \to 0$ for {\em all} good skeletons.
 Since,
 \[
 \calE_\beta \lb \calS~|~ \delta_\beta \rb
 = \sum_{\gamma_u \in \calS} \Omega_\beta (\gamma_u )
 +
 \lb
\sum_{\gamma_u\in\calL} \taub (\sfa (\gamma_u ) ) +
\frac{(\delta_\beta - \sfa (\calL ))^2}{2 D_\beta }
\rb ,
\]
it remains to apply the quantitative isoperimetric bound \eqref{eq:Iso-Stack-skel}.
}

\section{Two interacting contours}
\label{sec:Two}

In this concluding section we sketch the proof of Theorem~\ref{thm:entropic}.
The proof relies on  the skeleton calculus  developed in
\cite{DKS},  Ornstein-Zernike theory and random walk representation
of polymer models \cite{IV08}, which, in the particular case of Ising
polymers, was refined and
adjusted in \cite{IST15}.  We shall repeatedly refer to these
papers for missing
details.

Throughout this section we shall assume that the constants $\chi >1/2$ and
$\kappa_\beta$  are fixed
 so that \eqref{eq:chi-cond} holds.

\subsection{Low temperature skeleton calculus and modified surface tension.}
\label{sub:DKS}
We need to recall some ideas and techniques introduced in \cite{DKS}.

Consider an $\ell_N$-skeleton
$\gamma = \lb\sfu_0, \sfu_1, \dots , \sfu_n\rb$. It has $n+1$ edges
\[
 e_0 = (\sfu_0 , \sfu_1), \dots , e_n = (\sfu_n , \sfu_0 ).
\]
The last edge $e_n$ might have a shorter length than $\ell_N$, but for
the sake of the exposition we shall ignore the corresponding negligible
corrections. The edges of $\gamma$ are classified into being {\em good} or
{\em bad} as follows: Fix once and for all some small angle $\theta >0$
(note that the value of $\theta$ and hence the classification of edges we
are going to explain does not depend on $\beta$). With each edge
$e = (\sfu , \sfv ) $ we associate a diamond shape $\sfD_\theta (e )$,
\[
 \sfD (e ) = \sfD_\theta (\sfu , \sfv )= \lb \sfu + \calY_{\sfv - \sfu}\rb \cap
 \lb \sfv + \calY_{\sfu - \sfv}\rb ,
\]
where for any $\sfx\in\bbR^2\setminus 0$ we use $\calY_\sfx$ to denote
the symmetric cone of opening $\theta$ along the ray passing through $\sfx$.

An edge $e_i $ of $\gamma$ is called good if
\be{eq:good-edge}
\sfD (e_i )\cap \sfD (e_j) = \emptyset\quad \text{for any $j\neq i$}.
\ee
Otherwise the edge is called {\em bad}. We use $\frg (\gamma )$ and
$\frb (\gamma )$ for, respectively, the sets of good and
bad edges of $\gamma$.

{
Let $\Gamma$ be a large contour, and $\gamma$ be its skeleton, having the set $\frb (\gamma )$
of bad bonds. We denote by $\Gamma_{\frb}$ the portion of $\Gamma$, corresponding to bonds in $\frb (\gamma )$.
}

The modified surface tension $\hat \tau_{\beta} (\gamma )$ is defined as follows:
\be{eq:taub-hat}
\hat \tau_{\beta} (\gamma ) = \taub (\gamma ) - \sum_{e\in \frb (\gamma )}
\psi_\beta (e ) .
\ee
In its turn the function $\psi_\beta \ge 0$ is defined via functions
(compare with \eqref{eq:twopoint-st})
\begin{equation}
\label{eq:twopoint-st-hat}
G_{\beta}^*(x):= \sum_{\gamma: 0\to  x} w_{\beta
}^\sff (\eta)
{\rm e}^{\kappa_\beta \sum_{\calC\not\sim \eta } \abs{\Phi_\beta (\calC )}}
\quad\text{and}\quad  \tau_{\beta}^*(x) = -\lim_{n \to\infty}\frac
1{n}\log  G_{\beta}^* (\lfloor n x\rfloor),
\end{equation}
as
\be{eq:psi-beta}
\psi_\beta (x ) = \taub (x ) - 
 \tau_{\beta}^* (x).
\ee
The
skeleton calculus developed in \cite{DKS} implies the
following two crucial bounds, which hold uniformly in $\sfa (\gamma ) >a_0$:
\\
\be{eq:DKS-bound1}
\hat \tau_{\beta} (\gamma ) \geq \taub (\sfa (\gamma )) +
\alpha_0 \lb \Omega_\beta (\gamma )\rb .
\ee
 and
{
\be{eq:DKS-bound2}
 \left\vert
 \ln \frac
 { w_\beta^\sff
 \lb
 \1_{\lbr \Gamma \stackrel{\ell_N}{\sim}\gamma\rbr}
 {\rm e}^{\kappa_\beta \sum_{\calC\not\sim \Gamma_{\frb} } \abs{\Phi_\beta (\calC )}}
 \rb}
 { \prod_{(\sfu , \sfv )\in \frg (\gamma )} G_\beta (\sfv - \sfu )
 \prod_{(\sfu , \sfv )\in \frb (\gamma )}  G_\beta^* (\sfv - \sfu )}
 \right\vert
\leq {\smo{N} (1) \taub (\gamma )N}.
\ee
}
{
The estimate \eqref{eq:DKS-bound1} with the function $\alpha_0(x)=\frac{1}{2} x$ is
nothing else but the estimate (2.16.1) from the Lemma 2.16 of   \cite{DKS}. (Our $\Omega$ is
what is called $\Delta$ there, and the function $n_{\delta}$ there vanishes in the case of our interest.)
The estimate  \eqref{eq:DKS-bound2} is a very special case of the Theorem 4.16 of   \cite{DKS}, which establishes the
asymptotic independence
of the surface tension on the shape of the volume.
}

{
In view of the Ornstein-Zernike asymptotics (for instance (3.4) in \cite{IST15}) of
low temperature two-point functions $G_\beta$ in
\eqref{eq:twopoint-st} and $ G_\beta^*$ in \eqref{eq:twopoint-st-hat}
the upper bound \eqref{eq:DKS-one} {readily follows} from
 \eqref{eq:DKS-bound1} and \eqref{eq:DKS-bound2}.
 }
\subsection{Decoupling upper bound for two interacting skeletons.}
Consider two $\ell_N$-skeletons $\gamma_1$ and $\gamma_2$ as in the
formulation of Theorem~\ref{thm:entropic}. Upper bound \eqref{eq:DKS-bound2}
implies:
\be{eq:Good-bad-two}
\begin{split}
&w_\beta^{\sff}\otimes w_\beta^{\sff}
\lb
\1_{\lbr \Gamma_1 \sim \Gamma_2 \rbr}
\1_{\lbr \Gamma_1\stackrel{\ell_N}{\sim} \gamma_1 \rbr}
\1_{\lbr \Gamma_2\stackrel{\ell_N}{\sim} \gamma_2 \rbr}
{\rm e}^{\kappa_\beta \sum_{\calC\not\sim (\Gamma_1, \Gamma_2 )}
\abs{\Phi_{\beta}(\mathcal{C})}}
\rb\\
&\leq {\rm e}^{\smo{N} (1) N (\taub (\gamma_1 ) +\taub (\gamma_2 ))}
\prod_{(\sfu, \sfv )\in \frb (\gamma_1 )\cup\frb (\gamma_2 )}
\hat G_\beta (\sfv - \sfu )\times \\
&\bigotimes_{e \in \frg (\gamma_1 ) ,f \in \frg (\gamma_2 ) }
w_\beta^\sff
\lb
\prod_{e, f} \1_{\lbr \eta_e\stackrel{\ell_N}{\sim} e\rbr}
 \1_{\lbr \eta_f\stackrel{\ell_N}{\sim} f\rbr}\1_{\lbr \eta_e \sim \eta_f\rbr}
 {\rm e}^{\kappa_\beta \sum_{e, f} \sum_{\calC\not\sim (\eta_e, \eta_f)}
 \abs{\Phi_\beta (\calC )}} .
\rb
\end{split}
\ee
Above $\eta_e \stackrel{\ell_N}{\sim} e = (\sfu , \sfv )$ means that
$\eta_e$ is an admissible path from $\sfu$ to $\sfv$ which is, in addition,
compatible with the $\ell_N$-skeleton construction. As before,
$\calC\not\sim (\eta_e, \eta_f)$ means that $\calC$ is not compatible with
both $\eta_e$ and $\eta_f$.

As in the derivation of \eqref{eq:DKS-bound2} it is possible to check that
the contribution coming from $\sum_{\calC\not\sim (\eta_e, \eta_f)}
 \abs{\Phi_\beta (\calC )}$ could be ignored whenever
 $\sfD_\theta (e )\cap \sfD_\theta (f) =\emptyset$. In the latter case let us
 say that the edges $e$ and $f$ {\em are not associated}. Otherwise
 we say that $e\in \frg (\gamma_1 )$ and $f\in \frg (\gamma_2 )$ are
 {\em associated}. Since by construction
 $\sfD_\theta (f )\cap \sfD_\theta (f^\prime) =\emptyset$ for any two
 good edges $f, f^\prime \in \frg (\gamma_2 )$, any good edge
 $e\in \frg (\gamma_1 )$ can be associated to at most $m_\theta$ different
 good edges of $\frg (\gamma_2 )$. Proceeding as in the derivation
 of \eqref{eq:term-i} we conclude that we need to derive the
 upper bound on
 \be{eq:two-edges}
 w_\beta^\sff\otimes w_\beta^\sff
 \lb
 \1_{\lbr \eta_e\stackrel{\ell_N}{\sim} e\rbr}
 \1_{\lbr \eta_f\stackrel{\ell_N}{\sim} f\rbr}\1_{\lbr \eta_e \sim \eta_f\rbr}
 {\rm e}^{m_\theta \kappa_\beta  \sum_{\calC\not\sim (\eta_e, \eta_f)}
 \abs{\Phi_\beta (\calC )}} ,
 \rb
 \ee
 uniformly in $\sfD_\theta (e )\cap \sfD_\theta (f) \neq \emptyset$.

 Since $m_\theta$ does not depend on $\beta$,
 $\kappa_\beta^\prime = m_\theta \kappa_\beta$ satisfies \eqref{eq:chi-cond}
 with any $\chi^\prime < \chi$.  Therefore, Theorem~\ref{thm:entropic}
 would be a consequence of the following claim:

 \begin{proposition}
  \label{prop:two-edges-bound}
  Under the conditions of Theorem~\ref{thm:entropic}
  \be{eq:two-edges-bound}
  w_\beta^\sff\otimes w_\beta^\sff
 \lb
 \1_{\lbr \eta_e\stackrel{\ell_N}{\sim} e\rbr}
 \1_{\lbr \eta_f\stackrel{\ell_N}{\sim} f\rbr}\1_{\lbr \eta_e \sim \eta_f\rbr}
 {\rm e}^{\kappa_\beta  \sum_{\calC\not\sim (\eta_e, \eta_f)}
 \abs{\Phi_\beta (\calC )}}
 \rb \leq {\rm e}^{- \ell_N \lb \taub (\sfv - \sfu ) + \taub (\sfw -\sfz )\rb
 (1-\smo{N} (1 ))} ,
  \ee
uniformly in $\beta$ large and uniformly in
all pairs of $\ell_N$-edges $e = (\sfu , \sfv )$ and
$f = (\sfv ,\sfz )$ such that
$\sfD_\theta (e )\cap \sfD_\theta (f) \neq \emptyset$.
 \end{proposition}

The estimate  \eqref{eq:two-edges-bound}  is a
manifestation of the fact that under \eqref{eq:chi-cond} entropic repulsion
between paths $\eta_e$ and  $\eta_f$  beats
 the attractive potential $
 {\rm e}^{\kappa_\beta  \sum_{\calC\not\sim (\eta_e, \eta_f)}
 \abs{\Phi_\beta (\calC )}}$. 
 As a result, typical paths
stay far apart and their contributions to the surface tension just add up.

 In the concluding Subsection~\ref{sub:two-segments}
we shall prove \eqref{eq:two-edges-bound} in the most difficult case
when $e$ and $f$ stay close to the horizontal axis. This case is the most
difficult since it corresponds to the minimal strength of entropic
repulsion between $\eta_e$ and $\eta_f$.

\subsection{Effective random walk representation.}
\label{sub:two-segments}
Consider edges $f_1 = (\sfz , \sfw)$ and  $f_2= (\sfu , \sfv)$  with
$\sfu = (0,0) := \mathsf{0}$, $\sfv = (\ell_N , 0)$, $\sfz = (0,z)$ and
$\sfw = (\ell_N , z )$. Define the event (collection of paths
$(\gamma_1 , \gamma_2 )$)
\be{eq:Tplus-two-z}
 \calT^{2}_+ = \calT^{2}_+ (\ell_N \, |\, z) =
 \lbr(\gamma_1 , \gamma_2 )~:~   \gamma_1\stackrel{\ell_N}{\sim} f_1\, ; \,
 \gamma_2\stackrel{\ell_N}{\sim} f_2\,  ;\,  \gamma_1\sim \gamma_2\rbr .
\ee
In particular, if $z\geq 0$,
 $\lbr \gamma_1 , \gamma_2\rbr \in \calT^{2}_+ \lb f_1, f_2 \rb$
implies that $\gamma_1$ ``stays above'' $\gamma_2$. Note, however that
they can share edges, and that they might have overhangs.

We claim that the following holds:
\begin{proposition}
\label{prop:tubebound}
Assume \eqref{eq:chi-cond}.
There exist a finite constant $c_+ $ such that for all $\beta$ sufficiently large,
\be{eq:tubebound}
\sup_{z\geq 0} \wbf\otimes\wbf\lb \calT^{2}_+ (\ell_N \, |\, z ) ; {\rm e}^{\, \kappa_\beta
\sum_\calC \1_{\calC\not\sim \eta_1}
\1_{\calC\not\sim \eta_2} \abs{\Phi_\beta (\calC )}}
\rb \leq
c_+
{\rm e}^{-2\tau_\beta (\sfe_1 )
{
\ell_N
}
}
\ee
as soon as  $N$ is  sufficiently large.
\end{proposition}
Proving that $c_+$ does not depend on $\beta$ is the crux
of the matter, and it is based on a careful analysis of  non-intersection
probabilities
for effective random walks in a weak attractive potential.

Let $\calT^1 = \calT^1 (\ell_N )$ be the set of paths
$\gamma : \mathsf{0}\mapsto \sfv$ with $\sfv\cdot \sfe_1 = \ell_N$. Note that
by definition
$\lb \gamma_1 , \gamma_2 \rb \in \calT^2_+$ implies that $\gamma_2 \in \calT^1$ and
$\gamma_1\in (0,z) + \calT^1$.

Let $\calK$ be a positive (two-dimensional)
symmetric cone around $\sfe_1$ with an opening strictly between $\pi/2$ and $\pi$.
A high temperature expansions of polymer weights ${\rm e}^{\Phi (\calC )}$
leads (see (4.9) in \cite{IST15}) to the following irreducible decomposition of
decorated (open) contours $\left[ \gamma , \underline{\calC}\right]$, where
$\gamma\in \calT^1 (\ell_N )$ and $\underline{\calC}$ is a
collection of $\gamma$-incompatible clusters:
\be{eq:irred-animals}
\left[ \gamma , \underline{\calC}\right] = \fra^{\ell}\circ \fra^1\circ \dots \fra^m
\circ\fra^r .
\ee
The irreducible animals $\fra^i = \left[ \eta^i ,\underline{\calD}^i\right]$
belong to the family $\sfA = \lbr \sfa = \left[ \eta , \calD\right]\rbr$, which could be
characterized by the following two properties:

\noindent
{\bf a.} If $\eta$ is a path with endpoints at $\sfx$, $\sfy$, then
\be{eq:diamond}
 \eta \cup  \underline{\calD} \subseteq \sfD (\sfx , \sfy ) := \lb \sfx +\calK\rb\cap
 \lb \sfy - \calK\rb .
\ee

\noindent
{\bf b.} $\fra$ could not be split into concatenation of two non-trivial animals
satisfying {\bf a} above.

The left and right irreducible animals satisfy one-sided versions of
diamond-confinement condition \eqref{eq:diamond}. For instance if $\fra^\ell =
\left[ \eta^\ell, \underline{\calD}^{\ell}\right]$ and $\sfy$ is the right end-points
of $\eta^\ell$, then $\eta^\ell \cup \underline{\calD}^{\ell} \subseteq
\lb \sfy - \calK\rb $.

Given an animal $\sfa = \left[ \eta , \underline{\calD}\right]$ we use
$\sfX (\fra ) = \sfX (\eta )$
to denote the vector which connects the left and right end-points of $\eta$.
The horizontal and vertical coordinates of $\sfX$ are denoted by
$\sfH = \sfX\cdot\sfe_1$ and $\sfV = \sfX\cdot \sfe_2$. By construction,
$\sfH (\fra ) \in \bbN$ for any irreducible animal $\sfa$.

Returning to the decomposition \eqref{eq:irred-animals}, let us consider for each $k\in \bbN$ the subset $\calT^1_{k}$
of decorated
paths  $\left[ \gamma , \underline{\calC}\right]$ from $ \calT^1$ for which $\sfH (\fra^\ell ) +
\sfH (\fra^r ) = k$. Then, by (4.11) in \cite{IST15}
there exists $c_g <\infty$ and $\nu_g >0$, such that
\be{eq:g-gap}
 \wbf\lb \calT^1_{k} \lb \ell_N \rb \rb \leq
 c_g {\rm e}^{-\beta \nu_g k} \wbf\lb \calT^1  \lb  \ell_N \rb \rb
\ee
for all $\beta$ and $N$ sufficiently large. So in the sequel we shall restrict
attention to decorated paths  $\left[ \gamma , \underline{\calC}\right] \in \calT^1_{0}$
with empty
right and left irreducible animals, that is with $\fra^\ell , \fra^r = \emptyset$
in \eqref{eq:irred-animals}. In particular, we shall restrict attention to
$\calT^{2}_{0, +} ( \ell_N |z ):=
\calT^{2}_+\cap \lb (\sfz + \calT^1_{0})\times \calT^1_{0}\rb$,
and, instead of \eqref{eq:tubebound}, shall derive an upper bound on
\be{eq:tubebound-not}
\sup_{z\geq 0} \wbf\otimes\wbf\lb \calT^{2}_{0,+}( \ell_N |z )  ;  {\rm e}^{\, \kappa_\beta
\sum_\calC \1_{\calC\not\sim \gamma_1}
\1_{\calC\not\sim \gamma_2} \abs{\Phi_\beta (\calC )}}
\rb
\ee
A general case could be  done by a straightforward adaptation based
on the mass-gap property \eqref{eq:g-gap}.

Decorated paths $\left[ \gamma , \underline{\calC}\right] \in \calT^1_{0}$ have
an immediate probabilistic interpretation: Set $\taub = \tau (\sfe_1 )$. Then
(see Theorem~5 in \cite{IST15})
\be{eq:P-beta}
\bbP_\beta (\fra ) := {\rm e}^{\taub \sfH (\fra )}\wbf (\fra )
\ee
is a probability distribution on $\sfA$ with exponentially decaying tail:
\be{eq:g-gapA}
\sum_{\fra\in \sfA} \1_{\sfX (\fra )=  (h , v) }
\bbP_\beta (\fra )  \leq c_g {\rm e}^{- \nu_g \beta \lb h+|v| -1 \rb } ,
\ee
-- i.e. no normalization in  \eqref{eq:P-beta} is needed!
Given $\sfx = (0, x )$ consider random walk $\sfS_n = \sfx +\sum_1^n \sfX^i$,
where $\sfX^u$-s are independent $\bbN\times \bbZ$-valued
steps distributed according to
\[
 \bbP_\beta (\sfX = \sfy ) = \sum_{\fra \in \sfA} \1_{\sfX (\fra )= \sfy }
 \bbP_{\beta} (\fra ) .
\]
Let $\bbP_{\beta,x}$ be the corresponding measure on random walk paths.
In this way $\wbf \lb \calT^1_{0}\rb$
equals to
\be{eq:Tnull-1bound}
\wbf \lb \calT^1_{0}\rb =
{\rm e}^{-\ell_N \taub}
\bbP_{\beta,
0} \lb \ell_N  \in {\rm Range}\lb \sfS_n\cdot \sfe_1\rb\rb  .
\ee
Let us adapt the above random walk representation  of a single decorated
path to the case of a pair of decorated
paths, from $\lb \sfz + \calT^1_{0}\rb\times \calT^1_{0}$.
These have irreducible decompositions
\be{eq:irred-pair}
\underline{\fra} = \sfz + \fra^1\circ \dots \circ \fra^n\quad {\rm and}\quad
\underline{\frb} = \frb^1 \circ \dots \circ\frb^m .
\ee
Following \cite{CIL} we shall align  horizontal projections of underlying
random walks. Given \eqref{eq:irred-pair} consider sets
\be{eq:calH-sets}
\calH_a = \lbr 0, \sfH (\fra_1 ) , \sfH(\fra_1 ) + \sfH(\fra_2 ), \dots , \ell_N\rbr,
\
\calH_b = \lbr 0, \sfH (\frb_1 ) , \sfH(\frb_1 ) + \sfH(\frb_2 ), \dots , \ell_N\rbr.
\ee
Set $\calH (\underline{\fra } , \underline{\frb }) = \calH_a \cap\calH_b$.
This intersection $\calH$ is the set of horizontal
projections of end-points of   jointly   irreducible
pairs of strings of  animals. The alphabet $\sfA^2$
of such pairs could be described as follows:
$\lb\underline{\fra } , \underline{\frb }\rb \in \sfA^2$ if
$\sum \sfH(\fra^i ) =
 \sum \sfH (\frb^j ) := \sfH \lb\underline{\fra } , \underline{\frb }\rb$, and
\be{eq:sfA2}
 \calH (\underline{\fra } , \underline{\frb })  = \lbr 0,
 \sfH \lb\underline{\fra } , \underline{\frb }\rb
 \rbr .
\ee
In the sequel we shall refer to  elements $\frc \in \sfA^2$ as
irreducible pairs.

By elementary renewal theory, \eqref{eq:P-beta} induces a probability
distribution on $\sfA^2$ which inherits exponential tails from \eqref{eq:g-gapA}.
We continue to call this distribution $\bbP_\beta$. The i.i.d.
$\bbN\times\bbZ\times\bbZ$-valued
steps of the
induced random walk have distribution
\[
 \bbP_\beta \lb \sfX = (\sfH, \sfV_1 , \sfV_2 ) = (h, v_1 , v_2 )\rb =
 \sum_{\frc \in \sfA^2} \1_{\sfX (\frc )= (h, v_1 , v_2 )} \bbP_\beta (\frc ).
\]
Decorated paths from $\lb \sfz + \calT^1_0\rb\times \calT^1_0$ give rise to random
walks
\[
 (0, z, 0) + \sum_1^n \sfX_i .
\]
Note that in this notation
\be{eq:RW-tube}
 \calT^{2}_{0,+} \lb \ell_N\, |\, z\rb  \subset
 \bigcup_n
 \lbr \sum_1^n \sfH_i = n ; \calR_+^n (z) \rbr ,
\ee
where
\be{eq:Z-walk}
\calR_+^n (z) = \lbr \sfZ_k
 \geq 0\ \text{for $k=0, 1, \dots , n$}\rbr\ \text{and}\
\sfZ_k = z+ \sum_{1}^{k} \lb \sfV_i^1 -\sfV_i^2 \rb .
\ee
With a slight abuse of notation we shall use the very same symbol
$\bbP_{\beta , z}$ also for  the law of the random walk $\sfZ_k$ in
\eqref{eq:Z-walk}.

\subsection{Recursion and random walk analysis}
Define
\be{eq:rho-z}
{
\mu_\beta
}
(\ell_N  |z ) =
{
\bbP_{\beta , z}
}
\lb
\calT^{2}_{0,+}
(\ell_N ~\big| ~z ) ;
{\rm e}^{\, \kappa_\beta \sum_\calC \1_{\calC\not\sim \gamma_1}
\1_{\calC\not\sim \gamma_2} \abs{\Phi_\beta (\calC )}}
\rb .
\ee
Assume \eqref{eq:chi-cond}. Then,
following Subsection~6.1 of \cite{IST15},  one can develop the following recursion
relation for $\mu_\beta = \sup_N\sup_z \mu_\beta (\ell_N  |z )$: There exist
$\beta$-independent constants $c_1$ and $\nu$ such that
\be{eq:rho-beta}
\mu_\beta \leq 1 + \mu_\beta c_1 {\rm e}^{-2 \chi\beta } \max_z
\sum_n
{
\bbE
}_{\beta , z}\lb {\rm e}^{-\nu \beta  \sfZ_n} ;  \calR_+^n \rb.
\ee
The {importance} of \eqref{eq:rho-beta} is that
\be{eq:less-0ne}
c_1 {\rm e}^{-2 \chi\beta } \sup_{z\geq 0}
\sum_n \bbE_{\beta , z}\lb {\rm e}^{-\nu \beta  \sfZ_n} ; \calR_+^n \rb < 1
\ee
implies that $\mu_\beta$ is bounded, and,
{
since by \eqref{eq:P-beta} and \eqref{eq:Tnull-1bound},
\be{eq:wbeta-Pbeta}
\begin{split}
&\wbf\otimes\wbf
\lb
\calT^{2}_{0,+}
(\ell_N ~\big| ~z ) ;
{\rm e}^{\, \kappa_\beta \sum_\calC \1_{\calC\not\sim \gamma_1}
\1_{\calC\not\sim \gamma_2} \abs{\Phi_\beta (\calC )}}
\rb \\
&\quad =
{\rm e}^{-2\ell_N \tau_\beta (\sfe_1 )}
\bbP_{\beta , z}
\lb
\calT^{2}_{0,+}
(\ell_N ~\big| ~z ) ;
{\rm e}^{\, \kappa_\beta \sum_\calC \1_{\calC\not\sim \gamma_1}
\1_{\calC\not\sim \gamma_2} \abs{\Phi_\beta (\calC )}}
\rb  \leq {\rm e}^{-2\ell_N \tau_\beta (\sfe_1 )} \mu_\beta ,
\end{split}
\ee
}
one deduces  \eqref{eq:tubebound}
as an immediate consequence. It remains to prove:
\begin{lemma}
 \label{lem:sum}
 If $\chi > 1/2$, then
 \be{eq:target}
 \lim_{\beta\to\infty}
{\rm e}^{-2 \chi\beta } \sup_{z\geq 0}
\sum_n \bbE_{\beta , z}\lb {\rm e}^{-\nu \beta  \sfZ_n} ;\calR_+^n \rb = 0.
 \ee
 In particular, the inequality \eqref{eq:less-0ne} holds for all $\beta$ sufficiently
 large.
\end{lemma}
\begin{proof}[Proof of Lemma~\ref{lem:sum} ]
Let us fix $z\geq 0$.
Consider decomposition of paths from
$\calR_+^n (z)$ with respect to the left-most minimum $u$, $0\leq u \leq z$.
Define the strict version $\hat\calR_+^n$ as
\[
 \hat\calR_+^n = \lbr \sfZ_k
 >  0\ \text{for $k=0, 1, \dots , n$}\rbr .
\]
Since $\sfW = \sfV^1 - \sfV^2$ has symmetric distribution under $\bbP_\beta$,
we can rewrite
\be{eq:z-decomp}
\sum_n \bbE_{\beta , z}\lb {\rm e}^{-\nu \beta  \sfZ_n} ; \calR_+^n \rb
=
\sum_{u=0}^{z}{\rm e}^{-\nu \beta u}
\sum_n \bbP_{\beta, 0} \lb \hat\calR_+^n ; \sfZ_n = z-u\rb
\sum_m \bbE_{\beta ,0}
\lb {\rm e}^{-\nu \beta  \sfZ_m}; \calR_+^m \rb .
\ee
Below we shall use a shorter notation $\bbP_\beta$ for $\bbP_{\beta , 0}$.

Following Subsection~7.1 in \cite{IST15} let us describe in more detail the
distribution of steps $\sfW$ under $\bbP_\beta$. The fact that here we take
the cone $\calK$ to be symmetric with respect to the $\sfe_1$-axis simplifies
the exposition. In particular, $\sfW$ has a symmetric distribution. The analysis
of \cite{IST15} could be summarized as follows:
\be{eq:stepsW}
\begin{split}
1-p_\beta &:= \bbP_\beta \lb \sfW = 0\rb = 1- \bigof{{\rm e}^{-\beta}}, \
\bbP_\beta\lb \sfW = \pm 1\rb = \bigof{{\rm e}^{-\beta}}\\
&
\text{and, for $z\neq 0,\pm 1$,}\
\bbP_\beta \lb\sfW = z\rb \leq p_\beta {\rm e}^{-\nu\beta \abs{z}}
\end{split}
\ee
There is a natural Wald-type decomposition of random walk $\sfZ_k$ with i.i.d. steps
distributed according to \eqref{eq:stepsW}: Let $\xi_1, \xi_2, \dots$ be i.i.d.
Bernoulli random variables with probability of success $p_\beta$, and let $\sfU_k$
be another independent i.i.d. sequence with
\be{eq:U-dist}
\bbP_\beta \lb \sfU = z\rb =
p_\beta^{-1}\bbP_\beta \lb \sfW = z\rb\ \text{for $z\neq 0$}.
\ee
Set $\sfM_n = \sum_1^n\xi_i$.
Then $\sfZ_n$ could be represented as
\be{eq:Zn-decomp}
\sfZ_n = \sum_{k=1}^{\sfM_n} \sfU_k := \sfY_{\sfM_n},
\ee
where $\sfY_k$ is a random walk with i.i.d. steps $\sfU_1, \sfU_2, \dots$.
With a slight abuse of notation we continue to use $\calR_+^n$ for the corresponding
event for $\sfY_k$-walk.

Let $y \in \bbN$ and consider $\sum_n \bbP_\beta \lb \hat\calR_+^n ; \sfZ_n = y\rb$.
Let $\hat \calL_y$ be the event that $y$ is a strict ladder height of $\sfZ$, or equivalently,
of $\sfY$. Let $\hat N (y)$ be the number of strict ladder heights $v\leq y$. Then
(see Subsection~7.3 of \cite{IST15} for more detail),
\be{eq:ladder-expr-strict}
\bbP_\beta \lb \hat\calR_+^n ; \sfZ_n = y\rb =
\bbP_\beta \lb \hat\calL_y  ; \sfZ_n = y\rb = \frac{1}{n}\bbE_\beta\lb \hat N (y) ; \sfZ_n = y\rb
\leq \frac{y}{n} \bbP\lb  \sfZ_n = y\rb .
\ee
Hence,
\be{eq:ladder-strict-resum}
\sum_n \bbP_\beta \lb \hat\calR_+^n ; \sfZ_n = y\rb
\leq
y\sum_{\ell=1}^\infty \bbP_\beta (\sfY_\ell = y )\sum_{n=1}^\infty \frac{1}{n}
\bbP_\beta(\sfM_n = \ell )
\ee
Now, since $\sfM$ is a process with Bernoulli steps, there is a combinatorial
identity:
\[
 \sum_{n=1}^\infty \frac{1}{n}
\bbP_\beta(\sfM_n = \ell ) =
{
\frac{1}{\ell} \sum_{n=\ell}^\infty {n-1 \choose \ell-1 }
p_\beta^{\ell} \lb 1-p_\beta\rb^{n-\ell}
}
= \frac{1}{\ell}.
\]
We claim that there exists $c_1 <\infty$ such that
\be{eq:c1-Ysum}
\sum_{\ell=1}^\infty \frac{1}{\ell} \bbP_\beta (\sfY_\ell = y ) \leq \frac{c_1}{y}
\ee
holds for every $y\in\bbN$ and all $\beta$ sufficiently large.
Note that \eqref{eq:c1-Ysum}  would imply that
\be{eq:target-1st}
\sum_n \bbP_\beta \lb \hat\calR_+^n ; \sfZ_n = y\rb
\leq c_1
\ee
for all $\beta$ sufficiently large.

In its turn \eqref{eq:c1-Ysum} follows from routine estimates on characteristic
functions. Let $\phi_\beta$ be the characteristic function of $\sfU$ in
\eqref{eq:U-dist}. By direct computation,
\be{eq:sum-l-expr}
\sum_{\ell=1}^\infty \frac{1}{\ell} \bbP_\beta (\sfY_\ell = y ) =
\frac{1}{y} \cdot \frac{1}{2\pi}\int_0^{2\pi} F_\beta (\theta )
\lb \sum_1^{y-1}
{\rm e^{-i k\theta} }\rb \dd\theta = \frac{1}{y} \sum_1^{y-1}\hat F_\beta (k),
\ee
where,
\be{eq:F-beta}
F_\beta (\theta ) = F_\beta \lb {\rm e}^{i\theta }\rb  =
\frac{1- {\rm e}^{i\theta}}{1- \phi_\beta (\theta )}\phi^\prime_\beta  (\theta ) ,
\ee
and $\hat F_\beta$ is its Fourier transform. In view of the exponential decay of
probabilities in \eqref{eq:stepsW}, there exists $r >0$, such that $F_\beta$ have
uniformly bounded (in large $\beta$) analytic extension to the complex annulus
$\lbr \sfz\in \bbC~: ~ \abs{\sfz}\in (1, 1+r )\rbr$. Hence, $\hat F_\beta (k )$
tend to zero uniformly exponentially fast, and in particular $\sum_1^\infty
\abs{\hat F_\beta (k )}$ is uniformly bounded in $\beta$ large. \eqref{eq:c1-Ysum}
follows.
\smallskip

Let us turn to the second term
\be{eq:sum-nonstrict}
\sum_m \bbE_\beta
\lb {\rm e}^{-\nu \beta  \sfZ_m}; \calR_+^m \rb =
\sum_{y\geq 0} {\rm e}^{-\nu\beta y}
\sum_{m}\bbP_\beta\lb  \calR_+^m ;  \sfZ_m = y  \rb
\ee
in \eqref{eq:z-decomp}. We claim that there exists a finite constant $c_2$ such that
\be{eq:target-2}
\sum_{m}\bbP_\beta\lb  \calR_+^m ;  \sfZ_m = y  \rb \leq \frac{c_2 (y+1)}{p_\beta} .
\ee
for all $y\geq 0$ and $\beta$ large.

Note that \eqref{eq:target-2}  would imply  that
\[
 \sum_m \bbE_\beta
\lb {\rm e}^{-\nu \beta  \sfZ_m}; \calR_+^m \rb \leq \frac{c_3}{p_\beta} .
\]
By \eqref{eq:z-decomp} and \eqref{eq:target-1st} this would mean that
\[
 \sup_{z\geq 0} \sum_n \bbE_{\beta , z}\lb {\rm e}^{-\nu \beta  \sfZ_n} ; \calR_+^n \rb
 \leq \frac{c_1 c_3}{\lb 1 - {\rm e}^{-\nu\beta}\rb p_\beta }
\]
Since by \eqref{eq:stepsW} the order of $p_\beta = \bigof{{\rm e}^{-\beta}}$, the limit
 $\lim_{\beta\to\infty} \frac{{\rm e}^{-2\chi\beta}}{p_\beta} = 0$ whenever
$\chi >1/2$. Hence \eqref{eq:target}.

It, therefore, remains to verify \eqref{eq:target-2}. Let $N_y$ be the number of non-strict
ladder times for heights between $0$ and $y$. Let $M_y$ be the same variable for the
random walk $\sfY$. Clearly,
\[
 N_y = \sum_{i=1}^{M_y}  \zeta_i ,
\]
where $\zeta_i$-s are i.i.d. $\text{Geo$(p_\beta )$}$-random variables.
Proceeding as in the proof of
\eqref{eq:target-1st}, and in particular relying on Subsection~7.1 in \cite{IST15},
we estimate:
\be{eq:YW-2d-bound}
\begin{split}
\sum_{m}\bbP_\beta\lb  \calR_+^m ;  \sfZ_m = y  \rb &\leq
\frac{1}{p_\beta}\sum_n \frac{1}{n} \bbE_\beta\lb M_y ; \sfY_n = y\rb \\
&\leq \frac{\lb \bbE_\beta M_y^k\rb^{1/k}}{p_\beta}
\sum_n \frac{1}{n} \lb \bbP_\beta \lb \sfY_n = y\rb \rb^{(k-1)/k} .
\end{split}
\ee
Applying Lemma~22 in \cite{IST15} it is straightforward  to see that
under \eqref{eq:stepsW} for any $k\in \bbN$ there is a finite constant $c_k$
such that
\[
 \lb \bbE_\beta M_y^k\rb^{1/k} \leq c_k (y+1 )
\]
for all $y\geq 0$ and $\beta$ sufficiently large. On the other hand, again under
\eqref{eq:stepsW}, it is straightforward to check that
\[
 \max_y \bbP_\beta \lb \sfY_n = y\rb \leqs n^{-1/2}.
\]
uniformly in $n$ and  in $\beta$ large. Hence \eqref{eq:target-2}, and we are done.
\end{proof}


\begin{thebibliography}{10}

\bibitem{BCK03}
Marek Biskup, Lincoln Chayes, and Roman Koteck{\'y}.
\newblock Critical region for droplet formation in the two-dimensional {I}sing
  model.
\newblock {\em Comm. Math. Phys.}, 242(1-2):137--183, 2003.

\bibitem{B99}
Thierry Bodineau.
\newblock The {W}ulff construction in three and more dimensions.
\newblock {\em Comm. Math. Phys.}, 207(1):197--229, 1999.

\bibitem{BIV00}
Thierry Bodineau, Dmitry Ioffe, and Yvan Velenik.
\newblock Rigorous probabilistic analysis of equilibrium crystal shapes.
\newblock {\em J. Math. Phys.}, 41(3):1033--1098, 2000.
\newblock Probabilistic techniques in equilibrium and nonequilibrium
  statistical physics.

\bibitem{BSS05}
Thierry Bodineau, Roberto~H. Schonmann, and Senya Shlosman.
\newblock 3{D} crystal: how flat its flat facets are?
\newblock {\em Comm. Math. Phys.}, 255(3):747--766, 2005.

\bibitem{B03}
Hans~Paul Bonzel.
\newblock 3{D} equilibrium crystal shapes in the new light of {STM} and {AFM}.
\newblock {\em Physics Reports}, 385(1):1--67, 2003.

\bibitem{BYS}
H.P. Bonzel, D.K. Yu, and M.~M.~Scheffler.
\newblock The three-dimensional equilibrium crystal shape of {P}b: Recent
  results of theory and experiment.
\newblock {\em Appl. Phys. A}, 87:391--397.

\bibitem{BMF86}
J.~Bricmont, A.~El~Mellouki, and J.~Fr\"ohlich.
\newblock Random surfaces in statistical mechanics: roughening, rounding,
  wetting,...
\newblock {\em J. Statist. Phys.}, 42(5-6):743--798, 1986.

\bibitem{BFL82}
Jean Bricmont, Jean-Raymond Fontaine, and Joel~L. Lebowitz.
\newblock Surface tension, percolation, and roughening.
\newblock {\em J. Statist. Phys.}, 29(2):193--203, 1982.

\bibitem{CIL}
Massimo Campanino, Dmitry Ioffe, and Oren Louidor.
\newblock Finite connections for supercritical {B}ernoulli bond percolation in
  2{D}.
\newblock {\em Markov Process. Related Fields}, 16(2):225--266, 2010.

\bibitem{CLMST14}
Pietro Caputo, Eyal Lubetzky, Fabio Martinelli, Allan Sly, and Fabio~Lucio
  Toninelli.
\newblock Dynamics of {$(2+1)$}-dimensional {SOS} surfaces above a wall: slow
  mixing induced by entropic repulsion.
\newblock {\em Ann. Probab.}, 42(4):1516--1589, 2014.

\bibitem{CLMST13}
Pietro Caputo, Eyal Lubetzky, Fabio Martinelli, Allan Sly, and Fabio~Lucio
  Toninelli.
\newblock Scaling limit and cube-root fluctuations in {SOS} surfaces above a
  wall.
\newblock {\em J. Eur. Math. Soc. (JEMS)}, 18(5):931--995, 2016.

\bibitem{CMT14}
Pietro Caputo, Fabio Martinelli, and Fabio~Lucio Toninelli.
\newblock On the probability of staying above a wall for the
  {$(2+1)$}-dimensional {SOS} model at low temperature.
\newblock {\em Probab. Theory Related Fields}, 163, 2015.

\bibitem{Cerf}
Rapha\"el Cerf.
\newblock Large deviations of the finite cluster shape for two-dimensional
  percolation in the {H}ausdorff and {$L^1$} metric.
\newblock {\em J. Theoret. Probab.}, 13(2):491--517, 2000.

\bibitem{CerfK01}
Rapha{\"e}l Cerf and Richard Kenyon.
\newblock The low-temperature expansion of the {W}ulff crystal in the 3{D}
  {I}sing model.
\newblock {\em Comm. Math. Phys.}, 222(1):147--179, 2001.

\bibitem{CP00}
Rapha\"el Cerf and \'Agoston Pisztora.
\newblock On the {W}ulff crystal in the {I}sing model.
\newblock {\em Ann. Probab.}, 28(3):947--1017, 2000.

\bibitem{CKP01}
Henry Cohn, Richard Kenyon, and James Propp.
\newblock A variational principle for domino tilings.
\newblock {\em J. Amer. Math. Soc.}, 14(2):297--346, 2001.

\bibitem{Dobrushin}
R.~L. Dobrushin.
\newblock Gibbs states describing a coexistence of phases for the
  three-dimensional ising model.
\newblock {\em Th.Prob. and its Appl.}, 17(3):582--600, 1972.

\bibitem{DKS}
Roland~L. Dobrushin, Roman Koteck{\'y}, and Senya Shlosman.
\newblock {\em Wulff construction}, volume 104 of {\em Translations of
  Mathematical Monographs}.
\newblock American Mathematical Society, Providence, RI, 1992.
\newblock A global shape from local interaction, Translated from the Russian by
  the authors.

\bibitem{DS2}
Roland~L. Dobrushin and Senya~B. Shlosman.
\newblock Droplet condensation in the {I}sing model: moderate deviations point
  of view.
\newblock In {\em Probability and phase transition ({C}ambridge, 1993)}, volume
  420 of {\em NATO Adv. Sci. Inst. Ser. C Math. Phys. Sci.}, pages 17--34.
  Kluwer Acad. Publ., Dordrecht, 1994.

\bibitem{DS1}
Roland~L. Dobrushin and Senya~B. Shlosman.
\newblock Large and moderate deviations in the {I}sing model.
\newblock In {\em Probability contributions to statistical mechanics},
  volume~20 of {\em Adv. Soviet Math.}, pages 91--219. Amer. Math. Soc.,
  Providence, RI, 1994.

\bibitem{EBWTR-RW}
A.~Emundts, H.~P. Bonzel, P.~Wynblatt, K.~Thürmer, J.~Reutt-Robey, and E.~D.
  Williams.

\bibitem{FS03}
Patrik~L. Ferrari and Herbert Spohn.
\newblock Step fluctuations for a faceted crystal.
\newblock {\em J. Statist. Phys.}, 113(1-2):1--46, 2003.

\bibitem{FS05}
Patrik~L. Ferrari and Herbert Spohn.
\newblock Constrained {B}rownian motion: fluctuations away from circular and
  parabolic barriers.
\newblock {\em Ann. Probab.}, 33(4):1302--1325, 2005.

\bibitem{IS08}
Dmitry Ioffe and Senya Shlosman.
\newblock Ising model fog drip: the first two droplets.
\newblock In {\em In and out of equilibrium. 2}, volume~60 of {\em Progr.
  Probab.}, pages 365--381. Birkh\"auser, Basel, 2008.

\bibitem{IST15}
Dmitry Ioffe, Senya Shlosman, and Fabio~Lucio Toninelli.
\newblock Interaction versus entropic repulsion for low temperature {I}sing
  polymers.
\newblock {\em J. Stat. Phys.}, 158(5):1007--1050, 2015.

\bibitem{ISV15}
Dmitry Ioffe, Senya Shlosman, and Yvan Velenik.
\newblock An invariance principle to {F}errari-{S}pohn diffusions.
\newblock {\em Comm. Math. Phys.}, 336(2):905--932, 2015.

\bibitem{IV08}
Dmitry Ioffe and Yvan Velenik.
\newblock Ballistic phase of self-interacting random walks.
\newblock In M.~Penrose H.~Schwetlick P.~M\"{o}rters, R.~Moser and J.~Zimmer,
  editors, {\em Analysis and Stochastics of Growth Processes and Interface
  Models}, pages 55--79. Oxford University Press, 2008.

\bibitem{IV16}
Dmitry Ioffe and Yvan Velenik.
\newblock Low temperature interfaces: Prewetting, layering, faceting and
  {F}errari-{S}pohn diffusions.
\newblock {\em Preprint, arXiv:1611.00658}, 2016.

\bibitem{IVW17}
Dmitry Ioffe, Yvan Velenik, and Vitaly Wachtel.
\newblock {D}yson {F}errari-{S}pohn diffusions and ordered walks under area
  tilts.
\newblock {\em Probab. Theory Relat. Fields}, 2017.

\bibitem{Ken08}
Richard Kenyon.
\newblock Height fluctuations in the honeycomb dimer model.
\newblock {\em Comm. Math. Phys.}, 281(3):675--709, 2008.

\bibitem{MS99}
Salvador Miracle-Sole.
\newblock Facet shapes in a {W}ulff crystal.
\newblock In {\em Mathematical results in statistical mechanics ({M}arseilles,
  1998)}, pages 83--101. World Sci. Publ., River Edge, NJ, 1999.

\bibitem{Ok16}
Andrei Okounkov.
\newblock Limit shapes, real and imagined.
\newblock {\em Bull. Amer. Math. Soc. (N.S.)}, 53(2):187--216, 2016.

\bibitem{SS}
Roberto~H. Schonmann and Senya~B. Shlosman.
\newblock Constrained variational problem with applications to the {I}sing
  model.
\newblock {\em J. Statist. Phys.}, 83(5-6):867--905, 1996.

\end{thebibliography}
\end{document}